\documentclass[a4paper,twoside,11pt]{report}

\usepackage{fancyhdr}

\usepackage{graphicx,ifpdf}
\ifpdf \usepackage{epstopdf}
\else
\fi

\usepackage[a4paper,inner=32mm,outer=32mm]{geometry}

\usepackage{setspace}

\usepackage[colorlinks,linkcolor=black,citecolor=black]{hyperref}

\usepackage[colorlinks,linkcolor=black,citecolor=black]{hyperref}
\usepackage{yfonts}
\newcommand{\A}{\mathcal{A}_{reg}}
\newcommand{\M}{\mathcal{M}_{reg}}
\usepackage {graphicx}
\usepackage{amsfonts}
\usepackage{amsthm}
\usepackage{amsmath}
\usepackage{amsfonts}
\usepackage{latexsym}
\usepackage{amssymb}
\usepackage{epsfig,color}
\usepackage{graphicx, amssymb}

\usepackage[ps,all,arc,rotate]{xy}
\usepackage{hyperref}
\usepackage{mathtools}

\usepackage{setspace}
\usepackage[charter]{mathdesign}

\newtheorem{definition}{Definition}[chapter]

\newtheorem{example}[definition]{Example} 
\newtheorem{theorem}[definition]{Theorem} 
\newtheorem{remark}[definition]{Remark} 
\newtheorem{proposition}[definition]{Proposition} 
\newtheorem{conjecture}[definition]{Conjecture}

\newtheorem{teo}[definition]{Theorem} 
\newtheorem{rem}[definition]{Remark} 
\newtheorem{prop}[definition]{Proposition}

\newcommand{\Na}{\begin{eqnarray}}
\newcommand{\ENa}{\end{eqnarray}}

\newcommand{\PriM}{{\rm Prym}(S,\Sigma)}

\usepackage[ps,all,arc,rotate]{xy}
\usepackage {graphicx}
\usepackage{amsfonts}
\usepackage{amsthm}
\usepackage{amsmath}
\usepackage{amsfonts}
\usepackage{latexsym}
\usepackage{amssymb}
\usepackage{epsfig,color}
\usepackage{graphicx, amssymb}
\usepackage{hyperref}
\usepackage{mathtools}

\begin{document}


\begin{titlepage}
\begin{singlespace}
\pagenumbering{roman} 
\thispagestyle{empty}
\begin{center}
\vspace*{8mm}
\Huge {\bfseries {\sc Spectral data for $G$-Higgs bundles}}

\vspace*{45mm}

\Large Laura P. Schaposnik M.\\

\large \vspace*{1ex} New College\\
    
\large \vspace*{1ex} University of Oxford

\large \vspace*{20mm} \includegraphics[width=40mm]{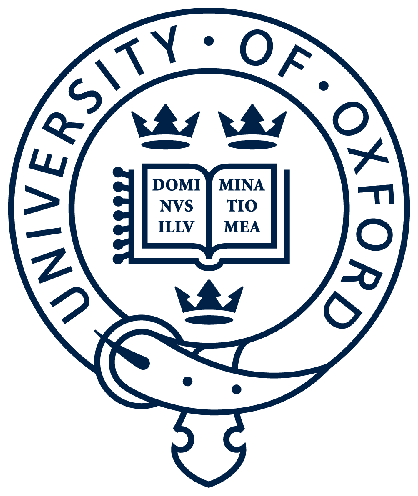}

 

\vspace*{20mm} A thesis submitted for the degree of

\vspace*{1ex}{\it {Doctor of Philosophy}}

\vspace*{2ex} August 2012
\end{center}
\end{singlespace}
\end{titlepage}

\mbox{}
\thispagestyle{empty}
\clearpage

\begin{center}
\thispagestyle{empty}  
\vspace*{20mm} 
{ \Large \textit{ In memory of my mother,} \\

\textit{ Patricia Massolo.}  }
\end{center}
\clearpage

\mbox{}
\thispagestyle{empty}
\clearpage

\begin{singlespace}
\begin{center}
\thispagestyle{empty}
{\Large{\bfseries Abstract}}

\vspace{2ex}
{\Large Spectral data for $G$-Higgs bundles}

\vspace{1.5ex}
{\large Laura P. Schaposnik M.}

\vspace{0.5ex}
New College

\vspace{1.5ex}
A thesis submitted for the degree of \emph{Doctor of Philosophy}, Trinity Term 2012

\vspace{2ex}

\end{center}
\end{singlespace}

This Thesis is dedicated to the study of principal $G$-Higgs bundles on a Riemann surface, where $G$ is a real form of a complex Lie group. The first three chapters of the thesis review the theory of classical Higgs bundles and principal Higgs bundles. In Chapters \ref{ch:split}-\ref{ch:sppp} we develop a new method of understanding $G$-Higgs bundles through their spectral data, for $G$ a split real form, $G=SL(2,\mathbb{R}),$ $U(p,p)$, $SU(p,p)$ and $Sp(2p,2p)$. Finally in Chapters \ref{ch:applications}-\ref{ch:further} we give some open questions and applications of our results.

In particular, for $G$ a split real form, we identify $G$-Higgs bundles with points of order two in the regular fibres of the $G^{c}$ Hitchin fibration in Chapter \ref{ch:split}, where $G^{c}$ is the complexification of $G$. Through this approach, we study the special case of $SL(2,\mathbb{R})$-Higgs bundles, for which we give an explicit description of the monodromy action in Chapter \ref{ch:monodromy}. 

In the case of $U(p,p)$-Higgs bundles, in Chapter \ref{ch:supp} we identify the  known topological invariants by the action of a natural involution at fixed points. Subsequently, we define the spectral data in terms of the Jacobian variety of an intermediate curve, and for $SU(p,p)$-Higgs bundles, in terms of the Prym variety of a quotient curve.

In Chapter \ref{ch:sppp} we study $Sp(2p,2p)$-Higgs bundles and define their spectral data through parabolic vector bundles in an intermediate cover of the Riemann surface. Through this approach, we discover a rank 2 vector bundle moduli space in the fibres of the $Sp(4p,\mathbb{C})$ Hitchin fibration.

Lastly, in Chapter \ref{ch:further} we discuss further applications of our methods, among which are implications concerning connectivity of moduli spaces of $G$-Higgs bundles, as well as non-vanishing of higher cohomology groups of these moduli spaces.

\clearpage

\mbox{}
\thispagestyle{empty}
\clearpage


\begin{center}
\thispagestyle{empty}
{\Large{\bfseries  Acknowledgements}}
\vspace{2ex}
\end{center}
First and foremost, I would like to thank my supervisor Nigel Hitchin. Working with him has been a pleasure and during the last four years, our weekly meetings have been the greatest source of inspiration and encouragement. I shall always be grateful for his enormous support and patience, and his ability to make me excited about mathematics.

I am very thankful for the financial support I received from the Oxford University Press and New College, through the joint Clarendon Award and Graduate Scholarship, and for 
the grant received from the Centre for the Quantum Geometry of Moduli Spaces (QGM) in Aarhus, which helped me to finish this thesis. I am also thankful to IHP in Paris,  CRM in Barcelona, and QGM in Aarhus for their hospitality during my research visits, and to the Mathematics Departments in   Cordoba, Buenos Aires and La Plata for their support.

Many mathematicians both in my department, and during research trips, have influenced the way I understand maths, and I would like to thank all of them. In particular, I am thankful to Tamas Hausel, Frances Kirwan, Peter Newstead and Jorge Solomin  for inspiring and helpful conversations, and to  Alan Thompson, Jeff Giansiracusa, Bernardo Uribe, Roberto Rubio and Steve Rayan for patiently explaining and discussing maths with me. 

My time in Oxford would not have been the same without teaching mathematics. I am greatly thankful to my college advisor Victor Flynn, who helped me to build my teaching skills. The time I spent as a Lecturer at Hertford College and Magdalen College, and as a Tutor at the Centre for International Education at Oxford,  was most enjoyable. 

I am  grateful to my friends back at home and also to the ones I have made around the world. Their amity has always been very important.

Finally, I am deeply thankful to my family for their support and encouragement, which has kept me going throughout these years. I have always loved their physics at home.   
And I would like to thank James Unwin, who saw this work be built from the first day. His love, patience and confidence in my research has made me stronger each day.



\clearpage

\mbox{}
\thispagestyle{empty}
\clearpage

\section*{Statement of Originality}
This thesis contains no material that has already been accepted, or is concurrently being submitted, for any degree or diploma or certificate or other qualification in this University or elsewhere. To the best of my knowledge and belief this thesis contains no material previously published or written by another person, except where due reference is made in the text.
\bigskip \bigskip

\hfill Laura P. Schaposnik M.
\hfill August 28, 2012
\clearpage

\mbox{}
\thispagestyle{empty}
\clearpage

\setlength{\headheight}{15pt}
\pagestyle{fancy}
\fancyhead{}
\renewcommand{\sectionmark}[1]{\markright{\thesection.\ #1}}
\renewcommand{\chaptermark}[1]{\markboth{Chapter \thechapter.\ #1}{}}

\fancyhead[CO,CE]{\fancyplain{}{\textit{Contents}}}

\tableofcontents
\cleardoublepage

\fancyhead[CO]{\fancyplain{}{\textit{\rightmark}}}
\fancyhead[CE]{\fancyplain{}{\textit{\leftmark}}}

\pagenumbering{arabic} 


\chapter{Overview and statement of results}\label{ch:intro}

 Since Higgs bundles were introduced in 1987 \cite{N1}, they have found applications in many areas of mathematics and mathematical physics. In particular, Hitchin showed in  \cite{N1} that their moduli spaces give examples of Hyper-K\"ahler manifolds and that they provide an interesting example of integrable systems \cite{N2}. More recently,  Hausel and Thaddeus \cite{Tamas1}  related Higgs bundles to mirror symmetry, and in the work of Kapustin and Witten \cite{Kap} Higgs bundles were used to give a physical derivation of
the geometric Langlands correspondence.

The moduli space  $\mathcal{M}_{G}$  of polystable $G$-Higgs bundles over a compact Riemann surface $\Sigma$, for $G$ a real form of a complex semisimple Lie group $G^{c}$,  may be identified through non-abelian Hodge theory with the moduli space of representations of the fundamental group of $\Sigma$ (or certain central extension of the fundamental
group) into $G$ (see \cite{GP09} for the Hitchin-Kobayashi correspondence for $G$-Higgs bundles). Motivated partially by this identification, the moduli space of $G$-Higgs bundles has been studied by various researchers, mainly through a Morse theoretic approach (see, for example, \cite{brad3} for an expository article on applications of Morse theory to moduli spaces of Higgs bundles).

Real forms of $SL(n,\mathbb{C})$ and $GL(n,\mathbb{C})$ were initially considered in \cite{N1} and \cite{N5}, where Hitchin studied $SL(2,\mathbb{R})$-Higgs bundles, and later on extended those results to $SL(n,\mathbb{R})$-Higgs bundles. Connectivity of the moduli space of $G$-Higgs bundles for the real forms $G=PGL(n,\mathbb{R})$, $PGL(2,\mathbb{R})$, $SU(p,q)$ and $PU(p,p)$ was studied by Bradlow, Gothen, Garcia Prada and Oliveira ( e.g., \cite{brad}, \cite{brad0}, \cite{ol1}), as well as by Xia and Markman \cite{Xia,Xia1,Xia3,Xia2},  their main results holding under certain constraints on the degrees of the vector bundles involved. Finally, Garcia Prada and Oliveira calculated the connected components for $U^{*}(2n)$-Higgs bundles in \cite{GP10}.

In the case of real forms of $Sp(2n,\mathbb{C})$, at the time of writing this thesis, only the group $Sp(2n,\mathbb{R})$ has been considered. Gothen studied real symplectic  rank 4 Higgs bundles in \cite{Go2}, and together with Bradlow and Garcia-Prada studied connectivity for $Sp(2n,\mathbb{R})$-Higgs bundles for arbitrary $n$ (e.g. \cite{brad2}). The study of connected components for the particular case of $Sp(4,\mathbb{R})$ has received extensive attention from Gothen, Garcia Prada, Mundet and Oliveira (among others, see \cite{brad}, \cite{GP08}).  

Finally, $G$-Higgs bundles  for $G$ a real form of $SO(2n+1,\mathbb{C})$ or $SO(2n,\mathbb{C})$ have been studied by Aparicio and Garcia Prada in \cite{ap} and \cite{ap1}, where connectivity results are given for $SO_{0}(p,q)$-Higgs bundles for $p=1$ and $q$ odd, and by Bradlow, Gothen and Garcia Prada in \cite{brad1}, where connected components for $SO^{*}(2n)$-Higgs bundle are given for maximal topological invariant. \\

This thesis is dedicated to the study of principal $G$-Higgs bundles and their moduli spaces, for $G$ a real form of a complex Lie group $G^{c}$. After introducing the background material needed in Chapter \ref{ch:complex} and Chapter \ref{ch:real}, we concentrate our attention on the case of $G$ being a split real form in general, and on the particular cases of $G=SL(2,\mathbb{R})$, $U(p,p)$, $SU(p,p)$ and $Sp(2p,2p)$. The original results and new methods developed in this thesis appear in Chapters \ref{ch:split}-\ref{ch:applications}, whilst further applications are given in Chapter \ref{ch:further}. We have organised the work in the following way. \\ \\

We begin Chapter \ref{ch:complex} by introducing classical Higgs bundles and principal $G^{c}$-Higgs bundles for a complex Lie group $G^{c}$, following \cite{N1} and \cite{N2}. Then, via the Hitchin fibration, the spectral data approach considered in \cite{N3} is presented, and a detailed description of the method is given for the classical Lie groups $G^{c}=GL(n,\mathbb{C})$, $SL(n,\mathbb{C})$, $Sp(2p,\mathbb{C})$, $SO(2n,\mathbb{C})$ and $SO(2n+1,\mathbb{C})$. \\

In Chapter \ref{ch:real} we recall the main properties of principal $G$-Higgs bundles for a real form $G$ of a complex Lie group $G^{c}$, describing both the Lie theoretic construction of Higgs bundles as well as their appearance as fixed points of a certain involution on the moduli space of $G^{c}$-Higgs bundles.  In preparation for the study of the spectral data of $G$-Higgs bundles, and since we know of no complete list, in this chapter we give a thorough description of $G$-Higgs bundles when the structure group is a non-compact real form of a classical complex Lie group. In particular, in each case we describe the involution defining the corresponding $G$-Higgs bundle, and study the associated ring of invariant polynomials. \\

In Chapter \ref{ch:split} we study $G$-Higgs bundles for $G$ a split real form. In this case, the Hitchin map 
 is surjective and so the generic fibre is a torus. Moreover, the Teichm\"uller component defined in \cite{N5} provides an origin for the bundles of the same topological type, and makes the fibre an abelian variety. It is via the Teichm\"uller component that we are able to describe the moduli space of $G$-Higgs bundles as fixed points in the moduli space of $G^{c}$-Higgs bundles:  

\noindent \textbf{Theorem \ref{teo:split}}. \textit{The intersection of the subspace of the Higgs bundle moduli space $\mathcal{M}_{G^{c}}$ corresponding to the split real form of $\mathfrak{g}^{c}$  with the smooth fibres of the Hitchin fibration
\[h:~\mathcal{M}_{G^{c}}\rightarrow \mathcal{A}_{G^{c}},\]
 is given by the elements of order two in those fibres.
}

As a consequence, this moduli space is a finite covering of an open set in the base. Hence, in the case of a split real form $G$, the natural way of understanding the topology of the moduli space of $G$-Higgs bundles is via the monodromy of the covering. Although generally this is known to be a very difficult problem,   the case of rank-2 Higgs bundles is considerably more tenable. \\

In Chapter \ref{ch:monodromy} we study $SL(2,\mathbb{R})$-Higgs bundles. Through Theorem \ref{teo:split}, the moduli space can be seen as points of order two in the classical Hitchin fibration. We use the results of Copeland \cite{cope1} to understand the generators of the monodromy action for $SL(2,\mathbb{R})$-Higgs bundles, and obtain an explicit description of the monodromy action: 

\noindent \textbf{Theorem \ref{teo}}. \textit{
 The monodromy action on the first mod 2 cohomology of the fibres of the Hitchin fibration is given by the group of matrices acting on $\mathbb{Z}_{2}^{6g-6}$ of the form
\begin{eqnarray}
\left(\begin{array} {c|c}
I_{2g}&A\\\hline
0&\pi
                \end{array}
\right),
\end{eqnarray}
\noindent  where  \begin{itemize}
\item $\pi$ is the quotient action on $\mathbb{Z}_{2}^{4g-5}/(1,\cdots,1)$ induced by the permutation action of the symmetric group $S_{4g-4}$ on $\mathbb{Z}_{2}^{4g-5}$;
\item  $A$  is any $(2g)\times (4g-6)$ matrix with entries in $\mathbb{Z}_{2}$.\\
                  \end{itemize}
}

The monodromy approach appeared to be very difficult to follow for higher rank Higgs bundles, and already for rank-3 Higgs bundles  calculations   using similar ideas to the rank-2 case did not lead to fruitful results.  In order to study higher rank Higgs bundles for non-split real forms, in Chapter \ref{ch:supp} we look at the particular case of $U(p,p)$-Higgs bundles. As in the case of complex Lie groups, we consider the characteristic polynomial $p(x)$ of a $U(p,p)$-Higgs bundle and the curve it defines.
Extending the spectral data methods introduced in Chapter \ref{ch:complex}, we obtain a description of the spectral data associated to  $U(p,p)$-Higgs bundles: 
\noindent \textbf{Theorem \ref{teo11}}. \textit{
 There is a one to one correspondence between  $U(p,p)$-Higgs bundles $(V\oplus W, \Phi)$
on a compact  Riemann surface $\Sigma$ of genus $g>1$ for which $\deg V >\deg W$, as given in Definition \ref{higgsupp}, which have non-singular spectral curve,  and triples $(\bar S,U_{1},D)$ where 
\begin{itemize}
 \item $\bar \rho:\bar S\rightarrow \Sigma$ is a non-singular $p$-fold cover of $\Sigma$  given by the equation
\[p(\xi)=\xi^{p}+a_{1}\xi^{p-1}+\ldots+a_{p-1}\xi+a_{p}=0\]
for $a_{i}\in H^{0}(\Sigma,K^{2i})$ and $\xi$ the tautological section of $\bar \rho^{*}K^{2}$;
 \item $U_{1}$ is a line bundle on $\bar S$ whose degree is 
$$\deg U_{1} = \deg V +(2p^{2}-2p)(g-1) ;$$
\item $D$ is a positive subdivisor of the divisor of $a_{p}$ of degree $\tilde{m}=\deg W-\deg V+2p(g-1).$
\end{itemize}
}

The above approach  relies heavily on the existence of a non singular locus in the $GL(2p,\mathbb{C})$ Hitchin base defining a smooth curve $\bar S$ associated to a $U(p,p)$-Higgs bundle. For groups where this locus is empty, a different construction needs to be made. 
In Chapter \ref{ch:sppp} we study the case of $Sp(2p,2p)$-Higgs bundles, for which we show that the characteristic equations always define  a reducible spectral curve. In this case, it is via the square root of the characteristic polynomial $P(x):=p^{2}(x)$ that we obtain the spectral data of $Sp(2p,2p)$-Higgs bundles: 

\noindent \textbf{Theorem \ref{T0}}. \textit{
Each stable  $Sp(2p,2p)$-Higgs bundle $ (E=V\oplus W, \Phi)$ on a compact Riemann surface $\Sigma$ of genus $g\geq 2$ for which $p(\eta^{2})=0$ defines a smooth curve, has an associated pair $(S,M)$ where 
\begin{enumerate}
 \item[(a)] the curve $\rho:S\rightarrow \Sigma$ is a smooth $2p$-fold cover of $\Sigma$ given by the equation
\[p(\eta^{2})=\eta^{2p}+b_{1}\eta^{2p-2}+\ldots+b_{p-1}\eta^{2}+b_{p}=0,\]
in the total space of $K$,  where $b_{i}\in H^{0}(\Sigma,K^{2i})$, and $\eta$ is the tautological section of $\rho^{*}K$. The curve $S$ has a natural involution $\sigma$ acting by $\eta \mapsto -\eta$;
 \item[(b)]  the vector bundle $M$ is a rank 2  vector bundle on the smooth curve $\rho:S\rightarrow \Sigma$ with determinant bundle $\Lambda^{2}M\cong\rho^{*}K^{-2p+1}$, and such that $\sigma^{*}M\cong M$.
Over the fixed points of the involution, the vector bundle $M$ is acted on by $\sigma$ with eigenvalues $+1$ and $-1$.
\end{enumerate}
Conversely, a pair $(S,M)$ satisfying (a) and (b) induces a stable  $Sp(2p,2p)$-Higgs bundle $ (\rho_{*} M=V\oplus W, \Phi)$  on the compact Riemann surface $\Sigma$.
}
\bigskip

The intention of the preceding chapters is to provide a new approach to study the topology of moduli spaces of flat G-bundles through the analysis of $G$-Higgs bundles.
Since the discriminant locus depends on the complex structure, it is very likely that the fixed point locus intersects a generic fibre.  In an attempt to show the utility of this new approach, in Chapter \ref{ch:applications} we give some applications of the results which appear in Chapters \ref{ch:split}-\ref{ch:sppp}. Using the results of  Chapters \ref{ch:split}-\ref{ch:monodromy}, we can calculate the number of connected components of the moduli space of $SL(2,\mathbb{R})$-Higgs bundles:

\noindent \textbf{Proposition \ref{coro}}. \textit{The number of connected components of the moduli space of semistable $SL(2,\mathbb{R})$-Higgs bundles  is $
2\cdot 2^{2g}+2g-3 $.
}

In the case of $U(p,p)$-Higgs bundles, we use the associated spectral data to describe the connected components of the moduli space $\mathcal{M}_{U(p,p)}$ which intersect regular fibres of the classical Hitchin fibration: 

\noindent \textbf{Proposition \ref{something4}}. \textit{ For each fixed invariant $m$ as given in Proposition \ref{teo2}, the invariant $0\leq \tilde{m}<2p(g-1)$ labels exactly one connected component of $\mathcal{M}_{U(p,p)}$ which intersects the non-singular fibres of the Hitchin fibration
\[\mathcal{M}_{GL(2p,\mathbb{C})}\rightarrow \mathcal{A}_{GL(2p,\mathbb{C})},\]
given by the fibration of $\alpha^{*}{\rm Jac}(\bar S)$ over  a Zariski open set in 
\[\mathcal{E}\oplus \bigoplus_{i=1}^{p-1}H^{0}(\Sigma,K^{2i}),\]
for $\mathcal{E}$ the total space of a vector bundle over a symmetric product as defined in Section \ref{defEE}.}

In the case of $Sp(2p,2p)$-Higgs bundles,  we can describe the spectral data  through the results of \cite{AG} in terms of parabolic vector bundles, and analyse  connectivity   via the work of Nitsure \cite{nitin}. In particular, we have the following results:

\noindent \textbf{Proposition \ref{something7}}. \textit{
There is exactly one connected component of the moduli space $\mathcal{M}_{Sp(2p,2p)}$ which intersects the regular fibres of the $Sp(4p,\mathbb{C})$ Hitchin fibration, and it is given  by  the fibration of a Zariski open set in $\mathcal{P}_{a}^{s}$, over a Zariski open set in the space 
 $$\bigoplus_{i=1}^{p} H^{0}(\Sigma, K^{2i}),$$
where $\mathcal{P}_{a}^{s}$ is the set of admissible rank 2 parabolic vector bundles on a quotient curve as defined in Subsection \ref{sec:par}. }

Since many paths have been opened by the spectral data study of $G$-Higgs bundles, we dedicate Chapter \ref{ch:further} to outline a number of interesting questions that could be approached using the spectral data of real Higgs bundles. In particular, we conjecture the following: \\

\noindent \textbf{Conjecture \ref{con1}}. \textit{
 For $G=U(p,p)$, $SU(p,p)$ and $Sp(2p,2p)$, any component of the fixed point set of the involution $\Theta_{G}$, as defined in Section \ref{secinvo}, intersects the so-called generic fibre of the Hitchin fibration $\mathcal{M}_{G^{c}}\rightarrow \mathcal{A}_{G^{c}}$.\\ }
 
Among other things, this would imply that the connectivity results given in Chapter \ref{ch:applications} do in fact realise all components of the moduli spaces $\mathcal{M}_{G}$. In Section \ref{sec:hig} we discuss how the methods developed in this thesis provide a new way of obtaining information about the higher cohomology groups of the moduli spaces $\mathcal{M}_{G}$. Finally, in Section \ref{sec:other} we give an insight into how the spectral data approach could be extended to other real forms.




\chapter{Higgs bundles for complex Lie groups }
\label{ch:complex}

In this Chapter we introduce classical Higgs bundles and then study principal $G^{c}$-Higgs bundles for   complex semisimple Lie groups $G^{c}$.
We begin with an introduction to the subject following \cite{N1,N2}, and in subsequent sections we study $G^{c}$-Higgs bundles for $G^{c}= SL(n,\mathbb{C})$, $Sp(n,\mathbb{C})$, $SO(2n+1,\mathbb{C})$ and $SO(2n,\mathbb{C})$. In each case, we give a description of the corresponding invariant polynomials, the Hitchin fibration and the associated spectral curve. Following \cite{N2} and \cite{N3}, we describe the generic fibres of the Hitchin fibration in terms of an associated spectral curve and a line bundle on it. We shall continue to follow this approach in Chapter \ref{ch:supp} and Chapter \ref{ch:sppp}.

\section{Classical Higgs bundles}  \label{sec:mod}

We devote this section to developing the general theory of Higgs bundles. More details can be found in \cite{N1}, \cite{N2}, \cite{donald}, \cite{cor}, \cite{simpson88}, \cite{nit} and \cite{simpson}.

\subsection{Moduli space of vector bundles}

Holomorphic vector bundles on a compact Riemann surface $\Sigma$ of genus $g\geq 2$ are  topologically classified by their rank $n$ and degree $d$. 

\begin{definition}
 The {\rm slope} of a holomorphic vector bundle $E$ on $\Sigma$ is given by
\[\mu(E):=\frac{{\rm deg}(E)}{{\rm rk}(E)}.\]
\end{definition}

 A vector bundle $E$ is said to be
\begin{itemize}
 \item { \it stable} if for any proper, non-zero sub-bundle $F\subset E$ we have $\mu(F)<\mu(E)$;
 \item  { \it semi-stable} if for any proper, non-zero sub-bundle $F\subset E$ we have the inequality $\mu(F)\leq \mu(E)$;
 \item   { \it polystable} if it is a direct sum of stable bundles all of which have the same slope.
\end{itemize}

 It is known that the space of holomorphic bundles of fixed rank and fixed degree, up to isomorphism, is not a Hausdorff space. However, through Mumford's Geometric Invariant Theory one can construct the moduli space $\mathcal{N}(n,d)$ of stable bundles of fixed rank $n$ and degree $d$, which has the natural structure of an algebraic variety.

\begin{theorem}
For coprime $n$ and $d$, the moduli space $\mathcal{N}(n,d)$  is a smooth projective algebraic variety of dimension $n^{2}(g-1)+1$.
\end{theorem}

\begin{remark}
All line bundles are stable, and thus $\mathcal{N}(1,d)$ contains all line bundles of degree $d$, and is isomorphic to the Jacobian ${\rm Jac}^{d}(\Sigma)$ of $\Sigma$, an abelian variety of dimension $g$.
\end{remark}

Let $G^{c}$ be a complex semisimple Lie group.  Following \cite{ram} we define stability for principal $G^{c}$-bundles as follows (the reader should refer to \cite[Section 1.1]{ap} for a detailed description of the definition below).

\begin{definition}
 A holomorphic principal $G^{c}$-bundle $P\rightarrow ~\Sigma$ is said to be {\rm stable} (respectively {\rm semi-stable}) if  for every reduction $\sigma : \Sigma \rightarrow P/Q $ to maximal parabolic subgroups $Q$ of $G^{c}$ we have
\[{\rm deg} \sigma^{*} T_{rel}>0 ~ {~\rm~}({~\rm ~resp.}~ \geq 0~),\]
where $T_{rel}$ denotes the relative tangent bundle for the projection $P/Q\rightarrow \Sigma$.
\end{definition}

The notion of polystability may be carried over to principal $G^{c}$-bundles, allowing one to construct the moduli space of polystable principal $G^{c}$-bundles of fixed topological type over the compact Riemann surface $\Sigma$.

\subsection{Moduli space of Higgs bundles}

Classically, a Higgs bundle  on the compact Riemann surface $\Sigma$ is defined as follows. 

  \begin{definition}\label{def:clasical}
   A {\rm Higgs bundle} is a pair $(E,\Phi)$ for $E$ a holomorphic vector bundle on $\Sigma$, and $\Phi$ a section in $H^{0}(\Sigma,{\rm End}(E)\otimes K)$. The map $\Phi$ is called the {\rm Higgs field}. 
  \end{definition}

A vector subbundle $F$ of $E$ for which $\Phi(F)\subset F\otimes K$ is said to be a $\Phi$-{\it invariant subbundle} of $E$. Stability for Higgs bundles is defined in terms of $\Phi$-invariant subbundles:

\begin{definition}
 A Higgs bundle $(E,\Phi)$ is
\begin{itemize}
 \item {\rm stable} if for each proper $\Phi$-invariant subbundle $F$ one has $\mu(F)<\mu(E)$;
 \item {\rm semi-stable} if for each $\Phi$-invariant subbundle $F$  one has $\mu(F)\leq \mu(E)$;
 \item {\rm polystable} if $(E,\Phi)=(E_{1},\Phi_{1})\oplus (E_{2},\Phi_{2})\oplus \ldots \oplus (E_{r},\Phi_{r})$, where $(E_{i},\Phi_{i})$ is stable with $\mu (E_{i})=\mu(E)$ for all $i$.
\end{itemize}

\end{definition}

\begin{example} \label{exa} For the Riemann surface $\Sigma$ of genus $g>1$, choose a square root $K^{1/2}$ of the canonical bundle $K$, and  a section $\omega$ of $K^{2}$. Consider $E=K^{\frac{1}{2}}\oplus K^{-\frac{1}{2}}$.  Then, the Higgs bundle $(E,\Phi)$ for $\Phi$  given by
\begin{eqnarray}
 \Phi=\left(\begin{array}{cc}
0&\omega\\1&0
            \end{array}\right) \in H^{0}(\Sigma, {\rm End}E\otimes K)
\nonumber
\end{eqnarray}
is stable. In fact, since $K^{\frac{1}{2}}$ is not $\Phi$-invariant, there are no subbundles of positive degree preserved by $\Phi$.  
\end{example}

Stable Higgs bundles satisfy many interesting properties. Among others, one should note that if a Higgs bundle $(E,\Phi)$ is stable, then for $\lambda\in \mathbb{C}^{*}$ and $\alpha$ a holomorphic automorphism of $E$, the induced Higgs bundles $(E,\lambda \Phi)$ and $(E,\alpha^{*}\Phi)$ are stable.

  \begin{definition}\label{principalLie}
   A $G^{c}${\rm -Higgs bundle } is a pair $(P,\Phi)$ where  $P$ is a principal $G^{c}$-bundle over $\Sigma$, and the Higgs field $\Phi$ is a holomorphic section of the vector bundle ${\rm ad}P\otimes_{\mathbb{C}} K$, for ${\rm ad}P$ the vector bundle associated to the adjoint representation. 
  \end{definition}
 
\begin{example}
Note that in Example \ref{exa}, the Higgs bundle $(E,\Phi)$ has traceless Higgs field, and the determinant bundle $\Lambda^{2}E$ is trivial. Hence, $(E,\Phi)$ is an example of an $SL(2,\mathbb{C})$-Higgs bundle.
\end{example}

When $G^{c}\subset GL(n,\mathbb{C})$, a $G^{c}$-Higgs bundle gives rise to a Higgs bundle in the classical sense, in general with
some extra structure reflecting the definition of $G^{c}$. Note that classical Higgs bundles are given by $GL(n,\mathbb{C})$-Higgs bundles.  

In order to define the moduli space of classical Higgs bundles, we shall first define an appropriate equivalence relation. For this, consider a strictly semi-stable Higgs bundle $(E, \Phi)$. As it is not stable, $E$ admits a subbundle $F\subset E$ of the same slope which is preserved by $\Phi$. If $F$ is a subbundle of $E$ of least rank and same slope which is preserved by $\Phi$, it follows that $F$ is stable and hence the induced pair $(F,\Phi)$ is stable.  Then, by induction one obtains a flag of subbundles
\[F_{0}=0\subset F_{1}\subset \ldots \subset F_{r}=E\] 
where $\mu(F_{i}/F_{i-1})=\mu(E)$ for $1\leq i\leq r$, and where the induced Higgs bundles $(F_{i}/F_{i-1}, \Phi_{i})$ are stable. This is the \textit{Jordan-H\"{o}lder filtration} of $E$, and it is not unique.  However, the graded object
\[{\rm Gr}(E,\Phi):=\bigoplus_{i=1}^{r}(F_{i}/F_{i=1},\Phi_{i})\]
is unique up to isomorphism.

\begin{definition}
 Two semi-stable Higgs bundles $(E,\Phi)$ and $(E',\Phi')$ are said to be $S${\rm -equivalent } if ${\rm Gr}(E,\Phi)\cong {\rm Gr}(E',\Phi')$.
\end{definition}

\begin{remark} If a pair $(E,\Phi)$ is strictly stable, then the induced Jordan-H\"{o}lder filtration is trivial, and so the isomorphism class of the graded object is the isomorphism class of the original pair. 
\end{remark}

From \cite[Theorem 5.10]{nit} we let $\mathcal{M}(n,d)$ be the moduli space of $S$-equivalence classes of semi-stable Higgs bundles of fixed degree $d$ and fixed rank $n$. The moduli space $\mathcal{M}(n,d)$  is a quasi-projective scheme, and has an open subscheme $\mathcal{M}'(n,d)$ which is the moduli scheme of stable pairs. Thus every point is represented by either a stable or a polystable Higgs bundle. When $d$ and $n$ are coprime, the moduli space $\mathcal{M}(n,d)$  is smooth.
The cotangent space of $\mathcal{N}(n,d)$ over the stable locus is contained in $\mathcal{M}(n,d)$ as a Zariski  open subset.  The moduli space $\mathcal{M}(n,d)$  is a non-compact  variety which has complex dimension $ 2n^{2}(g-1)+2$. Moreover,  it is a hyperk\"ahler manifold with natural symplectic form $\omega$ defined on the infinitesimal deformations
 $(\dot A,\dot \Phi)$ of a Higgs bundle $(E,\Phi)$,  for $\dot A \in \Omega^{01}({\rm End}_{0} E)$ and $\dot\Phi\in \Omega^{10}({\rm End}_0 E)$,  by
\begin{equation}\label{ch2:2.1}
 \omega((\dot{A}_{1},\dot{\Phi}_{1}),(\dot{A}_{2},\dot{\Phi}_{2}))=\int_{\Sigma}{\rm tr}(\dot{A}_{1}\dot{\Phi}_{2}-\dot{A}_{2}\dot{\Phi}_{1})
\end{equation}
(see \cite{N1},\cite{N2} for details). For simplicity, we shall fix $n$ and $d$ and write $\mathcal{M}$ for $\mathcal{M}(n,d)$.

By extending the stability definitions for principal $G^{c}$-bundles, one can define \textit{stable}, \textit{semi-stable} and $polystable$ $G^{c}$-Higgs bundles. Moreover, by reducing to parabolic subgroups one can define the corresponding  moduli space for $G^{c}$-Higgs bundles (for details about the corresponding constructions, the reader should refer to, e.g.,  \cite[Section 3]{biswas}, \cite[Section 1]{ap}):

\begin{definition}
We denote by $\mathcal{M}_{G^{c}}$ the moduli space of polystable $G^{c}${\rm -Higgs bundles}.
\end{definition}

We shall denote by $p_{i}$,  for $i=1,\ldots, k$, a homogeneous basis for the algebra of invariant polynomials on the Lie algebra  $\mathfrak{g}^{c}$ of $G^{c}$, and let $d_{i}$ be their degrees. Following \cite{N2}, the Hitchin fibration is given by
\begin{eqnarray}
 h~:~ \mathcal{M}_{G^{c}}&\longrightarrow&\mathcal{A}_{G^{c}}:=\bigoplus_{i=1}^{k}H^{0}(\Sigma,K^{d_{i}}),\\
 (E,\Phi)&\mapsto& (p_{1}(\Phi), \ldots, p_{k}(\Phi)).
\end{eqnarray}
The map $h$ is referred to as the Hitchin~map, and is a proper map for any choice of basis (see \cite[Section 4]{N2} for details). Furthermore, the dimension of the vector space $\mathcal{A}_{G^{c}}$ always satisfies
\[{\rm dim} \mathcal{A}_{G^{c}} ={\rm dim}\mathcal{M}_{G^{c}}/2,\]
and the Hitchin map makes  the Higgs bundle moduli space into an integrable system.
When the group $G^{c}$ being considered is implicit, we shall drop the subscript $G^{c}$ and refer to the moduli space $\mathcal{M}$ and the Hitchin~base $\mathcal{A}$.


In the remainder of this Section, following \cite{N3}, we shall describe the generic fibres of the Hitchin fibration for classical Higgs bundles. Then, in Section \ref{sec:prin} we  describe the generic fibres in the case of $G^{c}$-Higgs bundles for the classical groups $G^{c} = SL(n,\mathbb{C})$, $Sp(2n,\mathbb{C})$, $SO(2n+1,\mathbb{C})$ and $SO(2n,\mathbb{C})$.

As before, let $K$ be the canonical bundle of $\Sigma$, and $X$ its total space with projection $\rho:X\rightarrow \Sigma$. We shall denote by $\eta$ the tautological section of the pull back $\rho^{*}K$ on $X$, and abusing notation we denote with the same symbols the sections  of powers $K^{i}$ on $\Sigma$ and their pull backs to $X$. Consider a smooth curve $S$ in $X$ with equation
\begin{eqnarray}
 \eta^{n}+a_{1}\eta^{n-1}+a_{2}\eta^{n-2}+\ldots + a_{n-1}\eta+a_{n}=0,\label{smootheq}
\end{eqnarray}
for $a_{i}\in H^{0}(\Sigma,K^{i})$.  By the adjunction formula on $X$, since the canonical bundle $K$ has trivial cotangent bundle one has $K_{S}\cong \rho^{*}K^{n}$, and hence 
 \[g_{S}=1+n^{2}(g-1).\]
 Starting with a line bundle $M$ on the smooth curve $\rho:S\rightarrow \Sigma$ with equation as in (\ref{smootheq}), we shall obtain a classical Higgs bundle   by considering the direct image $\rho_{*}M$ of $M$. Recall that by definition of direct image, given an open set $\mathcal{U}\subset \Sigma$, one has
\begin{eqnarray}
 H^{0}(\rho^{-1}(\mathcal{U}),M)= H^{0}(\mathcal{U}, \rho_{*}M).\label{directimage}
\end{eqnarray}
Multiplication by the the tautological section $\eta$ induces the map
\begin{eqnarray}
\xymatrix{ H^{0}(\rho^{-1}(\mathcal{U}),M)\ar[r]^{\eta} &H^{0}(\rho^{-1}(\mathcal{U}),M\otimes\rho^{*}K ).}\nonumber
\end{eqnarray}
From (\ref{directimage}) this map can be pushed down to obtain
\begin{eqnarray}
\xymatrix{\Phi:    \rho_{*}M\ar[r]&  \rho_{*}M \otimes K,}\nonumber
\end{eqnarray}
defining a Higgs field $\Phi\in H^{0}(\Sigma,{\rm End}E \otimes K)$ for $E:=\rho_{*}M$.  From Grothendieck-Riemann-Roch, one has ${\rm deg} E={\rm deg} M + (n^{2}-n)(1 - g)$. Moreover, the Higgs field satisfies its characteristic equation, which by construction is 
\[\eta^{n}+a_{1}\eta^{n-1}+a_{2}\eta^{n-2}+\ldots + a_{n-1}\eta+a_{n}=0.\]
Since $\Phi$ satisfies the above equation, $\eta$ gives an eigenvalue of $\Phi$.
The characteristic polynomial of $\Phi$ restricted to an invariant subbundle would divide the characteristic polynomial of $\Phi$. Since $S$ is smooth, it is irreducible, and thus there are no invariant subbundles of the Higgs field. Hence, the induced Higgs bundle $(E,\Phi)$ is stable.
Recall that the Norm map
$${\rm Nm}: {\rm Pic}(S)\rightarrow {\rm Pic}(\Sigma),$$
 associated to $\rho$,  is defined on  divisor classes by ${\rm Nm}( \sum n_{i}p_{i})=\sum n_{i} \rho(p_{i})$. 
The kernel of the Norm map is the Prym variety, and is denoted by ${\rm Prym}(S,\Sigma)$.  From \cite[Section 4]{bobi} the determinant bundle of $M$ satisfies
\[\Lambda^{n}\rho_{*}M\cong {\rm Nm}(M)\otimes K^{-n(n-1)/2}.\]
For $\mathcal{V}\subset S$ an open set, we have $\mathcal{V}\subset \rho^{-1}(\rho(\mathcal{V}))$ and hence a natural restriction map
\[  H^0(\rho^{-1}(\rho(\mathcal{V})), M) \rightarrow H^0(\mathcal{V},M),\]
which gives the evaluation map $ev: ~\rho^*\rho_*M\rightarrow M$. Multiplication by $\eta$  commutes with this linear map and so the action of $\rho^*\Phi$ on the dual of the vector bundle $\rho^*\rho_*M$ preserves a one-dimensional subspace. Hence $M^*$ is an eigenspace of $\rho^*\Phi^t$, with eigenvalue $\eta$. Equivalently, $M$ is the cokernel of $\rho^*\Phi-\eta$ acting on $\rho^*E\otimes \rho^{*}K^{*}$. By means of the Norm map, this correspondence can be seen on  the curve $S$ via the exact sequence 
\begin{eqnarray}
 \xymatrix{0\ar[r]&M\otimes \rho^{*}K^{1-n}\ar[r]&\rho^{*}E\ar^{\rho^{*}\Phi-\eta}[r]&\rho^{*}(E\otimes K)\ar^{ev}[r]&M\otimes \rho^{*}K\ar[r]&0},\label{sequence1}
\end{eqnarray}
and its dualised sequence
\begin{eqnarray}
 \xymatrix{0\ar[r]& M^{*}\otimes \rho^{*}K^{*}\ar[r]&\rho^{*}(E^{*}\otimes K^{*})\ar[r]&\rho^{*}E^{*}\ar[r]&M^{*}\otimes \rho^{*}K^{n-1}\ar[r]&0}. \label{sequence1dual}
\end{eqnarray}
In particular, from the relative duality theorem one has that
\begin{eqnarray}
 \rho_{*}(M)^{*}\cong \rho_{*}(K_{S}\otimes \rho^{*}K^{-1}\otimes M^{*}),
\end{eqnarray}
and thus $E^{*}$ is the direct image sheaf $\rho_{*}(M^{*}\otimes \rho^{*}K^{n-1})$.


Conversely, let $(E,\Phi)$ be a classical Higgs bundle. The characteristic polynomial is
\[{\rm det}(x-\Phi)=x^{n}+a_{1}x^{n-1}+a_{2}x^{n-2}+\ldots + a_{n-1}x+a_{n}.\] 
The coefficients of the above polynomial define the $spectral~curve$ $S$ of the Higgs bundle $(E,\Phi)$ in the total space $X$. The equation of $S$ is as (\ref{smootheq}), i.e.,  
\begin{eqnarray}
 \eta^{n}+a_{1}\eta^{n-1}+a_{2}\eta^{n-2}+\ldots + a_{n-1}\eta+a_{n}=0, \nonumber
\end{eqnarray}
for $a_{i}\in H^{0}(\Sigma,K^{i})$. 

From \cite[Proposition 3.6]{bobi}, there is a bijective correspondence between Higgs bundles $(E,\Phi)$ and the line bundles $M$ on the spectral curve $S$ described previously. This correspondence identifies the fibre of the Hitchin map with the Picard variety of line bundles of the appropriate degree.  By tensoring the line bundles $M$ with a chosen line bundle of degree $-{\rm deg}(M)$, one obtains a point in the Jacobian $\rm{Jac}(S)$, the abelian variety of line bundles of  degree zero on $S$, which has dimension $g_{S}$. In particular, the Jacobian variety is the connected component of the identity in the Picard group $H^{1}(S,\mathcal{O}^{*}_{S})$.
Thus, the fibre of the Hitchin fibration $h :\mathcal{M}\rightarrow \mathcal{A} $ is isomorphic to the Jacobian of the spectral curve $S$. For more details, the reader should refer to \cite[Section 2]{N3}.

\begin{example}\label{newexample}
In the case of a classical rank 2 Higgs bundle $(E,\Phi)$, the characteristic polynomial of $\Phi$ defines a spectral curve $\rho:S \rightarrow \Sigma$. This is a 2-fold cover of $\Sigma$ in the total space of $K$, and has  equation
$\eta^{2}+a_{2}=0$,  
for $a_{2}$ a quadratic differential  and $\eta$ the tautological section of $\rho^{*}K$. By \cite[Remark 3.5]{bobi} the curve is smooth when $a_{2}$ has simple zeros, and in this case the ramification points are  given by the divisor of $a_{2}$. For $z$ a  local coordinate near a  ramification point, the covering is given by
$ z\mapsto z^{2}:=w.$ In a neighbourhood of $z=0$, a section of the line bundle $M$ looks like
\[f(w)=f_{0}(w)+zf_{1}(w).\]
Since the Higgs field is obtained via multiplication by $\eta$, one has
\begin{eqnarray}
\Phi (f_{0}(w)+zf_{1}(w)) = w f_{1}(w)+ z f_{0}(w),
\end{eqnarray}
and thus  a local form of the Higgs field $\Phi$ is given by 
\[\Phi=\left(\begin{array}
              {cc}
0&w\\
1&0
             \end{array}
\right).\]
                              

\end{example}

 A similar analysis of the regular fibres of the Hitchin fibration can be done for $G^{c}$-Higgs bundles. In proceeding sections, following \cite{N2} and \cite{N3}, we describe   the fibres of the Hitchin fibration in terms of spectral data of the corresponding $G^{c}$-Higgs bundles for classical complex semisimple Lie groups. 

\begin{remark}  For generic $G^{c}$, a  description of the fibres can be obtained by means of Cameral covers  \cite{donagi1} (see also \cite{donagi}), which is equivalent to the one given in the next section for classical Lie groups. 
\end{remark}

\section{Principal Higgs bundles}\label{sec:prin}

We describe here $G^{c}$-Higgs bundles and the corresponding Hitchin fibration for classical complex semisimple Lie groups $G^{c}$. For further details, the reader should refer to  \cite{N2} and \cite{N3}.

\subsection{\texorpdfstring{$G^{c}=SL(n,\mathbb{C})$ }{SL(n,C)}}

Let $G^{c}=SL(n,\mathbb{C})$. A basis for the invariant polynomials on the Lie algebra $\mathfrak{sl}(n,\mathbb{C})$ is given by the coefficients of the characteristic polynomial of a trace-free matrix $A\in \mathfrak{sl}(n,\mathbb{C})$, which is given by 
\[{\rm det}(x - A) = x^{n} + a_{2}x^{n-2} + \ldots+ a_{n}.\]
 
Concretely, an $SL(n,\mathbb{C})$-Higgs bundle is a classical Higgs bundle $(E,\Phi)$ where the rank $n$ vector bundle $E$ has trivial determinant and the Higgs field has zero trace.  In this case, the spectral curve $\rho: S\rightarrow \Sigma$ associated to the Higgs bundle has equation
\begin{eqnarray}
 \eta^{n}+a_{2}\eta^{n-2}+\ldots + a_{n-1}\eta+a_{n}=0,
\end{eqnarray}
where $a_{i}\in H^{0}(\Sigma,K^{i})$ are the coefficients of the characteristic polynomial of $\Phi$. Generically   $S$ is a smooth curve of genus $g_{S}=1+n^{2}(g-1).$

Considering the coefficients of the characteristic polynomial of an $SL(n,\mathbb{C})$-Higgs bundle $(E,\Phi)$, one has the Hitchin fibration
\begin{eqnarray}
 h~:~ \mathcal{M}_{SL(n,\mathbb{C})}\longrightarrow\mathcal{A}_{SL(n,\mathbb{C})}:=\bigoplus_{i=2}^{n}H^{0}(\Sigma,K^{i}).
\end{eqnarray}

In this case the generic fibres of the Hitchin fibration are given by the subset of ${\rm Jac}(S)$  of line bundles $M$ on $S$ for which $\rho_{*}M=E$ and $\Lambda^{n}\rho_{*}M$ is trivial. 
 As seen previously, from \cite[Section 4]{bobi}, the determinant bundle of $M$ satisfies
\[\Lambda^{n}\rho_{*}M\cong {\rm Nm}(M)\otimes K^{-n(n-1)/2}.\]
 Thus, $\Lambda^{n}\rho_{*}M$ is trivial if and only if
\begin{eqnarray}
  {\rm Nm}(M) \cong K^{n(n-1)/2}. \label{ab1}
\end{eqnarray}
Equivalently, since $${\rm Nm}( \sum n_{i}\rho^{-1}(p_{i}))= n \sum n_{i}p_{i},$$ the determinant bundle $\Lambda^{n}\rho_{*}M$ is trivial
 if $M\otimes \rho^{*}K^{-(n-1)/2}$ is in the Prym variety. In the case of even rank $n$, equation (\ref{ab1}) implies a choice of a square root of $K$.  

 Hence, the generic fibre of the $SL(n,\mathbb{C})$ Hitchin fibration is biholomorphically equivalent to the Prym variety of the corresponding spectral curve $S$ (see \cite{N1} and \cite[Section 2.2]{N3} for more details). A further study of the generic fibres of the $SL(2,\mathbb{C})$ Hitchin fibration is done in Chapter \ref{ch:monodromy}.

\subsection{\texorpdfstring{$G^{c}=Sp(2n,\mathbb{C})$ }{Sp(2n,C)}}\label{subsec:sp}

Let $G^{c}=Sp(2n,\mathbb{C})$, and let $V$ be  $2n$ dimensional vector space with a non-degenerate
skew-symmetric form $<~,~>$. Given $v_{i},v_{j}$ eigenvectors of $A\in \mathfrak{sp}(2n,\mathbb{C})$ for eigenvalues $\lambda_{i}$ and $\lambda_{j}$, one has that 
\begin{eqnarray}
 \lambda_{i}<v_{i},v_{j}>&=&<\lambda_{i}v_{i},v_{j}>\nonumber\\
&=&<Av_{i},v_{j}>\nonumber\\
&=&-<v_{i},Av_{j}>\nonumber\\
&=&-<v_{i},\lambda_{j}v_{j}>\nonumber\\
&=&-\lambda_{j}<v_{i},v_{j}>.\nonumber
\end{eqnarray}
From the above one has that $<v_{i}, v_{j}>=0$ unless $\lambda_{i}=-\lambda_{j}$. Since $<v_{j},v_{j}>=0$, from the non-degeneracy of the symplectic inner product it follows that if $\lambda_{i}$ is an eigenvalue so is $-\lambda_{i}$. Thus, distinct eigenvalues  of $A$ must occur in $\pm \lambda_{i}$ pairs, and the corresponding eigenspaces are paired by the symplectic form. The  characteristic polynomial of $A$ must therefore be of the form
\[{\rm det}(x-A)=x^{2n}+a_{1}x^{2n-2}+\ldots+a_{n-1}x^{2}+a_{n}.\]
A basis for the invariant polynomials on the Lie algebra $\mathfrak{sp}(2n,\mathbb{C})$ is given by $a_{1}, \ldots, a_{n}$.

An $Sp(2n,\mathbb{C})$-Higgs bundle is a pair $(E,\Phi)$ for $E$ a rank $2n$ vector bundle  with a symplectic form $\omega(~,~)$, and the Higgs field $\Phi\in H^{0}(\Sigma, {\rm End}(E)\otimes K)$ satisfying
$$\omega(\Phi v,w)=-\omega(v,\Phi w).$$ 
The volume form $\omega^{n}$ trivialises the determinant bundle $\Lambda^{2n}E^{*}$. The characteristic polynomial ${\rm det}(\eta-\Phi)$ defines a spectral curve $\rho: S\rightarrow \Sigma$ in  $X$ with equation
\begin{eqnarray}
 \eta^{2n}+a_{1}\eta^{2n-1}+\ldots+a_{n-1}\eta^{2}+a_{n}=0, \label{eqsim}
\end{eqnarray}
whose genus  is $g_{S}:=1+4n^{2}(g-1).$ The  curve $S$ has a natural involution $\sigma(\eta)=-\eta$ and thus one can define the quotient curve $\overline{S}=S/\sigma$, of which $S$ is a 2-fold cover
\[\pi: S\rightarrow \overline{S}.\]
Note that   the Norm map associated to $\pi$  satisfies $\pi^{*}{\rm Nm}(x)=x+\sigma x$, and thus the Prym variety ${\rm Prym}(S,\overline{S})$ is given by the line bundles  $L\in {\rm Jac}(S)$ for which $\sigma^{*}L\cong L^{*}$. As in the case of classical Higgs bundles,  the characteristic polynomial of a Higgs field $\Phi$ gives the Hitchin fibration
\begin{eqnarray}
 h~:~ \mathcal{M}_{Sp(2n,\mathbb{C})}\longrightarrow\mathcal{A}_{Sp(2n,\mathbb{C})}:=\bigoplus_{i=1}^{n}H^{0}(\Sigma,K^{2i}).
\end{eqnarray}

Given  an $Sp(2n,\mathbb{C})$-Higgs bundle $(E,\Phi)$, one has $\Phi^{t}=-\Phi$ and  an eigenspace $M$ of $\Phi$ with eigenvalue $\eta$ is transformed to $\sigma^{*}M$ for the eigenvalue $-\eta$. Moreover, since the line bundle $M$ is the cokernel of $\rho^{*}\Phi-\eta$ acting on $\rho^{*}(E\otimes K^{*})$, one can consider the corresponding exact sequences (\ref{sequence1}) and its dualised sequence (\ref{sequence1dual}), which identify $M^{*}$ with $M \otimes \rho^{*}K^{1-2n}$, or equivalently, $M^{2} = \rho^{*}K^{2n-1}.$ By choosing a square root $K^{1/2}$ one has a line bundle $M_{0}:=M\otimes \rho^{*}K^{-n+1/2}$ for which $\sigma^{*}M_{0}\cong M_{0}^{*}$, i.e., which is in the Prym variety  ${\rm Prym}(S,\overline{S})$.

Conversely, an $Sp(2p,\mathbb{C})$-Higgs bundle can be recovered from a line bundle $M_{o}$ in ${\rm Prym}(S,\overline{S})$, for $S$ a smooth curve with equation (\ref{eqsim}) and $\bar S$ its quotient curve. Indeed, by Bertini's theorem, such a smooth curve $S$ with equation (\ref{eqsim}) always exists. Letting $E:=\rho_{*}M$ for $M=M_{0}\otimes \rho^{*}K^{n-1/2}$, one has the exact sequences (\ref{sequence1}) and its dualised (\ref{sequence1dual}) on the curve $S$.  Moreover, since $M^{2}\cong \rho^{*}K^{2n-1}$, there is an isomorphism $E\cong E^{*}$ which induces the symplectic structure on $E$. Hence,  the generic fibres of the corresponding Hitchin fibration can be identified with the Prym variety ${\rm Prym}(S,\overline{S})$.

\subsection{\texorpdfstring{$G^{c}=SO(2n+1,\mathbb{C})$ }{SO(2n+1,C}}

We shall now consider the special orthogonal group $G^{c}=SO(2n+1,\mathbb{C})$ and the corresponding Higgs bundles. 
Following a similar analysis as in the previous case, one can see that for a generic matrix $A\in \mathfrak{so}(2n+1,\mathbb{C})$, its  distinct eigenvalues occur in $\pm \lambda_{i}$ pairs. Thus, the characteristic polynomial of $A$ must be of the form
\begin{eqnarray}{\rm det}(x-A)=x(x^{2n}+a_{1}x^{2n-2}+\ldots +a_{n-1}x^{2}+a_{n}),\label{charpolso}\end{eqnarray}
where the coefficients $a_{1}, \ldots, a_{n}$ give a basis for the invariant polynomials  on  $\mathfrak{so}(2n+1,\mathbb{C})$. 

An $SO(2n+1,\mathbb{C})$-Higgs bundle is a pair $(E,\Phi)$ for $E$ a holomorphic vector bundle of rank $2n+1$ with  a non-degenerate symmetric bilinear form $(v,w)$, and  $\Phi$ a Higgs field  in $H^{0}(\Sigma,{\rm End}_{0}(E)\otimes K)$ which satisfies
$$(\Phi v,w)=-(v,\Phi w).$$
The moduli space $\mathcal{M}_{SO(2n+1,\mathbb{C})}$  has two connected components, characterised by a class $w_{2} \in H^{2}(\Sigma,\mathbb{Z}_{2}) \cong  \mathbb{Z}_{2}$, depending on whether $E$ has a lift to a spin bundle or not. The spectral curve induced by the characteristic polynomial in (\ref{charpolso}) is a reducible curve: an $SO(2n+1,\mathbb{C})$-Higgs field $\Phi$ has always a zero eigenvalue, and from \cite[Section 4.1]{N3} the zero eigenspace $E_{0}$ is given by $E_{0}\cong K^{-n}$.

From (\ref{charpolso}), the characteristic polynomial ${\rm det}(\eta-\Phi)$  defines a component of the spectral curve, which for convenience we shall denote by $\rho:S\rightarrow \Sigma$, and  whose  equation is
\[\eta^{2n}+a_{1}\eta^{2n-2}+\ldots +a_{n-1}\eta^{2}+a_{n}=0,\]
where $a_{i}\in H^{0}(\Sigma, K^{2i})$. This is a $2n$-fold cover of $\Sigma$, with genus $g_{S}= 1+4n^{2}(g-1).$ 
The Hitchin fibration in this case is given by the map
\begin{eqnarray}
 h~:~ \mathcal{M}_{SO(2n+1,\mathbb{C})}\longrightarrow\mathcal{A}_{SO(2n+1,\mathbb{C})}:=\bigoplus_{i=1}^{n}H^{0}(\Sigma,K^{2i}),
\end{eqnarray} which sends each pair $(E,\Phi)$ to the coefficients of ${\rm det}(\eta-\Phi)$. As in the case of $Sp(2n,\mathbb{C})$, the  curve $S$ has an involution $\sigma$ which acts as $\sigma(\eta)=-\eta$. Thus, we may consider the quotient curve $\overline{S}=S/\sigma$ in the total space of $K^{2}$,  for which $S$ is a double cover
$\pi: S\rightarrow \overline{S}.$

Following \cite{N3}, the symmetric bilinear form $(v,w)$ canonically defines a skew form $(\Phi v,w)$ on $E/E_{0}$ with values in $K$. 
Moreover, choosing a square root $K^{1/2}$ one can define
\[V=E/E_{0}\otimes K^{-1/2},\]
on which the corresponding skew form  is non-degenerate. The Higgs field $\Phi$ induces a transformation $\Phi'$ on $V$ which has characteristic polynomial 
\[{\rm det }(x-\Phi')=x^{2n}+a_{1}x^{2n-2}+\ldots+ a_{n-1}x^{2}+a_{n}.\]
Note that this is exactly the case of $Sp(2n,\mathbb{C})$ described in Section \ref{subsec:sp}, and thus we may describe the above with a choice of a line bundle $M_{0}$ in the Prym variety ${\rm Prym}(S,\overline{S})$. In particular, $S$ corresponds to the smooth spectral curve of an $Sp(2n,\mathbb{C})$-Higgs bundle.

When reconstructing the vector bundle $E$ with an $SO(2n+1,\mathbb{C})$ structure from an $Sp(2n,\mathbb{C})$-Higgs bundle $(V,\Phi')$ as in \cite[Section 4.3]{N3}, there is a mod 2 invariant associated to each zero of the 
coefficient $a_{n}$ of the characteristic polynomial ${\rm det}(\eta-\Phi')$. This data comes from choosing a trivialisation of $M_{0}\in {\rm Prym}(S,\overline{S})$ over the zeros of $a_{n}$, and defines a covering $P'$ of the Prym variety ${\rm Prym}(S,\overline{S})$. The covering has two components corresponding to the spin and non-spin lifts of the vector bundle.  The identity component of $P'$, which corresponds to the spin case, is isomorphic to the dual of the symplectic Prym variety, and this is the generic fibre of the $SO(2n+1,\mathbb{C})$ Hitchin map.

\subsection{\texorpdfstring{$G^{c}=SO(2n,\mathbb{C})$}{SO(2n,C}}

Lastly, we consider $G^{c}=SO(2n,\mathbb{C})$. As in previous cases, the distinct eigenvalues  of a matrix  $A\in \mathfrak{so}(2n,\mathbb{C})$  occur in  pairs $\pm \lambda_{i}$, and thus the characteristic  polynomial of $A$ is of the form
\[{\rm det}(x-A)=x^{2n}+a_{1}x^{2n-2}+\ldots+a_{n-1}x^{2}+a_{n}.\]
In this case the coefficient $a_{n}$ is the square of a polynomial $p_{n}$, the Pfaffian, of degree $n$. A basis for the invariant polynomials on the Lie algebra $\mathfrak{so}(2n,\mathbb{C})$ is 
$$a_{1},a_{2},\ldots, a_{n-1}, p_{n},$$
(the reader should refer, for example, to \cite{asla} and references therein for further details). 

An $SO(2n,\mathbb{C})$-Higgs bundle is a pair $(E,\Phi)$, for $E$ a holomorphic vector bundle of rank $2n$ with a non-degenerate symmetric bilinear form $(~,~)$, and    $\Phi\in H^{0}(\Sigma,{\rm End}_{0}(E)\otimes K)$ the Higgs field  satisfying
$$(\Phi v,w)=-(v,\Phi w).$$           

Considering the characteristic polynomial ${\rm det}(\eta-\Phi)$ of a Higgs bundle $(E,\Phi)$ one obtains  a $2n$-fold cover $\rho:S\rightarrow \Sigma$ whose equation is given by
\[{\rm det}(\eta-\Phi)=\eta^{2n}+a_{1}\eta^{2n-2}+\ldots+a_{n-1}\eta^{2}+p_{n}^{2},\]
for $a_{i}\in H^{0}(\Sigma,K^{2i})$ and $p_{n}\in H^{0}(\Sigma, K^{n})$. Note that this curve has always singularities, which are given by   $\eta=0$.
The  curve $S$ has a natural involution $\sigma(\eta)=-\eta$, whose fixed points in this case are the singularities of $S$. 
The \textit{virtual} genus of $S$ can be obtained via the adjunction formula, giving
$g_{S}=1+4n^{2}(g-1).$
Furthermore, one may consider its non-singular model $\hat{\rho}:\hat{S}\rightarrow \Sigma$, whose genus is
\begin{eqnarray}
 g_{\hat{S}}&=& g_{S}-\# {\rm singularities}\nonumber\\
&=& 1+4n^{2}(g-1) -2n(g-1)\nonumber\\
&=& 1+2n(2n-1)(g-1).\nonumber
\end{eqnarray}
As the fixed points of $\sigma$ are double points, the involution extends to an involution $\hat{\sigma}$ on $\hat{S}$ which does not have fixed points.

Considering the associated basis of invariant polynomials for each Higgs field $\Phi$, one may define the Hitchin fibration
\begin{eqnarray}
 h~:~ \mathcal{M}_{SO(2n,\mathbb{C})}\longrightarrow\mathcal{A}_{SO(2n,\mathbb{C})}:=H^{0}(\Sigma,K^{n})\oplus \bigoplus_{i=1}^{n-1}H^{0}(\Sigma,K^{2i}).
\end{eqnarray}

In this case the line bundle associated to a Higgs bundle is defined on the desingularisation $\hat{S}$ of $S$. Since $\hat{S}$ is smooth we obtain an eigenspace bundle 
$M \subset  {\rm ker}(\eta-\Phi)$ inside the vector bundle $E$ pulled back to $\hat{S}$.
 In particular, this line bundle satisfies
 \[\hat{\sigma}^{*}M\cong M^{*}\otimes (K_{\hat{S}}\otimes K^{*})^{-1},\]
 thus defining a point in ${\rm Prym}(\hat{S},\hat{S}/\hat{\sigma})$ given by
$  L:=M\otimes (K_{\hat{S}}\otimes K^{*})^{1/2}.$

Conversely,   a Higgs bundle $(E,\Phi)$ may be recovered from a curve $S$ with has equation
$\eta^{2n}+a_{1}\eta^{2n-2}+\ldots+a_{n-1}\eta^{2}+p_{n}^{2}=0,$
and a line bundle $M$ on its desingularisation $\hat{S}$. Note that given the sections
\[s=\eta^{2n}+a_{1}\eta^{2n-2}+\ldots+a_{n-1}\eta^{2}+p_{n}^{2}\]
for fixed $p_{n}$ with simple zeros, one has a linear system whose only base points are when $\eta=0$ and $p_{n}=0$. Hence, by Bertini's theorem the generic divisor of the linear system defined by the sections $s$ has those base points as its only singularities. Moreover, as $p_{n}$ is a section of $K^{n}$, in general there are $2n(g-1)$ singularities which are generically  ordinary double points. A generic divisor of the above linear system defines a curve $S$ which has an involution $\sigma(\eta)=-\eta$ whose only fixed points are the base points. 

The involution $\sigma$ induces an involution $\hat{\sigma}$ on the desingularisation $\hat{S}$ of $S$ which has no fixed points, and thus we may consider the quotient $\hat{S}/\hat{\sigma}$ and the corresponding Prym variety ${\rm Prym}(\hat{S},\hat{S}/\hat{\sigma})$. Following a similar procedure as for the previous groups $G^{c}$, a line bundle $L\in{\rm Prym}(\hat{S},\hat{S}/\hat{\sigma}) $ induces a Higgs bundle $(E,\Phi)$ where $E$ is the direct image sheaf of $M=L\otimes (K_{\hat{S}}\otimes K^{*})^{-1/2}$.

It is thus the Prym variety of $\hat{S}$ which is a generic fibre of the corresponding Hitchin fibration.  
Since $\hat{\sigma}$ has no fixed points, the genus $g_{\hat{S}/\hat{\sigma}}$ of $\hat{S}/\hat{\sigma}$ satisfies
$2-2g_{\hat{S}}= 2(2-2g_{\hat{S}/\hat{\sigma}}).$
%
%
Hence, the dimension of the Prym variety ${\rm Prym}(\hat{S},\hat{S}/\hat{\sigma})$ is
\begin{eqnarray}
  {\rm dim}({\rm Prym}(\hat{S},\hat{S}/\hat{\sigma}))&=& g_{\hat{S}}-g_{\hat{S}/\hat{\sigma}}\nonumber \\
&=& g_{\hat{S}}-\frac{1}{2}-\frac{g_{\hat{S}}}{2}\nonumber\\
&=&\frac{1+2n(2n-1)(g-1)}{2}-\frac{1}{2}\nonumber\\
&=& n(2n-1)(g-1).\nonumber
\end{eqnarray}




\chapter{Higgs bundles for non-compact real forms }
\label{ch:real}

Higgs bundles have been shown to provide an ideal setting for the study of representations of the fundamental group of a surface into a simple Lie group (e.g. \cite{cor}, \cite{donald}, \cite{N1}, \cite{simpson}), giving a clear example of the interaction between geometry and topology. Topologically, one may study the moduli space (or character variety) of representations of the fundamental group of a closed oriented surface in a Lie group. By choosing a complex structure on the surface one turns it into a Riemann surface. The space of representations then emerges as a complex analytic moduli space of  principal Higgs bundles.

The relation between Higgs bundles and surface group representations  was originally studied by Hitchin and Simpson for complex reductive groups. The use of Higgs bundle methods to study character varieties for real groups was pioneered by Hitchin in \cite{N1} and \cite{N5}, and further developed in \cite{Go2},\cite{Go1}. In particular, the case of $G=SL(2,\mathbb{R})$ was studied by Hitchin \cite{N1}.

The results for $SL(2,\mathbb{R})$ were generalised in \cite{N5}, where Hitchin studied the case of $G=SL(n,\mathbb{R})$. Using Higgs bundles he counted the number of connected components and, in the case of split real forms, he identified a component homeomorphic to $\mathbb{R}^{{\rm dim} G(2g-2)}$ and which naturally contains a copy of a  Teichm\"uller space. This component, known as the Teichm\"uller or Hitchin component, has special geometric significance and has subsequently been studied by, among others,   Choi and Goldman \cite{choi1}, \cite{choi2}, Labourie \cite{lab1}, and by  Burger, Iozzi, Labourie and Wienhard \cite{burger}. In Chapter \ref{ch:split} we shall consider the Teichm\"uller component when we study the Hitchin fibration for split real forms.

Principal Higgs bundles have also been used by Xia and Xia-Markman in \cite{Xia}, \cite{Xia1}, \cite{Xia3}, \cite{Xia2}  to study various special cases of $G = PU(p,q)$. Among others,  Bradlow, Garcia-Prada, Gothen, Aparicio, Mundet and Oliveira have looked at connectivity questions in this area ( e.g. \cite{brad}, \cite{ap}, \cite{GP09}, \cite{GP10}, \cite{brad2}).

The aim of this Chapter is to introduce principal Higgs bundles for real forms. We begin by reviewing in Section \ref{sec:real}   definitions and properties related to real forms of Lie algebras and Lie groups (\cite{sym}, \cite{helga}, \cite{lie}, \cite{knapp} and \cite{int}), and define $G$-Higgs bundles for a real form $G$. Through the approach of \cite{N5}, we describe these Higgs bundles as the fixed points of  a certain involution on the moduli space of $G^{c}$-Higgs bundles.  In later sections we study $G$-Higgs bundles for some non-compact real forms $G$. Further analysis of the cases $G=SL(2,\mathbb{R})$, $SU(p,q)$, $U(p,q)$ and $Sp(2p,2p)$ is given in Chapters \ref{ch:monodromy}-\ref{ch:sppp}.

\section{Higgs bundles for real forms}\label{sec:real}

A theorem by Hitchin \cite{N2} and Simpson \cite{simpson88} gives the most important property of stable Higgs bundles on a compact Riemann surface $\Sigma$ of genus $g\geq2$: 

\begin{theorem}
 If a Higgs bundle $(E,\Phi)$ is stable and  ${\rm deg} ~E = 0$, then there is a unique unitary connection $A$ on $E$, compatible with the holomorphic structure, such that
\begin{eqnarray}
 F_{A}+ [\Phi,\Phi^{*}]=0~\in \Omega^{1,1}(\Sigma, {\rm End}~E), \label{2.1}
\end{eqnarray}
 where $F_{A}$ is the curvature of the connection.  
\end{theorem}
The equation (\ref{2.1}) and the holomorphicity condition
$
d_{A}''\Phi=0 \label{hit2}
$
 are known as the \textit{Hitchin equations}, where  $d_{A}''\Phi$ is the anti-holomorphic part of the covariant derivative of $\Phi$. Following \cite{N5}, the above equations can also be considered when $A$ is a connection on a principal $G$-bundle $P$, where $G$ is the compact real form of a complex Lie group $G^{c}$, and $a\rightarrow -a^{*}$ is the compact real structure on the Lie algebra. This motivates the study of Higgs bundles for real forms. 
In this section we shall first give a background on real forms for Lie groups and Lie algebras, and then introduce $G$-Higgs bundles for a real form $G$ of a complex semisimple Lie group $G^{c}$.

\subsection{Real forms}\label{split}

Let   $\mathfrak{g}^{c}$ be a complex Lie algebra with complex structure $i$, whose Lie group is $G^{c}$.  
\begin{definition}
 A {\rm real form} of $\mathfrak{g}^{c}$ is a real Lie algebra which satisfies
\[\mathfrak{g}^{c}=\mathfrak{g}\oplus i\mathfrak{g}.\]
\end{definition}

Given a real form $\mathfrak{g}$ of $\mathfrak{g}^{c}$, an element  $Z\in \mathfrak{g}^{c}$ may be written as $Z=X+iY$ for $X,Y\in \mathfrak{g}$. The mapping
\begin{eqnarray} X+iY\mapsto X-iY\label{conju} \end{eqnarray}
is called the \textit{conjugation} of $\mathfrak{g}^{c}$ with respect to $\mathfrak{g}$. A real form of $\mathfrak{g}^{c}$ can also be seen as follows:

\begin{remark}
 A real form $\mathfrak{g}$ of $\mathfrak{g}^{c}$ is given by the set of fixed points of an antilinear involution $\tau$ on $\mathfrak{g}^{c}$, i.e., a map satisfying
\begin{eqnarray}
 \begin{array}{ccc}
  \tau(\tau (X))=X,&~& \tau(zX)=\overline{z}\tau(X),\\
\tau(X+Y)=\tau(X)+\tau(Y),&~&\tau([X,Y])=[\tau (X),\tau (Y)], 
\end{array}\nonumber
\end{eqnarray}
for $X,Y\in \mathfrak{g}^{c}$ and  $z\in \mathbb{C}$. Note that conjugation with respect to $\mathfrak{g}$ satisfies these properties. 
\end{remark}
Real forms of complex Lie groups are defined in a similar way. 

\begin{definition}
 A {\rm real form} of a complex Lie group $G^{c}$ is  an antiholomorphic Lie group automorphism $\tau$ of order two:
\begin{eqnarray}
 \tau: G^{c}\rightarrow G^{c} ~,~ \tau^{2}=Id.
\end{eqnarray}
\end{definition}

Recall that every element $X\in \mathfrak{g}^{c}$ defines an endomorphism ${\rm ad}X$ of $\mathfrak{g}^{c}$ given by
\[{\rm ad}X(Y)=[X,Y] ~{~\rm~for~}~ Y\in \mathfrak{g}^{c}.\]
For ${\rm Tr}$ the trace of a vector space endomorphism, the bilinear form $$B(X,Y)={\rm Tr}({\rm ad}X{\rm ad}Y)$$ on $\mathfrak{g}^{c}\times \mathfrak{g}^{c}$ is called the \textit{Killing form} of $\mathfrak{g}^{c}$. 

\begin{definition}
 A real Lie algebra $\mathfrak{g}$ is called ${\rm compact }$ if the Killing form is negative definite on it. The corresponding Lie group $G$ is a compact Lie group.
\end{definition}

The definition of compact real Lie algebras can be considered in the context of real forms, obtaining the following classification.

\begin{definition}\label{defisplit} Let $\mathfrak{g}$ be a real form of a complex simple Lie algebra $\mathfrak{g}^{c}$, given by the fixed points of an antilinear involution $\tau$. Then, 

\begin{itemize}
             \item  if there is a Cartan subalgebra invariant under $\tau$ on which the Killing form is negative definite, the real form $\mathfrak{g}$ is called a {\rm compact real form}. Such a compact real form of $\mathfrak{g}^{c}$ corresponds to a compact real form $G$ of $G^{c}$; 

 \item if there is an invariant Cartan subalgebra on which the Killing form is positive definite,  the form is called a {\rm split (or normal) real form}. The corresponding Lie group $G$ is the split real form of $G^{c}$.   \label{compactsplit} \end{itemize}\label{realformsdef}
\end{definition}

Any complex semisimple Lie algebra  $\mathfrak{g}^{c}$ has a compact and a split real form which are unique up to conjugation via ${\rm Aut}_{\mathbb{C}} \mathfrak{g}^{c}$.

\begin{remark} Recall that all Cartan subalgebras $\mathfrak{h}$ of a finite dimensional Lie algebra $\mathfrak{g}$ have the same dimension. The rank of $\mathfrak{g}$ is defined to be this dimension, and a real form $\mathfrak{g}$ of a complex Lie algebra $\mathfrak{g}^{c}$ is split if and only if the real rank of $\mathfrak{g}$ equals the complex rank of $\mathfrak{g}^{c}$. 
\end{remark}

\begin{example}
The split real forms of the classical complex semisimple Lie algebras are 
\begin{itemize}
\item $\mathfrak{sl}(n,\mathbb{R})$ of $\mathfrak{sl}(n,\mathbb{C})$;
\item  $\mathfrak{so}(n,n+1)$ of $\mathfrak{so}(2n+1,\mathbb{C})$;
\item $\mathfrak{sp}(n,\mathbb{R})$ of $\mathfrak{sp}(n,\mathbb{C})$;
\item $\mathfrak{so}(n,n)$ of $\mathfrak{so}(2n,\mathbb{C})$.
\end{itemize}
The compact real forms of the classical complex semisimple Lie algebras are
\begin{itemize}
\item $\mathfrak{su}(n)$ of $\mathfrak{sl}(n,\mathbb{C})$;
\item  $\mathfrak{so}(2n+1)$ of $\mathfrak{so}(2n+1,\mathbb{C})$;
\item $\mathfrak{sp}(n)$ of $\mathfrak{sp}(n,\mathbb{C})$;
\item $\mathfrak{so}(2n)$ of $\mathfrak{so}(2n,\mathbb{C})$.
\end{itemize}
\end{example}

An involution $\theta$ of a real semisimple Lie algebra $\mathfrak{g}$ such that the symmetric  bilinear form
\[B_{\theta}(X,Y)=-B(X,\theta Y)\]
is positive definite is called a \textit{Cartan involution}. Any real semisimple Lie algebra has a Cartan involution, and any two Cartan  involutions $\theta_{1},\theta_{2}$ of $\mathfrak{g}$ are conjugate via an automorphism of $\mathfrak{g}$, i.e., there is a map $\varphi$ in ${\rm Aut} \mathfrak{g}$ such that $\varphi \theta_{1}\varphi^{-1}=\theta_{2}$. The decomposition of $\mathfrak{g}$ into eigenspaces of a Cartan involution $\theta$ is called the \textit{Cartan decomposition} of $\mathfrak{g}$. The following result (e.g. see \cite{knapp} ) relates Cartan involutions and real forms:

\begin{proposition}
 Let $\mathfrak{g}^{c}$ be a complex semisimple Lie algebra, and $\rho$ the conjugation with respect to a  compact real form of $\mathfrak{g}^{c}$. Then, $\rho$ is a Cartan involution of $\mathfrak{g}$.\label{conjucomp}
\end{proposition}

Consider a compact real form $\mathfrak{u}$ of a complex semisimple Lie algebra $\mathfrak{g}^{c}$,  and denote by $\theta:\mathfrak{u}\rightarrow \mathfrak{u}$ an involution of $\mathfrak{u}$.

\begin{proposition}[\cite{helga}] Any non-compact real form $\mathfrak{g}$ of a complex simple Lie algebra $\mathfrak{g}^{c}$ can be obtained from a pair $(\mathfrak{u},\theta)$, for $\mathfrak{u}$ its compact real form and $\theta$ an involution on $\mathfrak{u}$. \label{realmethod}
\end{proposition}

 For completion, we shall recall here the construction of real forms from \cite{helga}. Let $\mathfrak{h}$ be the $+1$-eigenspaces of $\theta$ and $i\mathfrak{m}$ the $-1$-eigenspace of $\theta$ acting on $\mathfrak{u}$. These eigenspaces give a decomposition of $\mathfrak{u}$ into 
\begin{eqnarray}\mathfrak{u}=\mathfrak{h} \oplus i\mathfrak{m},\label{decu}\end{eqnarray}
Note that
\begin{eqnarray}
 \mathfrak{g}^{c}
&=&  \mathfrak{h}\oplus \mathfrak{m} \oplus i (\mathfrak{h} \oplus \mathfrak{m}),  \label{dech}
\end{eqnarray}
and thus there is a natural non-compact real form $\mathfrak{g}$ of $\mathfrak{g}^{c}$ given by
\begin{eqnarray}\mathfrak{g}=\mathfrak{h} \oplus \mathfrak{m}.\label{decnon}\end{eqnarray}
Moreover, if a linear isomorphism $\theta_{0}$ induces the  decomposition as in (\ref{decnon}), then $\theta_{0}$ is a Cartan involution of  $\mathfrak{g}$ 
and   $\mathfrak{h}$ is the maximal compact subalgebra of $\mathfrak{g}$.

Following the notation of Proposition \ref{realmethod}, let $\rho$ be the antilinear involution defining the compact form $\mathfrak{u}$ of a complex simple Lie algebra $\mathfrak{g}^{c}$ whose decomposition via an involution $\theta$ is given by equation (\ref{decu}).
Moreover, let $\tau$ be an antilinear involution which defines  the corresponding non-compact real form $\mathfrak{g}=\mathfrak{h}\oplus \mathfrak{m}$ of $\mathfrak{g}^{c}$. Considering the action of the two antilinear involutions $\rho$ and $\tau$ on $\mathfrak{g}^{c}$, we may decompose the Lie algebra $\mathfrak{g}^{c}$ into eigenspaces 
\begin{eqnarray}
 \mathfrak{g}^{c}=\mathfrak{h}^{(+,+)}\oplus \mathfrak{m}^{(-,+)}\oplus (i\mathfrak{m})^{(+,-)}\oplus (i\mathfrak{h})^{(-,-)},\label{decg}
\end{eqnarray}
where the upper index $(\cdot,\cdot)$ represents the $\pm$-eigenvalue of $\rho$ and $\tau$ respectively.

 From the decomposition (\ref{decg}), the involution $\theta$ on the compact real form $\mathfrak{u}$  giving a non-compact real form $\mathfrak{g}$ of $\mathfrak{g}^{c}$ can be seen as acting on $\mathfrak{g}^{c}$ as  
\[\sigma :=\rho\tau.\]
Moreover, this induces an involution on the corresponding Lie group
\[\sigma :=G^{c}\rightarrow G^{c}.\]
\begin{remark}\label{fixedsigma}
The fixed point set $\mathfrak{g}^{\sigma}$ of $\sigma$  is given by
$\mathfrak{g}^{\sigma}=\mathfrak{h}\oplus i\mathfrak{h}, $
and thus it is the complexification of the maximal compact subalgebra $\mathfrak{h}$ of $\mathfrak{g}$. Equivalently, the anti-invariant set under $\sigma$ is given by $\mathfrak{m}^{\mathbb{C}}$. Examples of this are given  in Section \ref{sec:real1} through Section \ref{sec:real2}.

\end{remark}

\subsection{Higgs bundles for real forms}\label{secinvo}
As mentioned in Chapter \ref{ch:intro}, non-abelian Hodge theory on the compact Riemann surface $\Sigma$ gives a correspondence between the moduli space of reductive representations of $\pi_{1}(\Sigma)$ in a complex Lie group $G^{c}$ and the moduli space of $G^{c}$-Higgs bundles. The anti-holomorphic operation of conjugating by a real form $\tau$ of $G^{c}$ in the moduli space of representations can be seen via this correspondence as a holomorphic involution $\Theta$ of the moduli space of $G^{c}$-Higgs bundles.

Following \cite{N2}, in order to obtain a $G$-Higgs bundle, for $A$ the connection which solves Hitchin equations (\ref{2.1}), one requires the flat $GL(n,\mathbb{C})$ connection
\begin{eqnarray}
 \nabla=\nabla_{A} +\Phi +\Phi^{*}
\end{eqnarray}
to have  holonomy in a non-compact real form $G$ of $GL(n,\mathbb{C})$, whose real structure is $\tau$ and Lie algebra is $\mathfrak{g}$. Equivalently, for a complex Lie group $G^{c}$ with non-compact real form $G$ and real structure $\tau$, one requires
\begin{eqnarray}
\nabla=\nabla_A+\Phi-\rho(\Phi)
\end{eqnarray}
to have  holonomy in  $G$, where $\rho$ is the compact real structure of $G^{c}$. Since $A$ has holonomy in the compact real form of $G^{c}$, we have $\rho(\nabla_A)=\nabla_A$.   Hence, requiring  $\nabla=\tau(\nabla)$ is equivalent to  $\nabla_{A}=\tau(\nabla_{A})$
and $\Phi-\rho(\Phi)=\tau(\Phi-\rho(\Phi)).$
In terms of $\sigma=\rho\tau$, these two equalities are given by $\sigma(\nabla_{A})=\nabla_{A}$ and 
\begin{eqnarray}
 \Phi-\rho(\Phi)&=&\tau(\Phi-\rho(\Phi))\nonumber\\
                &=&\tau(\Phi)-\sigma(\Phi)\nonumber\\
                &=&\sigma(\rho(\Phi)-\Phi).\nonumber
\end{eqnarray}

Hence, $\nabla$ has holonomy in the real form $G$ if $\nabla_{A}$ is invariant under  $\sigma$, and $\Phi$ anti-invariant. 
In terms of a $G^{c}$-Higgs bundle $(P,\Phi)$, one has that for $\mathcal{U}$ and $\mathcal{V}$  two trivialising open sets  in the compact Riemann surface $\Sigma$, the involution $\sigma$ induces an action on the 
transition functions $g_{uv}: \mathcal{U}\cap \mathcal{V}\rightarrow G^{c}$ given by
\[g_{uv}\mapsto \sigma(g_{uv} ),\] 
and on the Higgs field by sending
\[\Phi\mapsto -\sigma(\Phi).\]
                              
Concretely,  from Remark \ref{fixedsigma}, for $G$ a real form of a complex semisimple lie group $G^{c}$, we may construct $G$-Higgs bundles as follows. For $H$ the maximal compact subgroup of $G$, we have seen that the Cartan decomposition of $\mathfrak{g}$ is given by
\[\mathfrak{g}=\mathfrak{h}\oplus \mathfrak{m},\]
for $\mathfrak{h}$ the Lie algebra of $H$, and $\mathfrak{m}$ its orthogonal complement. 
 This induces the following decomposition of the Lie algebra $\mathfrak{g}^{c}$ of $G^{c}$ in terms of the eigenspaces of the corresponding involution $\sigma$ as defined before:
\[\mathfrak{g}^{c}=\mathfrak{h}^{\mathbb{C}}\oplus \mathfrak{m}^{\mathbb{C}}.\]
Note that the Lie algebras satisfy
\begin{eqnarray}
 [\mathfrak{h}, \mathfrak{h}]\subset\mathfrak{h} ~,~{~\rm~}~[\mathfrak{h,\mathfrak{m}}]\subset\mathfrak{m}~,~{~\rm~}~[\mathfrak{m},\mathfrak{m}]\subset \mathfrak{h},\nonumber
\end{eqnarray}
and hence there is an induced isotropy representation 
\[{\rm Ad}|_{H^{\mathbb{C}}}: H^{\mathbb{C}}\rightarrow GL(\mathfrak{m}^{\mathbb{C}}). \]

From Remark \ref{fixedsigma} and the Lie theoretic definition of $G^{c}$-Higgs bundles given in Definition \ref{principalLie} in Chapter \ref{ch:complex}, one has a concrete description of $G$-Higgs bundles (for more details, see for example \cite{brad1} ):

\begin{definition}
 A {\rm principal} $G${\rm -Higgs bundle } is a pair $(P,\Phi)$ where
\begin{itemize}
 \item $P$ is a holomorphic principal $H^{\mathbb{C}}$-bundle on $\Sigma$,
 \item $\Phi$ is a holomorphic section of $P\times_{Ad}\mathfrak{m}^{\mathbb{C}}\otimes K$.
\end{itemize}

\end{definition}

\begin{example}\label{compactexample}
 For a compact real form $G$, one has $G=H$ and $\mathfrak{m}=\{0\}$, and thus $\sigma$ is the identity and the Higgs field must vanish.
Hence, a $G$-Higgs bundle in this case is just a principal $G^{c}$- bundle.
\end{example}
In terms of involutions, from the previous analysis we have the following:

\begin{remark}\label{invrelation}
Let $G$ be a real form of a complex semi-simple Lie group $G^{c}$, whose real structure is $\tau$. Then,  $G${\rm -Higgs bundles} are given by the fixed points in $\mathcal{M}_{G^{c}}$ of the involution $\Theta$ acting by
\[\Theta: ~(P,\Phi)\mapsto (\sigma(P),-\sigma(\Phi)),\]
where $\sigma=\rho\tau$, for $\rho$ the compact real form of $G^{c}$.     

\end{remark} 
Similarly to the case of $G^{c}$-Higgs bundles, there is a notion of stability, semi stability and polystability for $G$-Higgs bundles. Following \cite[Section 3]{brad} and \cite[Section 2.1]{brad1}, one can see that the polystability of a $G$-Higgs bundle for $G\subset GL(n,\mathbb{C})$ is equivalent to the polystability of the corresponding $GL(n,\mathbb{C})$-Higgs bundle. However, a $G$-Higgs bundle can be stable as a $G$-Higgs bundle but not as a $GL(n,\mathbb{C})$-Higgs bundle. We shall denote by $\mathcal{M}_{G}$ the moduli space of polystable $G$-Higgs bundles on the Riemann surface $\Sigma$.

\begin{remark}
 One should note that for $\Theta_{G}$ the involution on $\mathcal{M}_{G^c}$ associated to the real form $G$ of $G^{c}$, 
a fixed point of $\Theta_{G}$ in $\mathcal{M}_{G^c}$ gives a representation of $\pi_1(\Sigma)$ into the real form $G$ up to the equivalence of conjugation by the normalizer of $G$ in $G^c$. This may be bigger than $G$ itself, and thus two distinct classes in $\mathcal{M}_{G}$ could be isomorphic in $\mathcal{M}_{G^c}$ via a complex map. 
Hence, although there is a map from $\mathcal{M}_{G}$ to the fixed point subvarieties in $\mathcal{M}_{G^c}$, this might not be an embedding. Research in this area is currently being done by Garcia-Prada and Ramanan (details of the forthcoming work were given in \cite{GP12}). The reader should refer to \cite{GP09} for the Hitchin-Kobayashi type correspondence for real forms.
\end{remark}

As mentioned in Chapter \ref{ch:complex}, the moduli spaces $\mathcal{M}_{G^{c}}$ have a symplectic structure, which we shall denote by $\omega$. Moreover, following \cite{N2}, the involutions $\Theta_{G}$ send $\omega\mapsto -\omega$. Thus, at a smooth point, the fixed point set must be Lagrangian and so the expected dimension of $\mathcal{M}_{G}$ is half the dimension of $\mathcal{M}_{G^{c}}$. \\

By considering Cartan's classification of classical Lie algebras, we shall now construct explicitly the antilinear involutions giving rise to  non-compact real forms of a classical complex  Lie algebra $\mathfrak{g}^{c}$. The Lie groups for the classical simple complex Lie algebras and their compact real forms are:\\

\begin{table}[h]
\begin{center}
\begin{tabular}{*{4}{c}}
Lie algebra $\mathfrak{g}^{c}$ & Lie group $G^{c}$ & Compact real form $\mathfrak{u}$ & dim $\mathfrak{u}$  \\
\vspace{-0.2 cm} &&&\\
\hline
\vspace{-0.2 cm}
&&&\\
\vspace{0.2 cm}
$\mathfrak{a}_{n}$ ($n\geq 1$) & $SL(n+1,\mathbb{C})$	&$\mathfrak{su}(n+1)$  & $n(n+2)$    \\
\vspace{0.2 cm}
$\mathfrak{b}_{n}$ ($n\geq 2$) & $SO(2n+1,\mathbb{C})$	&$\mathfrak{so}(2n+1)$ & $n(2n+1)$   \\
\vspace{0.2 cm}
$\mathfrak{c}_{n}$ ($n\geq 3$) & $Sp(2n,\mathbb{C})$	&$\mathfrak{sp}(n)$    & $n(2n+1)$   \\
\vspace{0.2 cm}
$\mathfrak{d}_{n}$ ($n\geq 4$) & $SO(2n,\mathbb{C})$	&$\mathfrak{so}(2n)$   & $n(2n-1)$   \\
\label{compact table}
 \end{tabular}
\caption{Compact forms of classical Lie algebras}\label{table1}
\end{center}
\end{table}

In the following sections we shall study all the non-compact real forms $G$, with Lie algebras $\mathfrak{g}$, of classical complex Lie groups $G^{c}$, which we  obtain by using the methods described in Proposition \ref{realmethod}. For this, we consider the compact real forms $\mathfrak{u}$ as in Table \ref{table1}, and look at the non-compact real forms corresponding to different involutions $\theta$. 
For each non-compact real form $\mathfrak{g}=\mathfrak{h} \oplus \mathfrak{m}$ of $\mathfrak{g}^{c}$ with Lie group $G$, we   describe the  vector space associated to $\mathfrak{h}^{\mathbb{C}}$ and give a description of the corresponding $G$-Higgs bundles. In order to understand $G$-Higgs bundles as fixed points of the associated involution $\Theta$ acting on the moduli space of $G^{c}$-Higgs bundles, we study the action of $-\sigma$ on the corresponding Hitchin base. Note that for classical Lie algebras a basis for the ring of invariant polynomials is given by the coefficients of the characteristic polynomial of elements in the algebra, or factors of them.

\begin{rem}\label{traces}
 Let $\mathfrak{g}^{c}$ be one of the classical Lie algebras $\mathfrak{sl}(n,\mathbb{C})$, $\mathfrak{so}(2n+1,\mathbb{C})$,  and $\mathfrak{sp}(2n,\mathbb{C})$. Then, for $\pi: \mathfrak{g}^{c}\rightarrow \mathfrak{gl}(V)$ a representation of $\mathfrak{g}^{c}$, the ring of invariant polynomials of $\mathfrak{g}^{c}$ is generated by ${\rm Tr}(\pi(X)^{i})$, for $i\in \mathbb{N}$ and $X\in \mathfrak{g}^{c}$.
\end{rem}

For $I_{n}$ the unit matrix of order $n$, we denote by $I_{p,q},~ J_{n}$  and $K_{p,q}$ the matrices
\begin{eqnarray}I_{p,q}=\left(
\begin{array}
 {cc}
-I_{p}&0\\
0&I_{q}
\end{array}\right)
,~~
J_{n}=\left(
\begin{array}
 {cc}
0&I_{n}\\
-I_{n}&0
\end{array}
\right)
,~~ \small{
K_{p,q}=\left(
\begin{array}
 {cccc}
-I_{p}&0&0&0\\
0&I_{q}&0&0\\
0&0&-I_{p}&0\\
0&0&0&I_{q}
\end{array}
\right)}.\label{MatricesIJK}\end{eqnarray}

\section{\texorpdfstring{Real forms of $SL(n,\mathbb{C})$}{Real forms of SL(n,C)}}
\label{sec:real1}

\subsection{\texorpdfstring{$G=SL(n,\mathbb{R})$}{G=SL(n,R)}}
 
Consider the compact form $\mathfrak{u}=\mathfrak{su}(n)$ of $\mathfrak{g}^{c}=\mathfrak{sl}(n,\mathbb{C})$ and the involution $$\theta(X)=\overline{X}.$$
 In this case the decomposition $\mathfrak{u}=\mathfrak{h}\oplus i\mathfrak{m}$ is given by
\begin{eqnarray}
\mathfrak{h}&=& \mathfrak{so}(n),\\
i\mathfrak{m}&=& \{{\rm symmetric,~imaginary~}n\times n {~\rm matrices~of ~trace} ~0\}.
\end{eqnarray}
 Hence, the induced non-compact real form $\mathfrak{g}=\mathfrak{h}\oplus \mathfrak{m} $ is 
\[\mathfrak{g}=\mathfrak{sl}(n,\mathbb{R})=\{n \times n~ {\rm real ~matrices~of ~trace~}0\},\]
which is the split real form of $\mathfrak{sl}(n,\mathbb{C})$. 

 The antilinear involution on $\mathfrak{g}^{c}$ defining $\mathfrak{g}$ is 
$ \tau(X)=\overline{X}.$ 
The compact real structure is  $\rho(X)=-\overline{X}^{t}$ on $\mathfrak{sl}(n,\mathbb{C})$, and the involution relating both real  structures on $\mathfrak{sl}(n,\mathbb{C})$ is
\begin{eqnarray}\sigma(X) =\rho \tau (X)= -X^{t}.\label{sigma1}\end{eqnarray}
Following Remark \ref{invrelation}, we are interested in understanding the action of the involution $\Theta$ induced from (\ref{sigma1}), acting on $SL(n,\mathbb{C})$-Higgs bundles $(E,\Phi)$.
 
\begin{proposition}\label{invosln}
 The involution $\sigma$ induces the involution $$\Theta:(E,\Phi)\mapsto (E^{*},\Phi^{t})$$ on the moduli space of $SL(n,\mathbb{C})$ Higgs bundles. 
Thus,  $SL(n,\mathbb{R})$-Higgs bundles are given by the fixed points of $\Theta$ corresponding to automorphisms $f:E\rightarrow E^{*}$  giving  a symmetric  form on $E$.
\end{proposition}

Recalling that the trace is invariant under transposition, one has that 
the ring of invariant polynomials of $\mathfrak{g}^{c}$ is acted on trivially by the involution $-\sigma$ induced by (\ref{sigma1}). 

\subsection{\texorpdfstring{$G=SU^{*}(2m)$}{G=SU*(2m)}}
  
Consider the compact form $\mathfrak{u}=\mathfrak{su}(2m)$ of $\mathfrak{g}^{c}=\mathfrak{sl}(n,\mathbb{C})$, for $n=2m$, and let 
$$\theta(X)=J_{m}\overline{X}J_{m}^{-1},$$
for \[J_{m}=\left(
\begin{array}
 {cc}
0&I_{m}\\
-I_{m}&0
\end{array}
\right).\]
 In this case, we have that $\mathfrak{u}=\mathfrak{h}
\oplus i\mathfrak{m}$ for
\begin{eqnarray}
\mathfrak{h} &=& \mathfrak{sp}(m),
\\ 
i\mathfrak{m}&=& \left\{
\left(
\begin{array}
{cc}
Z_{1}&Z_{2}\\
\overline{Z}_{2}&-\overline{Z}_{1}
\end{array}
\right)~\left| ~Z_{1}\in \mathfrak{su}(m),~Z_{2}\in \mathfrak{so}(m,\mathbb{C})\right.
\right\}.
\end{eqnarray}
The induced non-compact real form $\mathfrak{g}=\mathfrak{h}
\oplus  \mathfrak{m}$ is
\[\mathfrak{g}=\mathfrak{su}^{*}(2m)=\left\{
\left(
\begin{array}{cc}
 Z_{1}&Z_{2}\\
-\overline{Z}_{2}&\overline{Z}_{1}
\end{array}
\right)~\left|
\begin{array}{c}
Z_{1}, Z_{2}~ m\times m {~\rm ~ complex~matrices, } \\
{\rm Tr}Z_{1}+{\rm Tr}\overline{Z}_{1}=0
\end{array}\right.
\right\}.\]

 The antilinear involution on $\mathfrak{g}^{c}$ which fixes $\mathfrak{g}$ is 
\begin{eqnarray}
\tau(X)=J_{m}\overline{X}J_{m}^{-1}.
\end{eqnarray}
The compact real structure of $\mathfrak{sl}(2m,\mathbb{C})$ is given by $\rho(X)=-\overline{X}^{t}$, and the involution relating the compact structure  and the non-compact structure $\tau$ is 
\begin{eqnarray}\sigma(X) = -J_{m}X^{t}J_{m}^{-1}.\label{sigma2}\end{eqnarray}

\begin{proposition}
The involution $\sigma$ induces an involution $$\Theta:(E,\Phi)\mapsto (E^{*},\Phi^{t})$$ on  $SL(2m,\mathbb{C})$ Higgs bundles. 
Since the maximal compact subgroup of $SU^{*}(2m)$ is $Sp(m)$, the isomorphism classes of $SU^{*}(2m)$-Higgs bundles are given by fixed points of the involution $\Theta$  corresponding to vector bundles $E$ which have an automorphism   $f:E\rightarrow E^{*}$ endowing it with a symplectic structure, and which trivialises its determinant bundle. 
\end{proposition}

Concretely, $SU^{*}(2m)$-Higgs bundles are defined as follows:

\begin{definition} An $SU^{*}(2m)$ Higgs bundle on a compact Riemann surface $\Sigma$ is given by
\[(E,\Phi)~ {\rm for}~ \left\{ \begin{array}{c}
                                E {\rm~ a~rank}~ 2m ~{\rm~vector ~bundle ~with ~a ~symplectic ~form~}\omega \\
\Phi \in H^{0}(\Sigma, {\rm End}_{0}(E)\otimes K)~{~\rm symmetric~with~respect~to~} \omega.
                               \end{array}\right\}\]
\end{definition}

As the trace is invariant under conjugation and transposition, one has that the involution  $-\sigma(X) = J_{m}X^{t}J_{m}^{-1}$  induced from (\ref{sigma2}) acts trivially on the ring of invariant polynomials of $\mathfrak{sl}(2m,\mathbb{C})$.

\subsection{\texorpdfstring{$G=SU(p,q)$}{G=SU(p,q)}}

Consider the compact form $\mathfrak{u}=\mathfrak{su}(p+q)$ of $\mathfrak{g}^{c}=\mathfrak{sl}(n,\mathbb{C})$, for $p+q=n$, and the involution  $$\theta(X)=I_{p,q}X I_{p,q},$$ 
where $I_{p,q}$ is defined at (\ref{MatricesIJK}).
 The compact form  may be decomposed via the action of $\theta$ as  $\mathfrak{u}=\mathfrak{h}
\oplus i\mathfrak{m}$ for
\begin{eqnarray}
\mathfrak{h}&=& 
\left\{
\left(
\begin{array}
 {cc}
a&0\\
0&b
\end{array}
\right)\left|
\begin{array}
{c}
a\in \mathfrak{u}(p),~b\in \mathfrak{u}(q)\\
\rm{Tr}(a+b)=0 
\end{array}
\right.
\right\},
\end{eqnarray}
\begin{eqnarray}
i\mathfrak{m}&=& 
\left\{
\left.
\left(
\begin{array}
 {cc}
0&Z\\
-\overline{Z}^{t}&0
\end{array}
\right)
\right|
Z ~p\times q {\rm~complex~matrix}
\right\}.
\end{eqnarray}
The induced non-compact real form  $\mathfrak{g}=\mathfrak{h}
\oplus   \mathfrak{m} $ is
\[\mathfrak{su}(p,q)=\left\{
\left(
\begin{array}{cc}
 Z_{1}&Z_{2}\\
\overline{Z}^{t}_{2}&Z_{3}
\end{array}
\right)~\left|
\begin{array}{c}
Z_{1}, Z_{3}~ {~\rm ~ skew~Hermitian~of ~order  ~}p {\rm ~and ~} q, \\
{\rm Tr}Z_{1}+{\rm Tr}Z_{3}=0~,~Z_{2}{~\rm arbitrary}
\end{array}\right.
\right\}.\]
Considering the composition $\theta\rho$ on $\mathfrak{sl}(p+q,\mathbb{C})$, for $\rho$ the antilinear involution $\rho(X)=-X^{*}$, we have that
\begin{eqnarray}
\theta\rho\left( \begin{array}{cc}
Z_{1}  &Z_{2}\\
Z_{3}& Z_{4}
 \end{array}\right)=
\theta\left( \begin{array}{cc}
  -\overline{Z}^{t}_{1}&  -\overline{Z}^{t}_{3}\\
  -\overline{Z}^{t}_{2}&  -\overline{Z}^{t}_{4}
 \end{array}\right)\nonumber
=
\left( \begin{array}{cc}
  -\overline{Z}^{t}_{1}&  \overline{Z}^{t}_{3}\\
  \overline{Z}^{t}_{2}&  -\overline{Z}^{t}_{4}
 \end{array}\right).\nonumber
\end{eqnarray}

 The fixed set of the antilinear involution  $\theta\rho$ is precisely $\mathfrak{su}(p,q)$ and hence the antilinear involution on $\mathfrak{g}^{c}$ which fixes $\mathfrak{g}$ is
 $\tau(X)=-I_{p,q}\overline{X}^{t}I_{p,q}.$ 
The compact real structure is given by $\rho(X)=-\overline{X}^{t}$ on $\mathfrak{sl}(n,\mathbb{C})$, and the  involution relating both structures on $\mathfrak{sl}(n,\mathbb{C})$ is
\begin{eqnarray}\sigma(X) = I_{p,q}X I_{p,q}.\label{sigma3}\end{eqnarray}

\begin{proposition}
 The involution $\sigma$ on the Lie algebra induces an involution $\Theta$ on the moduli space of  $SL(n,\mathbb{C})$-Higgs bundles given by $$(E,\Phi)\mapsto (E,-\Phi).$$   Hence, $SU(p,q)$ Higgs bundles are  fixed points of the involution corresponding to bundles $E$ which have an automorphism conjugate to $I_{p,q}$ sending $\Phi$ to $-\Phi$, and  whose $\pm 1$ eigenspaces have dimensions $p$ and $q$. 
\end{proposition}

The centre of $SU(p,q)$ is $U(1)$, and its  maximal compact subgroup is given by  $$H=S(U(p)\times U(q)),$$ whose complexified  Lie group is 
$$H^{\mathbb{C}}=S(GL(p,\mathbb{C})\times GL(q,\mathbb{C}))=\{(X,Y)\in GL(p,\mathbb{C})\times GL(q,\mathbb{C})~:~ {\rm det} Y = ({\rm det}X)^{-1}\}.$$ By considering the Cartan involution on the complexified Lie algebra of $SU(p,q)$ one has  
\[\mathfrak{su}(p,q)^{\mathbb{C}}=\mathfrak{sl}(p+q,\mathbb{C})=\mathfrak{h}^{\mathbb{C}}\oplus \mathfrak{m}^{\mathbb{C}},\]
where $\mathfrak{m}^{\mathbb{C}}$ corresponds to the off diagonal elements of $\mathfrak{sl}(p+q,\mathbb{C})$. Hence,  $SU(p,q)$-Higgs bundles are defined as follows:

\begin{definition}
 An $SU(p,q)$-Higgs bundle over $\Sigma$ is a pair $(E,\Phi)$ where
 $E=V_{p}\oplus V_{q}$ for $V_{p},V_{q}$ vector bundles over $\Sigma$ of rank $p$ and $q$ such that $\Lambda^{p}V_{p}\cong \Lambda^{q}V_{q}^{*}$, and the Higgs field $\Phi$ is given by
 \begin{eqnarray}
        \Phi=\left( \begin{array}
          {cc} 0&\beta\\
\gamma&0
         \end{array}\right), \label{hig}
        \end{eqnarray}
for $\beta:V_{q}\rightarrow V_{p}\otimes K$ and $\gamma:V_{p}\rightarrow V_{q} \otimes K$.  
\end{definition}

In this case $\sigma$ is an inner automorphism of $\mathfrak{sl}(p+q, \mathbb{C})$ and hence the invariant polynomials are acted on trivially by this involution. Then, when considering the action of $-\sigma$ one has that for $p$ an invariant polynomial, 
\[(-\sigma^{*}p)(X)=p(-\sigma(X))=p(-X),\]
and so the involution $-\sigma$ acts trivially on the polynomials of even degree.

\section{\texorpdfstring{Real forms of $SO(n,\mathbb{C})$}{Real forms of SO(n,C)}}

\subsection{\texorpdfstring{$G=SO(p,q)$}{G=SO(p,q)}}\label{sec:sopq}
Let $\mathfrak{u}=\mathfrak{so}(p+q)$ be the compact form of  $\mathfrak{g}^{c}=\mathfrak{so}(n,\mathbb{C})$ for $p+q=n$, and take $$\theta(X)=I_{p,q}X I_{p,q}~{~\rm for~}~p\geq q,$$
where $I_{p,q}$ is defined at (\ref{MatricesIJK}). Some details of the following constructions can be found in \cite[Chapter 3]{ap}.

  Considering the action of $\theta$, the decomposition  $\mathfrak{u}=\mathfrak{h}\oplus i\mathfrak{m}$ is given by
\begin{eqnarray}
\mathfrak{h}&=& 
\left\{
\left.
\left(
\begin{array}
 {cc}
X_{1}&0\\
0&X_{3}
\end{array}
\right)~\right|
X_{1}\in \mathfrak{so}(p),~X_{3}\in \mathfrak{so}(q)
\right\},
\\ 
i\mathfrak{m}&=& 
\left\{
\left.
\left(
\begin{array}
 {cc}
0&iX_{2}\\
iX_{2}^{t}&0
\end{array}
\right)
\right|
X_{2} ~{\rm real~} p\times q {~\rm matrix}
\right\}.
\end{eqnarray}
The induced non-compact real form  $\mathfrak{g}=\mathfrak{h}\oplus   \mathfrak{m}$ of $ \mathfrak{so}(n,\mathbb{C})$ is
\[\mathfrak{so}(p,q)=\left\{
\left(
\begin{array}{cc}
 X_{1}&X_{2}\\
X^{t}_{2}&X_{3}
\end{array}
\right)~\left|
\begin{array}{c}
{\rm All ~}X_{i}~{\rm real}~,~X_{2}~{\rm arbitrary}, \\
X_{1}, X_{3} ~{\rm ~ skew ~symmetric ~of ~order ~} p~{\rm and }~q
\end{array}\right.
\right\}.\]
In this case, the parity of $p+q$ gives the following:
\begin{itemize}
 \item if $p+q$ is even, $\mathfrak{g}$ is a split real form if and only if $p=q$;
 \item if $p+q$ is odd, $\mathfrak{g}$  is a split real form if and only if $p=q+1$.
\end{itemize}
The involution $X\mapsto I_{p,q}\overline{X}I_{p,q}$ fixes the matrices in $\mathfrak{so}(p+q,\mathbb{C})$ of the form
\[\left(\begin{array}
         {cc}
X_{1}&iX_{2}\\
iX_{2}^{t}&X_{3}
        \end{array}
\right) \cong 
\left(\begin{array}
         {cc}
X_{1}&X_{2}\\
X_{2}^{t}&X_{3}
        \end{array}
\right),  \]
where $X_{i}$ are all real, $X_{2}$ is arbitrary and $X_{1},X_{3}$ are skew symmetric.

  The antilinear involution on $\mathfrak{g}^{c}$ which fixes $\mathfrak{g}$ is
$
 \tau(X)=I_{p,q}\overline{X}I_{p,q}.
$  Note that the real structure $\tau$ commutes with the compact real structure $\rho(X)=\overline{X}$ on $\mathfrak{g}^{c}$, and thus the involution relating both structures on $\mathfrak{g}^{c}$ is
\begin{eqnarray}\sigma(X) = I_{p,q}X I_{p,q}.\label{sigma4}\end{eqnarray}

\begin{proposition}
 The involution $\sigma$ induces an involution $\Theta:~(E,\Phi)\mapsto (E, -\Phi)$
on the moduli space of $SO(p+q,\mathbb{C})$ Higgs bundles. The  $SO (p,q)$ Higgs bundles are fixed points of this involution corresponding to vector bundles $E$ which have an automorphism $f$ conjugate to $I_{p,q}$ sending $\Phi$ to $-\Phi$ and whose $\pm 1$ eigenspaces have dimensions $p$ and $q$.
\end{proposition}

The vector space $V$ associated to the standard representation of $\mathfrak{h}^{\mathbb{C}}$  can be decomposed into
$V=V_{p}\oplus V_{q}$, 
for $V_{p}$ and $V_{q}$ complex vector spaces of dimension $p$ and $q$ respectively, with orthogonal structures. 
 The maximal compact subalgebra of $\mathfrak{so}(p,q)$   is $\mathfrak{h}=\mathfrak{so}(p)\times \mathfrak{so}(q)$ and thus the Cartan decomposition of the complexification of $\mathfrak{so}(p,q)$ is given by
\[\mathfrak{so}(p,q)^{\mathbb{C}}=\mathfrak{so}(p+q,\mathbb{C})=(\mathfrak{so}(p,\mathbb{C})\oplus \mathfrak{so}(q,\mathbb{C}))\oplus \mathfrak{m}^{\mathbb{C}},\]
where
\[\mathfrak{m}=
\left\{
\left.
\left(
\begin{array}
 {cc}
0&X_{2}\\
X_{2}^{t}&0
\end{array}
\right)
\right|
X_{2} ~{\rm real~} p\times q {~\rm matrix}
\right\}.\]

\begin{definition}
 An $SO(p,q)$ Higgs bundle is a pair $(E,\Phi)$ where $E=V_{p}\oplus V_{q}$ for $V_{p}$ and $V_{q}$ complex vector spaces of dimension $p$ and $q$ respectively, with orthogonal structures, and the Higgs field is a section in $H^{0}(\Sigma, ({\rm Hom}(V_{q},V_{p})\oplus {\rm Hom}(V_{p},V_{q}))\otimes K)$  given by
\[\Phi=\left(\begin{array}{cc}
              0&\beta\\
\gamma&0
             \end{array}
\right)~{~\rm for~}~\gamma \equiv -\beta^{\rm T},\] 
where $\beta^{\rm T}$ is the orthogonal transpose of $\beta$.
\end{definition}

Since the ring of invariant polynomials of $\mathfrak{g}^{c}=\mathfrak{so}(2m+1,\mathbb{C})$ is generated by ${\rm Tr}(X^{i})$ for $X\in \mathfrak{g}^{c}$, for $p+q=2m+1$ one has
\[(-\sigma^{*} p)(X)=p(-\sigma( X))=p(-X)=p(X^{t})=p(X).\]
Therefore, the involution $-\sigma$ acts trivially on the ring of invariant polynomials of the Lie algebra $\mathfrak{so}(2m+1,\mathbb{C})$, i.e., when $p$ and $q$ have different parity. 

In the case of $\mathfrak{so}(2m,\mathbb{C})$, for $2m=p+q$, the ring of invariant polynomials is generated by ${\rm Tr}(X^{i})$ for $X\in \mathfrak{g}^{c}$ and $i<2m$, together with the Pfaffian $p_{m}$, which has degree $m$ (e.g. see \cite{asla}). Note that when the automorphism $f$ of $E$ has determinant 1, i.e., when $p$ and $q$ are even, the involution $\sigma$ preserves orientation. Hence, if $p\equiv q$ mod 4, one has
\begin{eqnarray}\begin{array}
                 {rclcl}p_{m}(-\sigma (X))&=& \sigma^{*}p_{m}(-X)& &{\rm by~definition}  \\
 &=& p_{m}(-X)&~&{\rm since~}\sigma {~\rm preserves ~orientation}  \\
&=& p_{m}(X)&~&{\rm since~}p_{m}{~\rm is~of~even~degree}.       
    \end{array}
\end{eqnarray}
Thus, the involution $-\sigma$ acts trivially on the ring of invariant polynomials of $\mathfrak{so}(2m,\mathbb{C})$ when $p$ and $q$ are even and congruent mod 4.
In particular, this is the case of the split real form $SO(p,p)$ for even $p$.

When $p$ and $q$ are odd, the involution $\sigma$ is orientation reversing, and thus one has
\begin{eqnarray}\begin{array}
                 {rclcl}p_{m}(-\sigma (X))&=& \sigma^{*}p_{m}(-X)& &{\rm by~definition}  \\
 &=& -p_{m}(-X)&~&{\rm since~}\sigma {~\rm reverses ~orientation}  \\
&=& p_{m}(X)&~&{\rm if~}p_{m} {~\rm is~of~odd~degree}.       
    \end{array}
\end{eqnarray}
 Hence, for $p$ and $q$ odd,  $-\sigma$  acts trivially on the ring of invariant polynomials of $\mathfrak{so}(2m,\mathbb{C})$ if $p\equiv q$ mod 4.

\subsection{\texorpdfstring{$G=SO^{*}(2m)$}{G=SO*(2m)}}
 Let   $\mathfrak{u}=\mathfrak{so}(2m)$ be the compact form of  $\mathfrak{so}(n,\mathbb{C})$, for $2m=n$, and consider the involution $$\theta(X)=J_{m}\overline{X}J_{m}^{-1},$$
 where $J_{m}$ is given at (\ref{MatricesIJK}).
 
   The action of  $\theta$  decomposes the compact form into $\mathfrak{u}=\mathfrak{h}\oplus i\mathfrak{m}$ for
\begin{eqnarray}
\mathfrak{h}&=& \mathfrak{u}(m)\cong \mathfrak{so}(2m)\cap \mathfrak{sp}(m),
\\ 
i\mathfrak{m}&=& 
\left\{
\left.
\left(
\begin{array}
 {cc}
X_{1}&X_{2}\\
X_{2}&-X_{1}
\end{array}
\right)
\right|
X_{1},~X_{2} \in \mathfrak{so}(m)
\right\}.
\end{eqnarray}
The induced non-compact real form $\mathfrak{g}=\mathfrak{h}\oplus \mathfrak{m}$ is
\[\mathfrak{g}=\mathfrak{so}^{*}(2m)=\left\{
\left(
\begin{array}{cc}
 Z_{1}&Z_{2}\\
-\overline{Z}_{2}&\overline{Z}_{1}
\end{array}
\right)~\left|
\begin{array}{c}
Z_{1}, Z_{2}~ m\times m {~\rm ~ complex~matrices } \\
Z_{1} ~{\rm skew~symmetric,~} Z_{2}~{\rm Hermitian}
\end{array}\right.
\right\}.\]
 
 The antilinear involution on $\mathfrak{g}$ which fixes $\mathfrak{g}$ is given by
$
 \tau(X)=J_{m}\overline{X}J_{m}^{-1}.
$
The real structure $\tau$ commutes with the compact real structure $\rho(X)=\overline{X}$ on $\mathfrak{g}^{c}$, and thus the involution relating both structures on $\mathfrak{g}^{c}$ is
\begin{eqnarray}\sigma(X) = J_{m}X J_{m}^{-1}.\label{anteultima}\end{eqnarray}

\begin{proposition}
 The involution $\sigma$ induces an involution $\Theta:~(E,\Phi)\mapsto (E, -\Phi)$ on the moduli space of $SO(2m,\mathbb{C})$ Higgs bundles. Hence, $SO^{*}(2m)$-Higgs bundles are   fixed points of $\Theta$ corresponding to vector bundles $E$ which have an orthogonal automorphism  $f$ conjugate to $J_{m}$, sending $\Phi$ to $-\Phi$ and which squares to $-1$, equipping $E$ with a symplectic structure.  
\end{proposition}

  When considering $\mathfrak{h}^{\mathbb{C}}$, the vector space associated to its standard representation has an orthogonal and symplectic structure $J$. Since $J^{-1}=J^{t}$ and $J^{2}=-1$, the vector space  may be expressed in terms of the $\pm i$ eigenspaces of $J$ as
\[E=V\oplus V^{*},\]
for $V$ a rank $m$ vector space. Thus, we have the following definition:

\begin{definition}
 An   $SO^{*}(2m)$ {\rm -Higgs bundle} is given by a pair $(E,\Phi)$ where $E=V\oplus V^{*}$ for $V$ a rank $m$ holomorphic vector bundle, and where the Higgs field $\Phi$ is given by
\[\Phi=\left(
\begin{array}
{cc}
0&\beta\\
\gamma&0 
\end{array}
\right)~{~\rm ~for~}~ \left\{
\begin{array}
 {l}
\gamma:~V\rightarrow V^{*}\otimes K  {~\rm satisfying~}\gamma=-\gamma^{t} \\
\beta:~V^{*}\rightarrow V \otimes K {~\rm satisfying~}\beta=-\beta^{t} 
\end{array}
 \right. .\]
\end{definition}

Since ${\rm det}(J_{m})=1$, the involution $\sigma$ is an inner automorphism of $\mathfrak{g}^{c}=\mathfrak{so}(2m,\mathbb{C})$ and thus it acts trivially on the ring of invariant polynomials. Furthermore, for the same reasons as in the previous case, the involution  $-\sigma$ also acts trivially on the ring of invariant polynomials of $\mathfrak{g}^{c}$.

\section{\texorpdfstring{Real forms of $Sp(2n,\mathbb{C})$}{Real forms of Sp(2n,C)}}
\label{sec:real2}

\subsection{\texorpdfstring{$G=Sp(2n,\mathbb{R})$}{G=Sp(2n,R)}}\label{sec:sp1}

In this section and the one which follows we consider the non-compact real forms of the complex Lie group $Sp(2n,\mathbb{C})$. For this, recall that the symplectic Lie algebra $\mathfrak{sp}(2n,\mathbb{C})$ is given by the set of $2n\times 2n$ complex matrices $X$ that satisfy $J_{n}X + X^{t}J_{n} = 0 $ or equivalently,
$X=- J^{-1}_{n}X^{t}J_{n}.$
Let $\mathfrak{u}$ be the compact real form $\mathfrak{u}=\mathfrak{sp}(n)$ and $$\theta(X)=\overline{X}, $$
 which can be written as $\theta(X)=J_{n}\overline{X}J_{n}^{-1}.$  The Lie algebra $\mathfrak{sp}(n)$ is given by the quaternionic skew-Hermitian matrices; that is, the set of $n\times n$ quaternionic matrices $X$ which satisfy
$X=-\overline{X}^{t}$.
The compact form  may be decomposed as $\mathfrak{u}=\mathfrak{h}\oplus i\mathfrak{m}$, for
\begin{eqnarray}
\mathfrak{h}&=& 
\mathfrak{u}(n)\cong \mathfrak{so}(2n)\cap \mathfrak{sp}(n),
\\ 
i\mathfrak{m}&=& 
\left\{
\left(
\begin{array}
 {cc}
Z_{1}&Z_{2}\\
Z_{2}&-Z_{1}
\end{array}
\right)
\left|
\begin{array}
 {c}
Z_{1}\in \mathfrak{u}(n), ~{\rm purely~imaginary}\\
Z_{2}~~{\rm symmetric,~ purely ~imaginary}
\end{array}
\right.
\right\}.
\end{eqnarray}
 The induced non-compact real form $\mathfrak{g}=\mathfrak{h}\oplus   \mathfrak{m} $
 is
\[\mathfrak{g}=\mathfrak{sp}(2n,\mathbb{R})=\left\{
\left(
\begin{array}{cc}
 X_{1}&X_{2}\\
X_{3}&-X^{t}_{1}
\end{array}
\right)~\left|
\begin{array}{c}
X_{i}~{\rm real~} n\times n {~\rm ~matrices } \\
X_{2}, X_{3}~ {\rm symmetric}
\end{array}\right.
\right\},\]
which is a split real form of $\mathfrak{g}^{c}$. 

 The antilinear involution on $\mathfrak{sp}(2n,\mathbb{C})$ which fixes $\mathfrak{g}$ is given by
$
 \tau(X)=\overline{X},
$ 
and the compact real structure of $\mathfrak{g}^{c}$ is given by $\rho(X)=J_{n}\overline{X}J_{n}^{-1}$. 
 Hence, the involution relating both real structures on $\mathfrak{g}^{c}$ is
\begin{eqnarray}\sigma(X) = J_{n}X J_{n}^{-1}.\label{sigma5}\end{eqnarray}

\begin{proposition}
 The involution $\sigma$ induces an involution \[\Theta:~(E,\Phi)\mapsto (E, -\Phi)\]
on the moduli space of $Sp(2n,\mathbb{C})$ Higgs bundles. Hence, $Sp(2n,\mathbb{R})$-Higgs bundles are given by the fixed points of $\Theta$
corresponding to vector bundles $E$ which have a symplectic isomorphism sending $\Phi$ to $-\Phi$, and whose square is the identity, endowing $E$ with an orthogonal structure. 
\end{proposition}

 As in the previous case, the $2n$ dimensional vector space associated to the standard representation of $\mathfrak{h}^{\mathbb{C}}$ has an orthogonal and a symplectic structure $J$, and hence  it may be expressed in terms of the $\pm i$ eigenspaces of $J$ as
$E=V\oplus V^{*}.$

\begin{definition}
 An   $Sp(2n,\mathbb{R})${\rm -Higgs bundle} is given by a pair $(E,\Phi)$ where $E=V\oplus V^{*}$ for $V$ a rank $n$ holomorphic vector bundle, and for $\Phi$ the Higgs field given by
\[\Phi=\left(
\begin{array}
{cc}
0&\beta\\
\gamma&0 
\end{array}
\right)~
{~\rm ~for~}~ \left\{
\begin{array}
 {l}
\gamma:~V\rightarrow V^{*}\otimes K {~\rm satisfying~}\gamma=\gamma^{t} \\
\beta:~V^{*}\rightarrow V\otimes K {~\rm satisfying~}\beta=\beta^{t} 
\end{array}
 \right. .\]
\end{definition}

The invariant polynomials of $\mathfrak{g}^{c}$ are of even degree, and hence the involution $-\sigma$ induced from (\ref{sigma5}) acts trivially on them.

\subsection{\texorpdfstring{$G=Sp(2p,2q)$ }{G=Sp(2p,2q)}}\label{sec:sp2}
Finally, we shall consider the compact real form $\mathfrak{u}=\mathfrak{sp}(p+q)$ of $Sp(2(p+q),\mathbb{C})$ and the involution $$\theta(X)=K_{p,q} X K_{p,q},$$ where
\begin{small}$$K_{p,q}=\left(
\begin{array}
 {cccc}
-I_{p}&0&0&0\\
0&I_{q}&0&0\\
0&0&-I_{p}&0\\
0&0&0&I_{q}
\end{array}
\right),$$\end{small}
as at (\ref{MatricesIJK}). The action of  $\theta$ on the compact form $\mathfrak{u}$ gives a decomposition $\mathfrak{u}=\mathfrak{h}\oplus i\mathfrak{m}$ for\begin{small}
\begin{eqnarray}
\mathfrak{h}&=& 
\left\{
\left(
\begin{array}
 {cccc}
Z_{11}&0&Z_{13}&0\\
0&Z_{22}&0&Z_{24}\\
-\overline{Z}_{13}&0&\overline{Z}_{11}&0\\
0&-\overline{Z}_{24}&0&\overline{Z}_{22}
\end{array}
\right)~\left|
\begin{array}
 {c}
Z_{11}\in \mathfrak{u}(p),~ Z_{22}\in \mathfrak{u}(q),\\
Z_{13}~p\times p, Z_{24}~q\times q  \\
{\rm symmetric ~matrices}.
\end{array}
\right.
\right\},\nonumber
\\ 
i\mathfrak{m}&=& 
\left\{
\left(
\begin{array}
 {cccc}
0&Z_{12}&0&Z_{14}\\
-\overline{Z}_{12}^{t}&0&Z_{14}^{t}&0\\
0&-\overline{Z}_{14}&0&\overline{Z}_{12}\\
-\overline{Z}_{14}^{t}&0&-Z_{12}^{t}&0
\end{array}
\right)
\left|
\begin{array}
 {c}
Z_{12~} {\rm and }~ Z_{14}~{\rm arbitrary,}\\
{\rm complex} ~p\times q {\rm ~matrices}.
\end{array}
\right.
\right\}.\nonumber
\end{eqnarray}\end{small}

Note that $\mathfrak{h} \cong \mathfrak{sp}(p)\oplus \mathfrak{sp}(q)$ via the map\begin{small}
\begin{eqnarray}
 \left(
\begin{array}
 {cccc}
Z_{11}&0&Z_{13}&0\\
0&Z_{22}&0&Z_{24}\\
-\overline{Z}_{13}&0&\overline{Z}_{11}&0\\
0&-\overline{Z}_{24}&0&\overline{Z}_{22}
\end{array}
\right) 
\mapsto 
\left\{\left(
\begin{array}
 {cc}
Z_{11}&Z_{13}\\
-\overline{Z}_{13}&\overline{Z}_{11}
\end{array}
\right),\left(
\begin{array}
 {cc}
Z_{22}&Z_{24}\\
-\overline{Z}_{24}&\overline{Z}_{22}
\end{array}
\right)\right\}.\nonumber
\end{eqnarray}\end{small}

Then, the induced non-compact real form $\mathfrak{g}=\mathfrak{h}\oplus  \mathfrak{m} $ is
\begin{small}
\[\mathfrak{sp}(p,q)=\left\{
\left(
\begin{array}{cccc}
 Z_{11}&Z_{12}&Z_{13}&Z_{14}\\
 \overline{Z}^{t}_{12}&Z_{22}&Z^{t}_{14}&Z_{24}\\
 -\overline{Z}_{13}& \overline{Z}_{14}& \overline{Z}_{11}& -\overline{Z}_{12}\\
 \overline{Z}^{t}_{14}& -\overline{Z}_{24}&-Z^{t}_{12}& \overline{Z}_{22}\\
\end{array}
\right)~\left|
\begin{array}{c}
 Z_{i,j}~ {~\rm ~ complex~matrices, } \\
Z_{11}, ~Z_{13} {~\rm order~ } p,\\
Z_{12}, ~Z_{14} ~p\times q~{\rm matrices},\\
Z_{11}, ~Z_{22} ~{\rm skew ~Hermitian},\\
Z_{13}, ~Z_{24}~{\rm symmetric}.
\end{array}\right.
\right\}.\]\end{small}

Considering the antilinear involution $-K_{p,q}X^{*}K_{p,q}$ on $\mathfrak{sp}(2(p+q),\mathbb{C})$ one has that\begin{small}
\begin{eqnarray}
\left( \begin{array}
  {cccc}
 Z_{11}      & Z_{12} & Z_{13}       & Z_{14}	\\
 Z_{21}      & Z_{22} & Z_{14}^{t}   & Z_{24}	\\
 Z_{31}      & Z_{41} & -Z_{11}^{t}  & -Z_{21}^{t}	\\
 Z_{41}^{t}  & Z_{42} & -Z_{12}^{t}  & -Z_{22}^{t}
 \end{array}\right)
\mapsto
\left( \begin{array}
  {cccc}
 -\overline{Z}^{t}_{11}      & \overline{Z}^{t}_{21} & -\overline{Z}^{t}_{31}       & \overline{Z}_{41}	\\
\overline{Z}^{t}_{12}      & -\overline{Z}^{t}_{22} & \overline{Z}^{t}_{41}   & -\overline{Z}^{t}_{42}	\\
 -\overline{Z}^{t}_{13}      & \overline{Z}_{14} & \overline{Z}_{11}  & -\overline{Z}_{12}	\\
 \overline{Z}^{t}_{14} & -\overline{Z}^{t}_{24} & -\overline{Z}_{21}  & \overline{Z}_{22}
 \end{array}\right).
\end{eqnarray}\end{small}
Here, $Z_{13},~Z_{24},~Z_{31},$ and $Z_{42}$ are symmetric matrices. 

                                                                       The antilinear involution on $\mathfrak{g}^{c}=\mathfrak{sp}(2(p+q),\mathbb{C})$ which fixes $\mathfrak{g}=\mathfrak{sp}(p,q)$ is given by
$
 \tau(X)=-K_{p,q}X^{*}K_{p,q}.
$
                                                                   The real structure $\tau$ commutes with the compact real structure $\rho(X)=J_{p+q}\overline{X}J_{p+q}^{-1}$ on $\mathfrak{g}^{c}$. Let $\tilde{K}_{p,q}$ be defined as \begin{eqnarray}
\tilde{K}_{p,q}=\left(
\begin{array}
 {cccc}
0&0&-I_{p}&0\\
0&0&0&I_{q}\\
I_{p}&0&0&0\\
0&-I_{q}&0&0
\end{array}
\right).\label{Flabel2}
\end{eqnarray}
 Then, the involution relating both structures on $\mathfrak{g}^{c}$ is
\begin{eqnarray}\sigma(X) = -\tilde{K}_{p,q}X^{t} \tilde{K}_{p,q}^{-1}=K_{p,q} X K_{p,q}.\label{sigmaultima}\end{eqnarray}

 Note that $\mathfrak{h}^{\mathbb{C}}\cong \mathfrak{sp}(2p,\mathbb{C})\oplus \mathfrak{sp}(2q,\mathbb{C})$. The vector space $E$ associated to the standard representation of $\mathfrak{h}^{\mathbb{C}}$ may be expressed as
$E=V_{2p}\oplus V_{2q},$ 
for $V_{2p}$ and $V_{2q}$ vector spaces of dimension $2p$ and $2q$ respectively, which have symplectic structures. 
The maximal compact subgroup of $Sp(2p,2q)$ is $$H=Sp(p)\times Sp(q).$$ Then, one can decompose the complexification of the Lie algebra $\mathfrak{sp}(2p,2q)$ as follows:
\[\mathfrak{sp}(2p,2q)^{\mathbb{C}}=\mathfrak{sp}(2(p+q),\mathbb{C})=(\mathfrak{sp}(2p,\mathbb{C})\oplus \mathfrak{sp}(2q,\mathbb{C}))\oplus \mathfrak{m}^{\mathbb{C}},\]
where the Lie algebra $\mathfrak{m}^{\mathbb{C}}\subset\mathfrak{sp}(p+q,\mathbb{C}) $ is given by matrices of the form
\[
\left(\begin{array}{cccc}
 0&Z_{12}&0&Z_{14}\\
Z_{21}&0&Z_{14}^{t}&0\\
0&Z_{41}&0&-Z_{21}^{t}\\
Z_{41}^{t}&0&-Z_{12}^{t}&0
\end{array}\right).
\]
We shall define
\[
B:=\left(\begin{array}
   {cc}
Z_{12}&Z_{14}\\
Z_{41}&-Z_{21}^{t}  \end{array}
\right)
~~{~\rm~and ~}~~
C:=
\left(\begin{array}
   {cc}
Z_{21}&Z_{14}^{t}\\
Z_{41}^{t}&-Z_{12}^{t}  \end{array}
\right).
\]
Note that if $Z_{12}=Z_{21}=0$, then $B=C^{t}$. With this notation, one has that
\[
 \mathfrak{m}^{\mathbb{C}}\cong\left\{
 \left(\begin{array}{cc}
0&C^{t}\\
C&0
        \end{array}\right)~{\rm for}~  
C~{\rm ~a}~(2q\times 2p)~{\rm matrix}   \right\}.
  \]

\begin{proposition}
 The involution $\sigma$ on the Lie algebra $\mathfrak{sp}(2(p+q),\mathbb{C})$ induces an involution 
\[\Theta:~(E,\Phi)\mapsto (E, -\Phi^{{\rm T}})\]
on the moduli space of $Sp(2(p+q),\mathbb{C})$-Higgs bundles.  Hence,  $Sp(p,q)$-Higgs bundles are given by the fixed points of the involution $\Theta$ 
 corresponding to symplectic vector bundles $E$ which have an endomorphism $f:E\rightarrow E$ conjugate to $\tilde{K}_{p,q}$ defined as at (\ref{Flabel2}), sending $\Phi$ to the symplectic transpose $-\Phi^{\rm T}$, and whose $\pm 1$ eigenspaces are of dimension $2p$ and $2q$. 
\end{proposition}

\begin{definition}
 An $Sp(2p,2q)${\rm -Higgs bundle} is given by a pair $(E,\Phi)$ where $E=V_{2p}\oplus V_{2q}$ is a direct sum of symplectic vector spaces of rank $2p$ and $2q$, and where the symplectic Higgs field is given by
\[\Phi=
\left(\begin{array}{cc}
0& -\gamma^{\rm T}\\
\gamma&0
        \end{array}\right)~{\rm for~}~\left\{
\begin{array}{cc}
\gamma&: V_{2p}\rightarrow V_{2q}\otimes K  \\
-\gamma^{\rm T}&: V_{2q}\rightarrow V_{2p} \otimes K
\end{array} \right. ,\]
for $\gamma^{\rm T}$ the symplectic transpose of $\gamma$.
\end{definition}

 The symplectic transpose of the map $\gamma$ may be thought of as follows. Since $V$ and $W$ are symplectic vector bundles, there are two induced isomorphisms
$q_{V}:V^{*} \rightarrow V $ and $ q_{W}:W \rightarrow W^{*}$.  
Via these isomorphisms, we define
$\gamma^{{\rm T}}= q_{V}\gamma^{t}q_{W},$
where $\gamma^{t}$ is the dual action of $\gamma$.

As the trace is invariant under conjugation and transposition, the involution $-\sigma$ induced from (\ref{sigmaultima}) acts trivially on the ring of invariant polynomials of the Lie algebra $\mathfrak{g}^{c}=\mathfrak{sp}(2(p+q),\mathbb{C})$.

\section{A geometric consequence}
\label{sec:realforms}\label{geocon}

Having studied the action of the different involutions $-\sigma$ on the ring of invariant polynomials of each complex simple lie algebra $\mathfrak{g}^{c}$, we have the following.

\begin{proposition}\label{poliin1}  For each non-compact real form of $\mathfrak{g}^{c}=\mathfrak{so}(2m+1,\mathbb{C})$ and $\mathfrak{g}^{c}=\mathfrak{sp}(2n,\mathbb{C})$, the corresponding involution $-\sigma$ acts trivially on the ring of invariant polynomials of $\mathfrak{g}^{c}$.
\end{proposition}

Hence, for $G^{c}=Sp(2n,\mathbb{C})$ and $SO(2m+1,\mathbb{C})$,  the induced involutions $$\Theta~:(E,\Phi)\mapsto (\sigma(E),-\sigma(\Phi))$$ on the moduli spaces $\mathcal{M}_{G^{c}}$ preserve each fibre of the Hitchin fibration. In the case of $SO(2m,\mathbb{C})$ one has two different situations:

\begin{proposition}\label{p2}
 For  $\mathfrak{g}^{c}=\mathfrak{so}(2m,\mathbb{C})$, the action of the different involutions $-\sigma$ is given as follows:
\begin{itemize}

 \item the involution $-\sigma$ associated to $SO^{*}(2m)$ acts trivially on the ring of invariant polynomials of $\mathfrak{g}^{c}$;

 \item the involution $-\sigma$ induced by $SO(p,q)$ for $p$ and $q$ both even, or both odd, acts trivially on the invariant polynomials of $\mathfrak{so}(2m,\mathbb{C})$ if $p\equiv q$ mod 4. Otherwise, it acts trivially  on the basic  polynomials except for the Pfaffian.

\end{itemize}

\end{proposition}

From Proposition \ref{p2}, each fibre of the $SO(2m,\mathbb{C})$ Hitchin fibration is preserved by the involution $\Theta$ defining the real form $SO^{*}(2m)$,  and by the involution defining  the real form $SO(p,q)$ only if $p\equiv q$ mod 4.

\begin{proposition}\label{p1}
 Consider $\mathfrak{g}^{c}=\mathfrak{sl}(n,\mathbb{C})$. Then, 
\begin{itemize}
 \item the involution $-\sigma$ associated to the non-compact real forms $\mathfrak{sl}(n,\mathbb{R})$ and $\mathfrak{su}^{*}(2m)$, when $n=2m$, acts trivially on the ring of invariant polynomials of $\mathfrak{g}^{c}$; 
\item the involution $-\sigma$ associated to the non-compact real form $\mathfrak{su}(p,q)$ of $\mathfrak{sl}(p+q,\mathbb{C})$ acts trivially only on the invariant polynomials of even degree.
\end{itemize}
  
\end{proposition}

From Proposition \ref{p1}, the involution $-\sigma$ induces  $\Theta~:(P,\Phi)\mapsto (\sigma(P),-\sigma(\Phi))$
 on the moduli space of $SL(n,\mathbb{C})$-Higgs bundles, which preserve each fibre of the $SL(n,\mathbb{C})$ Hitchin fibration except when $\sigma$ corresponds to the real form $SU(p,q)$. In this case, $\Theta$ only preserves the fibres over differentials of even degree.




\chapter{Higgs bundles for split real forms}\label{ch:split}

As mentioned in Chapter \ref{ch:real}, in \cite{N5} Hitchin gives an explicit description  of a component of the space of representations of $\pi_{1}(\Sigma)$ into split real forms in terms of Higgs bundles. In this chapter
we shall consider this component to study the space of Higgs bundles for split real forms. 

We begin by introducing three dimensional subalgebras as defined by Kostant (see \cite{kos}, \cite{kos2}), and follow the steps in \cite{N5} to define an involution $\sigma_{s}$ on a complex Lie algebra which is closely related to its split real form, and which coincides in the case of classical groups with the involution $\sigma$ associated to split real forms as defined in Chapter \ref{ch:real}. By means of the involution in the Lie theoretic setting, in Section \ref{hsplit}, we study  the fixed points of the induced involution $\Theta$ on the Hitchin fibration for the corresponding complex Lie group.


\section{An algebraic approach}\label{three}\label{algebraic}

In this section we set up the Lie theoretic background needed to study Higgs bundles for split real forms. We shall denote by $\mathfrak{g}^{c}$  a complex simple Lie algebra of rank $r$.

\subsection{Three dimensional subalgebras}

From \cite{kos} we consider the following subalgebra of a complex simple Lie algebra $\mathfrak{g}^{c}$ (see also details in \cite[Section 4]{N5}):

\begin{definition}A {\rm three-dimensional subalgebra
 $\mathfrak{s}$ } of $\mathfrak{g}^{c}$ is a subalgebra of $\mathfrak{g}^{c}$ generated by a semisimple element $h_{0}$, and nilpotent elements
 $e_{0}$ and $f_{0}$ of  $\mathfrak{g}^{c}$ satisfying the relations
\begin{eqnarray}[h_{0},e_{0}]=e_{0}~;~ [h_{0},f_{0}]=-f_{0}~;~ [e_{0},f_{0}]=h_{0}.\label{condition}\end{eqnarray}\label{defF}
\end{definition}

By the Jacobson-Morosov lemma (see \cite{jacobson} and \cite{morosov}), any nilpotent element of $\mathfrak{g}^{c}$ can be embedded into a three dimensional subalgebra whose generators satisfy (\ref{condition}), i.e., a copy of $\mathfrak{sl}(2,\mathbb{C})$ in $\mathfrak{g}^{c}$. A nilpotent element $e_{0}\in \mathfrak{g}^{c}$ is called  \textit{principal} if its centralizer is $r$-dimensional. When the nilpotent elements of a three-dimensional subalgebra  are principal, the subalgebra is called a \textit{principal three dimensional subalgebra}. Any complex simple Lie algebra $\mathfrak{g}^{c}$ has a unique (up to conjugacy) principal three dimensional subalgebra \cite[Section 5]{kos}.

 Following \cite[Section 5]{kos}, one may construct a principal three dimensional subalgebra of a complex simple Lie algebra $\mathfrak{g}^{c}$ as follows. 
Consider  $\mathfrak{h}$ a Cartan  subalgebra of $\mathfrak{g}^{c}$, and denote by $\Delta$ the root system. For $\alpha \in \Delta$, we let $x_{\alpha}$ be the root vectors satisfying the standard relation
\[[x_{\alpha},x_{-\alpha}]=\alpha.\]
We write $\Delta^{+}$ for  a system of positive roots and define $\Pi=\{\alpha_{1},\alpha_{2},\ldots,\alpha_{r}\}$  the corresponding simple roots. Let $\{h_{\alpha_{i}},e_{\alpha},e_{-\alpha}\}$ be a fixed basis, for $\alpha\in\Delta^{+}$ and $\alpha_{i}\in \Pi$, and  $h_{0}$ the element defined by
\begin{eqnarray}
h_{0}:=\frac{1}{2}\sum_{i=1}^{r}h_{\alpha_{i}}.
\end{eqnarray}
Then, $h_{0}$ may be expressed as $$h_{0}=\sum_{i=1}^{r}r_{i}\alpha_{i}$$ for some positive half integers $r_{i}$.  By considering the numbers $r_{i}$, we further define $f_{0}$ and $e_{0}$ as follows:
\begin{eqnarray}
e_{0}&:=&\sum_{i=1}^{r}c_{i}e_{\alpha_{i}}, \label{ecero}\\
f_{0}&:=&\sum_{i=1}^{r}(r_{i}/c_{i})e_{-\alpha_{i}},\label{fcero}
\end{eqnarray}
where  $\{c_{i}\}_{i=1}^{r}$ are any non-zero complex numbers. The subalgebra generated by $\{h_{0},e_{0},f_{0}\}$ is a principal three dimensional subalgebra.

 Given a root $\lambda\in \Delta$, the \textit{order} $o(\lambda)$ of $\lambda$ is the integer
\begin{eqnarray}
 o(\lambda)=\sum_{\alpha \in \Pi}n_{\alpha},
\end{eqnarray}
where  $n_{\alpha}\in \mathbb{N}$ are uniquely defined by
\[\lambda=\sum_{\alpha \in \Pi}n_{\alpha}\alpha.\]
Let $\psi$ be the highest root. Then, we may write
\[\psi=\sum_{i=1}^{r}q_{i}\alpha_{i}.\]

\begin{remark}Considering  $e_{0}$ and $f_{0}$ as defined as in (\ref{ecero}) and (\ref{fcero}), for $c_{i}:=\sqrt{q_{i}}$ and $r_{i}=q_{i}$,   the subalgebra
\begin{eqnarray}
\tilde{\mathfrak{s}}_{0}=<\psi,e_{0},f_{0}>\label{treslie}
\end{eqnarray}
is a principal three dimensional subalgebra of $\mathfrak{g}^{c}$.\label{segprin}\end{remark}

\begin{example}
In the case of $\mathfrak{g}^{c}=\mathfrak{sl}(n+1,\mathbb{C})$, the element $f_{0}$  of a principal three dimensional subalgebra can be expressed as\begin{small}
\begin{eqnarray}
f_{0}=\left( \begin{array}
  {ccccc}
0&1&0&\cdots&0\\
0&0&1&\ddots&\vdots\\
0&0&\ddots&\ddots&0\\
\vdots&\ddots&\ddots&0&1\\
0&\cdots&0&0&0\\
 \end{array}\right).
\end{eqnarray}
\end{small}
\end{example}

In the following sections we shall consider  
\begin{eqnarray}\mathfrak{s}_{0}=<e_{0},f_{0},h_{0}>_{\mathbb{C}}\label{Flabel3}
\end{eqnarray}   a principal three dimensional subalgebra  of $\mathfrak{g}^{c}$, which satisfies (\ref{condition}). 
 By means of the adjoint representation of $\mathfrak{s}_{0}$ we may decompose $\mathfrak{g}^{c}$ into irreducible  $\mathfrak{s}_{0}$-modules  
\begin{eqnarray}\mathfrak{g}^{c}=V_{1}\oplus \ldots \oplus V_{r}.\label{wspace}\end{eqnarray}
The irreducible summands  $V_{i}$ satisfy the following properties (see \cite[Section 3]{kos}):
\begin{itemize}
 \item Each $V_{i}$ is real of dimension  $2k_{i}+1$,  where the integers $k_{i}$ are the   \textit{exponents} of $\mathfrak{g}^{c}$.
 \item For $\{v^{i}_{-k_{i}},v^{i}_{-k_{i}+1},\ldots,v^{i}_{k_{i}-1}, v^{i}_{k_{i}}\}$ a basis of $V_{i}$, we have that
\begin{eqnarray}
 ({\rm ad}h_{0})v^{i}_{j}&=& j\cdot v_{j}^{i},\\
 ({\rm ad}e_{0})v^{i}_{j}&=& v_{j+1}^{i},\\
 ({\rm ad}f_{0})v^{i}_{j}&=&  v_{j-1}^{i}.
\end{eqnarray}
 \item The elements $v^{i}_{k_{i}}$ are   the \textit{highest weight vectors} and shall be denoted by $e_{i}:=v^{i}_{k_{i}}$. A basis for $V_{i}$ is given by
\begin{eqnarray}
 ({\rm ad}f_{0})^{k}e_{i}~{~\rm for~}~ 0\leq k \leq 2k_{i}+1.\label{hvec}
\end{eqnarray}
\item Without loss of generality, we shall assume that $V_{1}=\mathfrak{s}_{0}$, where $\mathfrak{s}_{0}$ is as at (\ref{Flabel3}), and that the the exponents of $\mathfrak{g}^{c}$ satisfy \begin{eqnarray}1=k_{1}\leq k_{2}\leq \ldots\leq k_{r}.\label{defKos2}\label{maxexp}\end{eqnarray} In particular, $e_{1}=e_{0}$. 
\end{itemize}


For $\mathfrak{g}_{j}$ the subspace of $\mathfrak{g}^{c}$ on which ${\rm ad} h_{0}$ acts with eigenvalue $j$, one may decompose the Lie algebra $\mathfrak{g}^{c}$ as
\begin{eqnarray}
 \mathfrak{g}^{c}=\bigoplus_{j=-k_{r}}^{k_{r}}\mathfrak{g}_{j}.\label{dec.e}
\end{eqnarray}
Note, in particular, that  $e_{i}\in \mathfrak{g}_{k_{i}}$.

\subsection{Compact real structure}

Let  $G^{c}$ be the complex Lie group of the complex simple Lie algebra  $\mathfrak{g}^{c}$. The principal three dimensional subalgebra $\mathfrak{s}_{0}=<h_{0},e_{0},f_{0}>_{\mathbb{C}}$ defines a homomorphism from $SL(2,\mathbb{C})$ to $G^{c}$ and thus from its compact real form $SU(2)$ to the compact real form $U$ of $G^{c}$.

One can take $\mathfrak{s}_{0}$ to be real with respect to the compact real form of $G^{c}$, and if $\rho$ is the antilinear involution on $\mathfrak{g}^{c}$ defining the compact real form, it acts on the three dimensional subalgebra as follows:
\begin{eqnarray}
\rho(h_{0})=-h_{0},~{~}~\rho(e_{0})=f_{0}.
\label{compactstructure}
\end{eqnarray}

The Lie algebra $\mathfrak{u}$ of the compact real form $U$ decomposes into the direct sum of the real representations, which by abuse of notation we denote $V_{1},\ldots,  V_{r}$. In particular,  (\ref{wspace}) is the complexification of
\[\mathfrak{u}=\bigoplus_{i=1}^{r}V_{i}.\]

\subsection{A natural involution on Lie algebras}

As before, consider $\mathfrak{g}^{c}$  a complex simple Lie algebra of rank $r$, and let $\mathfrak{s}_{0}=<h_{0},e_{0},f_{0}>$ be a principal three dimensional subalgebra as defined at (\ref{Flabel3}) in Section \ref{algebraic}. Denote by $e_{1},\ldots, e_{r}$ the highest weight vectors of the irreducible representations $V_{i}$. Following \cite[Proposition 6.1]{N5} we define the following involution.

\begin{definition}\label{invoh}There is a natural involution $\sigma_{s}$ on the Lie algebra $\mathfrak{g}^{c}$  given by
\begin{eqnarray}\sigma_{s}( [{\rm ad } f_{0}]^{n}e_{i})=(-1)^{n+1}[{\rm ad } f_{0}]^{n}e_{i}.\label{involution}\end{eqnarray}
This involution is a Lie algebra automorphism and is uniquely defined by  
\[ \sigma_{s}(e_{i})=-e_{i}~{\rm~and~}~ \sigma_{s}(f_{0})=-f_{0}.\]

\end{definition}

Note that 
as $e_{1}=e_{0}$ we have that $({\rm ad } f_{0})e_{1}=h_{0}$ and hence
$\sigma_{s}(h_{0})=h_{0}$. 
The fixed points of $\sigma_{s}$ are described as follows:

\begin{proposition}[\cite{N5}] Let $\mathfrak{g}^{c}$ be  a complex simple Lie algebra with principal three dimensional subalgebra $\mathfrak{s}_{0}=<h_{0},e_{0},f_{0}>_{\mathbb{C}}$, and $e_{i}$ the corresponding highest weight vectors. Then,  the fixed point set of the involution $\sigma_{s}$ defined in (\ref{invoh})
is the complexification of a maximal compact subalgebra of the split real  form $\mathfrak{g}$ of $\mathfrak{g}^{c}$, as defined in Definition \ref{realformsdef}.\label{N5prop}
\end{proposition}

\begin{proof} For completeness, we sketch here the proof given in \cite[Section 6]{N5} of this proposition, and refer to interesting results from \cite{kos}. Let us consider the usual compact real form $\rho$ as defined in (\ref{compactstructure}), which is the antilinear extension of 
\[\rho(x_{\alpha})=x_{-\alpha},~ {~}~ \rho(\alpha)=-\alpha.\]
One can see that the principal three dimensional subalgebra $\tilde{\mathfrak{s}}_{0}=<\psi,e_{0},f_{0}>$ from Remark \ref{segprin} is invariant under $\rho$.
Furthermore, $\tilde{\mathfrak{s}}_{0}$ is real with respect to $\rho$, and so are the subspaces $V_{i}$. The involution $\sigma_{s}$, which acts as  $\pm 1$ on $V_{i}$, preserves the real structure and thus 
\[\sigma_{s} \rho=\rho\sigma_{s} :=\tau.\]
Then, the map $\tau$ is an antilinear involution of $\mathfrak{g}^{c}$ which defines another real structure. Moreover, from this construction, one can see that the involution $\sigma_{s}$ considered here coincides with the involution $\sigma$ introduced in Chapter \ref{ch:real} associated to split real forms.

 We shall understand the fixed points of $\sigma_{s}$ by looking at the real structure $\tau$. Firstly we will show that $\tau$ is a split real form of $\mathfrak{g}^{c}$ and then we will prove the proposition by looking at the fixed points of $\tau$.

From Definition \ref{compactsplit}, in order to show that $\tau$  is a split real form we need to define an appropriate invariant Cartan subalgebra on which the Killing form is positive definite. For this, we consider the element $z\in \mathfrak{g}^{c}$ given by
\[z_{0}=e_{0}+x_{-\psi}.\]

By \cite[Lemma 6.4B]{kos} the centralizer $\mathfrak{g}^{z_{0}}={\rm ker}({\rm ad}z_{0})$ is a Cartan subalgebra $\mathfrak{h}'$ of $\mathfrak{g}^{c}$. Let $y\in \mathfrak{g}^{c}$ such that $[y,z_{0}]=0$. Note that $\rho(z_{0})=f_{0}+x_{\psi}$, and as $\psi$ is the highest root, \begin{small}
\begin{eqnarray}
 [z_{0},\rho(z_{0})]&=&[e_{0}+x_{-\psi},f_{0}+x_{\psi} ]\nonumber\\
&=&[e_{0},f_{0}]+[x_{-\psi},x_{\psi}]\nonumber\\
&=& \psi-\psi\nonumber\\
&=& 0.\nonumber
\end{eqnarray}
\end{small}
Then,  $\mathfrak{h}'$ is also the centralizer of $\rho(z_{0})$ and thus is  preserved by the compact real structure $\rho$. 
For $y\in \mathfrak{h}'$ one has $[z_{0},y]=0$. Recalling that $\psi$ is the highest root,  one can show that
\[y=u+v,\] 
where $u$ commutes with $e_{0}$ and $v$ commutes with $f_{0}$. By definition, the involution $\sigma_{s}$ acts as $-1$ on the centralizer of $e_{0}$ and $f_{0}$, and hence acts as $-1$ on  $\mathfrak{h}'$. The Cartan subalgebra $\mathfrak{h}'$ is preserved by $\tau$, and for $y\in \mathfrak{h}'$ fixed by $\rho$ we have that
\[\tau(iy)=\sigma_{s}\rho(iy)=\sigma_{s}(-iy)=iy.\]
Recall that the Killing form restricted to the compact real form $\rho$  is negative definite. Hence,   on the fixed points of $\tau$, the Killing form is positive definite and thus the real form $\tau$ is a split real form. 

Finally, in order to understand the fixed point set of $\sigma_{s}$, we shall treat $\mathfrak{g}^{c}$ as a real Lie algebra, and decompose it into the eigenspaces of $\sigma_{s}$ and $\rho$:
\begin{eqnarray}
 \mathfrak{g}^{c}=\mathfrak{g}^{+}_{+}\oplus\mathfrak{g}^{+}_{-}\oplus\mathfrak{g}^{-}_{+}\oplus\mathfrak{g}^{-}_{-},\label{desc1}
\end{eqnarray}
where  the lower index gives the sign of the $\sigma_{s}$-eigenvalue, and the upper index gives the sign of the $\rho$-eigenvalue. With this notation, the fixed point set of $\sigma_{s}$ is
\begin{eqnarray}\mathfrak{g}^{\sigma_{s}}=\mathfrak{g}^{+}_{+}\oplus \mathfrak{g}^{-}_{+},\end{eqnarray}
and  the fixed point set of  $\tau=\rho\sigma_{s}$ is
\begin{eqnarray}\mathfrak{g} =\mathfrak{g}^{+}_{+}\oplus\mathfrak{g}^{-}_{-}.\label{desc2}\end{eqnarray}
By definition the Killing form is negative definite on $\mathfrak{g}^{-}_{-}$ and positive definite on $\mathfrak{g}^{+}_{+}$. Hence, $\mathfrak{g}^{+}_{+}$ is the maximal compact subalgebra of $\mathfrak{g}$. Furthermore, $\mathfrak{g}^{+}_{+}$  is a compact form of $\mathfrak{g}^{\sigma_{s}}$ defined by the restriction of $\rho$ to the fixed point set of $\sigma_{s}$. Then, the complex Lie algebra $\mathfrak{g}^{\sigma_{s}}$ can be seen as the complexification of the maximal compact subalgebra $\mathfrak{g}_{+}^{+}$ of the split real form $\mathfrak{g}^{r}$ as required.
\end{proof}

\begin{remark}Note that in the case of classical Lie algebras, the involution $\sigma_{s}$ considered in Proposition \ref{N5prop} corresponds to the involution $\sigma$ introduced in Chapter \ref{ch:real} for the non-compact real forms $\mathfrak{sl}(n,\mathbb{R})$, $\mathfrak{so}(p,p)$, $\mathfrak{so}(p,p+1)$ and $\mathfrak{sp}(2n,\mathbb{R})$.
 \end{remark}

In the proof of Proposition \ref{N5prop} we considered the decomposition  
\[ \mathfrak{g}^{c}=\mathfrak{g}^{+}_{+}\oplus\mathfrak{g}^{+}_{-}\oplus\mathfrak{g}^{-}_{+}\oplus\mathfrak{g}^{-}_{-},\]
where  the upper index corresponds to the $\pm 1$ eigenspace of the compact antilinear-involution of $\mathfrak{g}^{c}$, and the lower index corresponds to the $\pm 1$ eigenspace of $\sigma_{s}$.
 With this notation, the compact real form $\mathfrak{u}$ considered in Chapter \ref{ch:real} is given by
\begin{eqnarray}
 \mathfrak{u}&=&\mathfrak{g}^{+}_{+}\oplus\mathfrak{g}^{+}_{-},  
\end{eqnarray}
In particular,   $i\mathfrak{g}_{+}^{+}=\mathfrak{g}_{+}^{-}$ and $i\mathfrak{g}_{-}^{+}=\mathfrak{g}_{-}^{-}$.

\subsection{The Kostant slice}

The highest weight vectors $e_{i}$ of $V_{i}$  introduced in (\ref{hvec}) generate the following  affine subspace of $\mathfrak{g}^{c}$:

\begin{definition}\label{defKos}
We call the {\rm Kostant slice} $\mathcal{K}$ of $\mathfrak{g}^{c}$ the space given by
\begin{eqnarray}\mathcal{K}:=\{f\in \mathfrak{g}^{c}~|~ f= f_{0} +a_{1}e_{1}+ a_{2}e_{2}+\ldots+ a_{r}e_{r}\},\label{decf}\end{eqnarray}
for $e_{i}$ the highest weight vectors of $\mathfrak{g}^{c}$ and $a_{i}$  complex numbers, and $f_{0}$ as at Definition \ref{defF}. 
\end{definition}

\begin{remark}
 Note that by Definition \ref{invoh}, the involution $\sigma_{s}$ acts as $-1$ on elements in the Kostant slice, and hence the involution $-\sigma_{s}$ preserves $\mathcal{K}$. \label{presK}
\end{remark}

Denote by $\mathcal{G}$ the ring of invariant polynomials in $\mathfrak{g}^{c}$ and let
\begin{eqnarray}
 \{p_{1}, \ldots, p_{r}\},\label{polba}
\end{eqnarray}
be a choice of a homogeneous basis for this ring, where ${\rm deg}(p_{1})\leq\ldots\leq {\rm deg}(p_{r}).$ 
From \cite{cartan} and \cite{ch}, we may choose the basis such that ${\rm deg}(p_{i})=k_{i}+1$, and from \cite[Theorem 7]{kos2} the choice can be done such that if $f\in \mathcal{K}$, then
$p_{i}(f)=a_{i}.$ 
 For $f\in \mathfrak{g}^{c}$,   denote by  $\mathfrak{g}_{f}$ the orbits of $f$ of the adjoint action of $G$ on $\mathfrak{g}^{c}$. The set
\[\mathcal{D}=\bigcup_{f\in \mathcal{K}} \mathfrak{g}_{f}\]
is a dense set in $\mathfrak{g}^{c}$. 

\begin{proposition}
 The involution $-\sigma_{s}$  on $\mathfrak{g}^{c}$ induced from Definition \ref{invoh} acts trivially on the ring of invariant polynomials of $\mathfrak{g}^{c}$.\label{menos}
\end{proposition}

  \begin{proof} Let us consider the  basis $\{p_{1},\ldots,p_{r}\}$ of invariant polynomials as defined previously. 
 As the basis of invariant polynomials is invariant under conjugation, we have that for $f\in\mathcal{K}$ and $X\in \mathfrak{g}_{f}$, the involution $\sigma_{s}$ acts as follows:
\begin{eqnarray}
\begin{array}
 {rclcl}
 (-\sigma_{s}^{*}p_{i})(X)&=& (-\sigma_{s}^{*})p_{i}(f)&~&{~\rm~since~}X\in \mathfrak{g}_{f}\\
&=& p_{i}(-\sigma_{s}(f))&~&{~\rm~by~definition~of}~\sigma_{s}\\
&=&p_{i}(f)&~& {~\rm ~by~Remark~\ref{presK},}{~\rm~since~}f\in\mathcal{K}\\
&=&p_{i}(X)&~&{~\rm~since~}X\in \mathfrak{g}_{f}.\\
\end{array}\nonumber
\end{eqnarray}
As these equalities hold for any $X$,  by continuity the ring of invariant polynomials is acted on trivially by $-\sigma_{s}$. 
\end{proof}

\begin{remark}
 Note that Proposition \ref{menos} agrees with Section \ref{geocon}, since the involution $\sigma_{s}$ as defined in Definition \ref{invoh} corresponds to the involution $\sigma$ associated to the split real forms of $\mathfrak{sl}(n,\mathbb{C})$, $\mathfrak{so}(n,\mathbb{C})$ and $\mathfrak{sp}(2n,\mathbb{C})$.
\end{remark}

\section{The Hitchin fibration}\label{hsplit}

As  seen in Chapter \ref{ch:real}, the involution $\sigma_{s}$ associated to a split real form $\mathfrak{g}$ of $\mathfrak{g}^{c}$ induces an involution on the moduli space $\mathcal{M}_{G^{c}}$ of $G^{c}$-Higgs bundles, whose action on a Higgs bundle $(E,\Phi)$ is given by
\[\Theta:~(E,\Phi)\mapsto (\sigma_{s}(E),-\sigma_{s}(\Phi)).\] 

By means of the explicit description of $\sigma_{s}$ given in the previous section, we shall give an interpretation of Proposition \ref{menos} from the point of view of Higgs bundles.

\subsection{The Teichm\"{u}ller component}
 We shall describe here the construction of the  Teichm\"{u}ller component following \cite{N5}. Consider the adjoint group  ${\rm Ad}G^{c}$ of a complex simple Lie group $G^{c}$, and $P$ a principal ${\rm Ad}G^{c}$-bundle. We let $\mathfrak{g}^{c}$ be the corresponding complex simple Lie algebra, and $\mathfrak{s}_{0}=<h_{0},e_{0},f_{0}>_{\mathbb{C}}$ be its principal three dimensional subalgebra as at (\ref{Flabel3}).

Recall that for $\mathcal{M}_{{\rm Ad}G^{c}}$ the moduli space of ${\rm Ad}G^{c}$-Higgs bundles,  the Hitchin fibration is given by the proper map
\begin{eqnarray}
h: \mathcal{M}_{{\rm Ad}G^{c}}&\rightarrow& \mathcal{A}_{{\rm Ad}G^{c}}:=\bigoplus H^{0} (\Sigma, K^{k_{i}+1})\\
  (E,\Phi)&\mapsto& (p_{1}(\Phi),p_{2}(\Phi), \ldots,p_{r}(\Phi) ),
\end{eqnarray}
where $\{p_{1},\ldots, p_{r}\}$ is a basis of invariant polynomials on $\mathfrak{g}^{c}$ as defined in (\ref{polba}). 
The  \textit{Teichm\"{u}ller component} \cite{N5} is induced by the Kostant slice introduced in Definition \ref{defKos}, and is constructed as follows.
From the decomposition of $\mathfrak{g}^{c}$ into eigenspaces of ${\rm ad}h_{0}$,   consider the vector bundle  
\begin{eqnarray}
 E= {\rm ad}P\otimes \mathbb{C}=\bigoplus_{j=-k_{r}}^{k_{r}}\mathfrak{g}_{j}\otimes K^{j},
\end{eqnarray}
where $k_{r}={\rm max}\{k_{i}\}_{i}$ is the maximal exponent of $\mathfrak{g}^{c}$ (see (\ref{maxexp})).  This is the adjoint bundle of $\mathfrak{g}^{c}$ associated to the principal ${\rm Ad}G^{c}$-bundle $P=P_{1}\times_{i} G$, where $P_{1}$ is the holomorphic principal $SL(2,\mathbb{C})$-bundle associated to $K^{-1/2}\oplus K^{1/2}$, and $i:SL(2,\mathbb{C})\hookrightarrow {\rm Ad}G^{c}$ is the inclusion corresponding to the principal three dimensional subalgebra $\mathfrak{s}_{0}$. Although $P_{1}$ involves a choice of square root $K^{1/2}$, the bundle $E$ is independent of it.  The ${\rm Ad}G^{c}$-Higgs bundle $(E,\Phi)$, for
\begin{eqnarray}\Phi= f_{0} +a_{1}e_{1}+ a_{2}e_{2}+\ldots+ a_{r}e_{r}\label{decp}\end{eqnarray}
and $a_{i}\in H^{0}(\Sigma,K^{k_{i}+1})$, is stable. Note that $({\rm ad}h_{0})f_{0}=-f_{0}$ and thus we may regard $f_{0}$ as a section of $(\mathfrak{g}_{-1}\otimes K^{-1})\otimes K$. Furthermore, the highest weight vectors $e_{j}\in \mathfrak{g}_{k_{j}}$ and thus $a_{j}e_{j}$ is a section of $\mathfrak{g}_{k_{j}}\otimes K^{k_{j}+1}$, making $\Phi$ a well defined holomorphic section of $E\otimes K$. 
As $p_{i}(\Phi)=a_{i}$, the above construction defines a section $s: (a_{1},\ldots,a_{r})\mapsto (E,\Phi)$ of $\mathcal{M}_{{\rm Ad}G^{c}}$ whose image is the \textit{Teichm\"{u}ller component}. This component defines an origin in the smooth fibres of $h$, and thus one has a fibration of abelian varieties.

\subsection{Principal Higgs bundles and split real forms}

The natural involution $\sigma_{s}$ on $E={\rm ad}P\otimes \mathbb{C}$ given in Definition \ref{invoh} acts on Higgs fields of the form of (\ref{decp}) as $\sigma_{s}(\Phi)=-\Phi.$
Hence, the involution $-\sigma_{s}$ fixes these Higgs bundles, and thus preserves the Teichm\"uller component.
From Chapter \ref{ch:real}, the involution $\sigma_{s}$ induces an involution $\Theta$ on $\mathcal{M}_{{\rm Ad}G^{c}}$ which acts as
\[\Theta:~(E,\Phi)\mapsto(\sigma_{s}(E),-\sigma_{s}(\Phi)).\]
The fixed point sets in $\mathcal{M}_{{\rm Ad}G^{c}}$ of the involution $\Theta$ correspond to the moduli space of reductive representations of $\pi_{1}(\Sigma)$ into the split real form of ${\rm Ad}G^{c}$. The following result (\cite[Theorem 7.5]{N5}) relates the Teichm\"{u}ller component to the space of representations of the split real form of ${\rm Ad}G^{c}$:

\begin{theorem} 
 The section $s$ of $\mathcal{M}_{{\rm Ad}G^{c}}$ defines a smooth connected component of the moduli space of reductive representations of $\pi_{1}(\Sigma)$ into the split real form of ${\rm Ad}G^{c}$. 
\end{theorem}

 From Proposition \ref{menos} the involution $-\sigma_{s}$ acts trivially on the ring of invariant polynomials on $\mathfrak{g}^{c}$, and thus $\Theta$ preserves each fibre of the Hitchin fibration.  In order to study the fixed point set of $\Theta$, we note that the involution $\Theta$ acts trivially on the Teichm\"{u}ller component of $\mathcal{M}_{{\rm Ad}G^{c}}$. 
As mentioned before, by choosing a square root $K^{1/2}$, we obtain a choice of Teichm\"{u}ller component which makes the fibres be abelian varieties. Moreover, since $G^{c}$ gives a covering of ${\rm Ad}G^{c}$, when considering the moduli space $\mathcal{M}_{G^{c}}$, there could be more than one lift of the Teichm\"uller component depending on whether the maximal three dimensional subgroup is $SL(2,\mathbb{R})$ or $SO(3)$. Thus, by choosing one of these components we obtain an origin in the fibres of the Hitchin fibration
$\mathcal{M}_{G^{c}}\rightarrow \mathcal{A}_{G^{c}}.$

Recall that the fixed point sets of $\Theta$ induced by $\sigma_{s}$ give the subspace of Higgs bundles corresponding to the split real form of $G^{c}$. The reader should refer to \cite[Theorem 10.2]{N5} for an example of the analysis of the Teichm\"uller components in the case of $PSL(n,\mathbb{R})$.

\begin{theorem}\label{teo:split}
The intersection of the subspace of the Higgs bundle moduli space $\mathcal{M}_{G^{c}}$ corresponding to the split real form of $\mathfrak{g}^{c}$  with the smooth fibres of the Hitchin fibration
\[h:~\mathcal{M}_{G^{c}}\rightarrow \mathcal{A}_{G^{c}},\]
 is given by the elements of order two in those fibres.
\end{theorem}
\begin{proof}
An infinitesimal deformation of a Higgs bundle $(E,\Phi)$ is given by $(\dot A,\dot \Phi)$, where $\dot A \in \Omega^{01}({\rm End}_{0} E)$ and $\dot\Phi\in \Omega^{10}({\rm End}_0 E)$.
The holomorphic involution $\Theta$ on $\mathcal{M}_{G^{c}}$ induces an involution on the tangent space $\mathcal{T}$ of $\mathcal{M}_{G^{c}}$ at a fixed point of $\Theta$. 
Moreover, as seen in (\ref{2.1})  there is a natural symplectic form $\omega$ defined on the infinitesimal deformations by
\begin{equation}
 \omega((\dot{A}_{1},\dot{\Phi}_{1}),(\dot{A}_{2},\dot{\Phi}_{2}))=\int_{\Sigma}{\rm tr}(\dot{A}_{1}\dot{\Phi}_{2}-\dot{A}_{2}\dot{\Phi}_{1}).
\end{equation}
As the trace is invariant under $\Theta$ and $\sigma_{s}(\Phi_{i})=-\Phi_{i}$, the induced involution on the tangent space maps $\omega \mapsto -\omega $.
It follows that the $\pm 1$-eigenspaces $\mathcal{T}_{\pm}$ of this involution are isotropic and complementary, and hence Lagrangian. 
Let us denote by $Dh$ the derivative of  \begin{small}
\[h:\mathcal{M}_{G^{c}}\rightarrow \mathcal{A}_{G^{c}},\]\end{small}
which maps the tangent spaces of $\mathcal{M}_{G^{c}}$ to the tangent space of the base $\mathcal{A}_{G^{c}}$. As the map $h$ is invariant under the involution $\Theta$, the eigenspace $\mathcal{T}_{-}$ is contained in the kernel of $Dh$. Since the derivative is surjective at a regular point, its kernel has dimension ${\rm dim}(\mathcal{M})/2$ and thus it equals $\mathcal{T}_{-}.$ Then, $Dh$ is an isomorphism from $\mathcal{T}_{+}$ to the tangent space of the base. Since $h$ is a proper submersion on the fixed point set, it defines a covering space. The tangent space to the identity in the fibres is acted as $-1$ by the involution $\Theta$ and as the fibres are connected \cite{donagi0}, by exponentiation, the action of $\Theta$ on the regular fibres corresponds to $x\mapsto -x$. Hence, the points of order two in the fibres   $\M$ of $\mathcal{M}_{G^{c}}$ over the regular locus $\A$ correspond to  Higgs bundles for the associated split real form, i.e., to fixed points of $\Theta$.
\end{proof}

For $\sigma_{s}$ a Lie algebra involution associated to the split real form of a complex Lie group $G^{c}$, the fixed points of the involution
$\Theta:(E,\Phi)\mapsto (\sigma_{s}(E),-\sigma_{s}(\Phi))$
on $\mathcal{M}_{G^{c}}$ give a covering of the regular locus of  $\mathcal{A}_{reg}$. Such a covering is defined by the action of $\pi_{1}(\mathcal{A}_{reg})$, giving a permutation on the fibres.
In the case of the split real form $SL(n,\mathbb{R})$ of $SL(n,\mathbb{C})$, the fixed points of $\Theta$ in $\mathcal{M}_{SL(n,\mathbb{C})}$ are given by points of order two in the Prym varieties. In the following chapter we shall study the particular case of $SL(2,\mathbb{R})$-Higgs bundles via the action of $\pi_{1}(\mathcal{A})$.




\chapter{\texorpdfstring{Monodromy of the $SL(2,\mathbb{R})$ Hitchin fibration}{Monodromy for the SL(2,R)} Hitchin fibration}\label{ch:monodromy}

In Chapter \ref{ch:split} we gave a description of Higgs bundles associated to a split real form $G$ of a complex simple Lie group $G^{c}$ in terms of the points of order two in the regular fibres of the Hitchin fibration of  $G^{c}$-Higgs bundles. It is thus natural to consider the monodromy action in order to study the moduli space of $G$-Higgs bundles for a split real form. 

In this Chapter we shall focus on the split real form $SL(2,\mathbb{R})$ of $SL(2,\mathbb{C})$, and the associated Higgs bundles. For this in Section \ref{SLC} we introduce  the Hitchin fibration for  $SL(2,\mathbb{C})$-Higgs bundles, expanding the descriptions given in Chapter \ref{ch:real}. Then, in Section \ref{sec:mono} we define the holonomy homomorphism and review   Copeland's analysis of the monodromy action   \cite{cope1}. Finally, in Section \ref{SLR} we give an explicit description of the monodromy action for the $SL(2,\mathbb{R})$ Hitchin fibration. The results from this Chapter have been accepted for publication in \cite{Lau}, and an application of them is given in Chapter \ref{ch:applications}.

\section{\texorpdfstring{The $SL(2,\mathbb{C})$ Hitchin fibration}{The SL(2,C) Hitchin fibration}}\label{SLC}

We  begin by describing the Hitchin fibration for $SL(2,\mathbb{C})$, and then introduce the work of Copeland \cite{cope1}, in which the monodromy of the Hitchin fibration is analysed from a graph theoretic point of view.

As in previous chapters, let $\Sigma$ be a Riemann surface of genus $g\geq 3$, and let $K$ be its canonical bundle. We have seen in Chapter \ref{ch:real} that an $SL(2,\mathbb{C})$-Higgs bundle as defined by  Hitchin \cite{N1} and  Simpson \cite{simpson} is given by a pair $(E,\Phi)$, where $E$ is a rank 2 holomorphic vector bundle with ${\rm det}(E)=\mathcal{O}_{\Sigma}$ and the Higgs field $\Phi$ is a section of ${\rm End}_{0}(E)\otimes K$, where ${\rm End}_{0}(E)$ denotes the bundle of traceless endomorphisms of $E$.

For simplicity, we shall drop the subscript of $\mathcal{M}_{SL(2,\mathbb{C})}$ and denote by  $\mathcal{M}$   the moduli space of $S$-equivalence classes of semistable $SL(2,\mathbb{C})$-Higgs bundles. Considering the map $\Phi \mapsto {\rm det}(\Phi)$, one has the Hitchin fibration \cite{N1}
\begin{eqnarray}h:\mathcal{M}&\rightarrow& \mathcal{A}:=H^{0}(\Sigma,K^{2}),\nonumber
\end{eqnarray}
where $\mathcal{A}$ is the Hitchin base, introduced in Chapter \ref{ch:complex} as $\mathcal{A}_{SL(2,\mathbb{C})}$.
The moduli space $\mathcal{M}$ is homeomorphic to the moduli space of reductive representations of the fundamental group of $\Sigma$ in $SL(2,\mathbb{C})$ via non-abelian Hodge theory \cite{cor}, \cite{donald}, \cite{N1}, \cite{simpson}.

\subsection{The regular fibres of the Hitchin fibration }

\label{sec1}

  From \cite[Theorem 8.1]{N1} the Hitchin map $h$ is proper and surjective, and its regular fibres are abelian varieties. Moreover, for any $a\in \mathcal{A}-\{0\}$ the fibre $\mathcal{M}_{a}$ is connected  \cite[Theorem 8.1]{goand}.  For any isomorphism class of  $(E,\Phi)$ in $\mathcal{M}$, one may consider the zero set of its characteristic polynomial 
\[{\rm det}(\Phi-\eta I)= \eta^{2}+a=0,\]
where $a={\rm det}(\Phi)\in \mathcal{A}$. This defines a spectral curve $\rho:S\rightarrow \Sigma$ in the total space $X$ on $K$, for $\eta$  the tautological section of the pull back of $K$ on $X$. We shall denote by $\M$ the regular fibres of the Hitchin map $h$, and let $\A$ be the regular locus of the base, which is given by quadratic differentials with simple zeros. Note that the curve $S$ is non-singular over the regular locus $\A$, and the ramification points are given by the intersection of $S$ with the zero section.
The curve $S$ has a  natural involution $\tau(\eta)=-\eta$ and thus we can define the Prym variety ${\rm Prym}(S,\Sigma)$ as the set of line bundles $M\in {\rm Jac}(S)$ which satisfy
\[\tau^{*} M\cong M^{*}.\]
In particular, this definition is consistent with the one given in Chapter \ref{ch:complex} by means of the Norm map. Furthermore, as seen in Chapter \ref{ch:complex}, from \cite[Theorem 8.1]{N1}  the regular fibres of $\mathcal{M}$ are isomorphic to  ${\rm Prym}(S,\Sigma)$.

\subsection{\texorpdfstring{The involution $\Theta$}{The involution theta}}

Recall from Proposition \ref{invosln} that the  involution on $SL(2,\mathbb{C})$ corresponding to the real form $SL(2,\mathbb{R})$ defines an antiholomorphic involution on the moduli space of representations which, in the Higgs bundle complex structure, is the holomorphic involution 
$$\Theta:(E,\Phi)\mapsto (E,-\Phi),$$ 
whose associated automorphism $f:E\rightarrow E^{*}$ gives a symmetric form on $E$.
In particular, the isomorphism classes of stable Higgs bundles with $\Phi\neq 0$ fixed by the involution $\Theta$ correspond to $SL(2,\mathbb{R})$-Higgs bundles  $(E=V\oplus V^{*}, \Phi)$, where $V$ is a line bundle on $\Sigma$, and the Higgs field $\Phi$ is given by
\[\Phi=\left(\begin{array}
              {cc}
0&\beta\\\gamma&0
             \end{array}
\right),\] 
for $\beta:V^{*}\rightarrow V\otimes K$ and $\gamma:V\rightarrow V^{*}\otimes K$. As mentioned in Example \ref{compactexample}, the isomorphism class of an $SL(2,\mathbb{C})$-Higgs bundle with $\Phi=0$ corresponds to the compact real form.

 From \cite {N1}, the smooth fibres are tori of real dimension $6g-6$. 
Indeed, the Higgs bundles in Example \ref{exa} give a section of the fibration fixed by $\Theta$, which defines the Teichm\"uller component introduced in Chapter \ref{ch:split},  and this allows us to identify each fibre with an abelian variety, in fact, a Prym variety as described in Chapter \ref{ch:complex}.
 Since the involution $\Theta$ leaves invariant ${\rm det}(\Phi)$, it defines an involution on each fibre of the Hitchin fibration. 

As shown in Chapter \ref{ch:split}, the fixed points of the involution $\Theta$ on the regular fibres of the Hitchin fibration $h:\mathcal{M}\rightarrow \mathcal{A}$ are the elements of order 2 in the abelian varieties.  Hence, the points corresponding to $SL(2,\mathbb{R})$-Higgs bundles give a covering space of $\A$. Since a generic fibre of $h$ is given by the abelian variety ${\rm Prym}(S,\Sigma)$, i.e., by the quotient of a complex vector space $V$ by some lattice $\wedge$, one has an associated exact sequence of homology groups
\[\ldots \rightarrow \pi_{1}(\wedge)\rightarrow \pi_{1}(V)\rightarrow \pi_{1}({\rm Prym}(S,\Sigma))\rightarrow \pi_{0}(\wedge)\rightarrow \ldots .\]
Hence, there is a short exact sequence
\[0\rightarrow \pi_{1}({\rm Prym}(S,\Sigma))\rightarrow \wedge \rightarrow 0,\]
from where $\wedge\cong \pi_{1}({\rm Prym}(S,\Sigma))$. Therefore, $\pi_{1}({\rm Prym}(S,\Sigma))$ is an abelian group, i.e., $$\pi_{1}({\rm Prym}(S,\Sigma))\cong H_{1}({\rm Prym}(S,\Sigma),\mathbb{Z}).$$
We shall denote by $P[2]$ the elements of order 2 in ${\rm Prym }(S,\Sigma)$, which are equivalent classes in $V$ of points $x$ such that $2x \in \wedge$.
 Then, $P[2]$ is given by $\frac{1}{2}\wedge$ modulo $\wedge$, and as $\wedge$ is torsion free, \[P[2]\cong \wedge/2\wedge\cong H_{1}({\rm Prym}(S,\Sigma),\mathbb{Z}_{2}).\]
Moreover, $H^{1}({\rm Prym}(S,\Sigma),\mathbb{Z}_{2})\cong {\rm Hom}(H_{1}({\rm Prym}(S,\Sigma),\mathbb{Z}),\mathbb{Z}_{2})$ and thus
\[H^{1}({\rm Prym}(S,\Sigma),\mathbb{Z}_{2})\cong {\rm Hom}(\wedge,\mathbb{Z}_{2})\cong \wedge/2\wedge\cong P[2].\]
Hence, the covering space $P[2]$  is determined by the action of $\pi_{1}(\A)$ on the first cohomology of the fibres with $\mathbb{Z}_{2}$ coefficients. In this Chapter we study this action, and thus obtain information about the moduli space of  $SL(2,\mathbb{R})$ representations of $\pi_{1}(\Sigma)$.

\label{sec2}

\section{The monodromy action}\label{sec:mono}

We shall now consider the monodromy homomorphism on the homologies of $\PriM$. Given a canonical connection on the homologies of the fibres of $\M\rightarrow \A$, the \textit{Gauss-Manin connection}, we define the corresponding monodromy action.

\subsection{Gauss-Manin connection}

Consider a fibration $p:Y \rightarrow B$ which is locally trivial, i.e. for any point $b\in B$ there is an open neighbourhood $U_{b}\in B$ such that $p^{-1}(U_{b})\cong U_{b}\times Y_{b}$ where $Y_{b}$ denotes the fibre at $b$.   The $n$th homologies of the fibres $Y_{b}$ form a locally trivial vector bundle over $B$, which we denote $\mathcal{H}_{n}(B)$. This bundle carries a canonical flat connection, the \textit{Gauss-Manin connection}.

To define this connection we  identify the fibres of $\mathcal{H}_{n}(B)$ at nearby points $b_{1},b_{2}\in B$, i.e. $H_{n}(Y_{b_{1}})$ and $H_{n}(Y_{b_{2}})$. Consider $N\subset B$ a contractible open set which includes $b_{1}$ and $b_{2}$.  The inclusion of the fibres $Y_{b_{1}}\hookrightarrow p^{-1}(N)$ and $Y_{b_{2}}\hookrightarrow p^{-1}(N)$ are homotopy equivalences,  and hence we obtain an isomorphism between the homology of a fibre over a point in a contractible open set  $N$ and $H_n(p^{-1}(N))$: \[H_{n}(Y_{b_{1}})\cong H_{n}(p^{-1}(N))\cong H_{n}(Y_{b_{2}}).\]
This means that the vector bundle $\mathcal{H}_n(B)$ over $B$  has a flat connection, the Gauss-Manin connection.   The \textit{monodromy} of $\mathcal{H}_n(B)$   is the holonomy of this connection, i.e. a homomorphism $\pi_{1}(B)\rightarrow {\rm Aut}(H_{n}(Y))$ as an action of $\pi_{1}(B)$ on $H_{n}(Y)$. 
By applying these results  to the fibration $\M \rightarrow \A$ we have the following:

\begin{proposition}
 The  Gauss-Manin connection on the cohomology of the fibres of $$h:\M\rightarrow \A$$ defines the monodromy action for the Hitchin fibration. As each regular fibre is a torus, the monodromy is generated by the action of $\pi_{1}(\A)$ on $H^{1}( {\rm Prym}(S,\Sigma),\mathbb{Z})$.
\end{proposition}

\subsection{\texorpdfstring{A combinatorial approach to monodromy for $SL(2,\mathbb{C})$}{A combinatorial approach to monodromy for SL(2,C)}}

The generators and relations of the monodromy action for hyperelliptic surfaces in the case of $SL(2,\mathbb{C})$-Higgs bundles were studied from a combinatorial point of view by Copeland in \cite[Theorem 1.1]{cope1}. Furthermore, by \cite[Section 4]{walker} one may extend these results to any compact Riemann surface $\Sigma$. 

We shall denote by $\Sigma^{[n]}$ the configuration space of $n$ unordered points in $\Sigma$. Then, there is a natural map \begin{eqnarray}
                                                                                                                            p: \A \rightarrow \Sigma^{[4g-4]} \label{mapap} 
                                                                                                                            \end{eqnarray}
 which takes a quadratic differential to its zero set. Furthermore, there is a map to the $4g-4$ surface braid group $\pi_{1}(\Sigma^{[4g-4]})$ given by
\begin{eqnarray}
 p_{*}:\pi_{1}(\A)\rightarrow \pi_{1}(\Sigma^{[4g-4]}).\label{mapapp} 
\end{eqnarray}
 As we have seen before, for $\omega \in \A$, we can define a non-singular spectral cover $S\rightarrow \Sigma$ with equation $\eta^{2}+\omega=0$, together with an involution $\tau:S\rightarrow S$.

\begin{proposition}[\cite{cope1}] The kernel of $p_{*}$ acts by $\{1,\tau\}$ on $H_{1}({\rm Prym}(S,\Sigma),\mathbb{Z})$ via the monodromy action.
\end{proposition}

Hence, the interest for the Prym variety is in studying the action of the image $p_{*}\pi_{1}(\A)$ in $\pi_{1}(\Sigma^{[4g-4]})$. By considering the  Abel-Jacobi map ${\rm Ab}:=\Sigma^{[4g-4]}\rightarrow Pic^{4g-4}(\Sigma),$
 and its induced  map from the $4g-4$ surface braid group given by \[Ab:\pi_{1}(\Sigma^{[4g-4]})\rightarrow H_{1}(\Sigma,\mathbb{Z}),\]
one has the following result:
\begin{proposition}[\cite{cope1}] For $\Sigma$  hyperelliptic of genus $g\geq 3$,
\[p_{*}\pi_{1}(\A)={\rm ker} ~Ab:\pi_{1}(\Sigma^{[4g-4]})\rightarrow H_{1}(\Sigma,\mathbb{Z}).\]
\end{proposition}

 By using graph theory, Copeland was able to construct a polyhedral decomposition of $\Sigma$, given by a graph $\check{\Gamma}$ whose vertices are the zeros of a certain quadratic differential $\omega \in \A$. His arguments are based on the dual graph $\Gamma$ of $\check{\Gamma}$ which can be thought of as having marked points in each face of the polyhedral. To apply the following theorem, the dual graph $\Gamma$ is used in \cite{cope1}:
 \begin{proposition}[\cite{cope2}] Let  $\Sigma$ be a polyhedron of genus $g$ with $n$ faces such that no face is a neighbour of itself and no two faces share more than one edge. Then,  $${\rm ker}(Ab:\pi_{1}(\Sigma^{[n]})\rightarrow H_{1}(\Sigma,\mathbb{Z}))$$ is generated by elements $\tilde{\sigma}_{e}$ labelled by each edge $e$. The  basepoint  of $\Sigma^{[n]}$ may be chosen to be a marked point  in the interior of each face, and each $\tilde{\sigma}_{e}$ may be viewed as a transposition of the marked points on the faces it separates. 
\end{proposition}
The action of the element $\tilde{\sigma}_{e}$ can be thought of as interchanging the two vertices of $e$ in the graph $\check{\Gamma}$, dual of the polyhedron $\Gamma$ used in the previous theorem. 
 Furthermore, the elements $\tilde{\sigma}_{e}$ in $p_{*}\pi_{1}(\A)$ can be lifted to $\pi_{1}(\A)$. Based on the above results, Copeland deduces the following theorem where he describes the monodromy action of $\pi_{1}(\A)$ on $H_{1}({\rm Prym}(S,\Sigma),\mathbb{Z})$:

\begin{theorem}[\cite{cope1,walker}]
 To each compact Riemann surface $\Sigma$ of genus greater than 2, one may associate a graph $\check{\Gamma}$ with edge set $E$, and a skew bilinear pairing  $<e , e'>$ on edges  $e,e' \in \mathbb{Z}[E]$ such that
\begin{enumerate}
 \item[(i)] the monodromy representation of $\pi_{1}(\A)$ acting on $H_{1}({\rm Prym}(S,\Sigma),\mathbb{Z})$ is generated by elements $\sigma_{e}$ labelled by the edges $e\in E$,
\item[(ii)] one can define an action of $\pi_{1}(\A)$ on $e'\in \mathbb{Z}[E]$ given by
\[\sigma_{e}(e')=e'-<e',e>e,\]
\item[(iii)] the monodromy representation of the action of $\pi_{1}(\A)$ on $H_{1}({\rm Prym}(S,\Sigma),\mathbb{Z})$ is a quotient of this module $\mathbb{Z}[E]$.
\end{enumerate}\label{copeland}
\end{theorem}

In order to construct the graph $\check{\Gamma}$ Copeland looks at the particular case of $\Sigma$ given by the non-singular compactification of the zero set of $y^{2}=f(x)=x^{2g+2}-1$. 
 The space of quadratic differentials is described as
\begin{eqnarray}H^{0}(\Sigma,K^{2})=H^{0}(\Sigma,K^{2})^{+} \oplus H^{0}(\Sigma,K^{2})^{-},\label{decomp}\end{eqnarray}
where $+$ and $-$ denote the  $+1$ and $-1$ eigenspace of the induced action of the hyperelliptic involution. Hence, any quadratic differential $\omega$ can be expressed as
\begin{eqnarray}
 \omega=\omega^{+} + \omega^{-}= p(x)\left(\frac{dx}{y}\right)^{2}+q(x)y\left(\frac{dx}{y}\right)^{2}, \label{decomp1}
\end{eqnarray}
where the degree of the polynomial  $p(x)$ is at most $2g-2$, and the degree of $q(x)$ is at most $g-3$.  The dimension of the components of $H^{0}(\Sigma,K^{2})$ are  $$\dim(H^{0}(\Sigma,K^{2})^{+})=2g-1 ~~{\rm  and} ~~\dim(H^{0}(\Sigma,K^{2})^{-})=g-2.$$
Copeland firstly considers  $\omega_{0} \in \mathcal{A}$ given by
\[\omega_{0}=(x-2\zeta^{2})(x-2\zeta^{4})(x-2\zeta^{6})(x-2\zeta^{8})\prod_{9\leq j\leq 2g+2}(x-2\zeta^{j})\left(\frac{dx}{y}\right)^{2},\]
for $\zeta=e^{2\pi i /2g+2}$. By means of $\omega_{0}$, it is shown in \cite[Section 7]{cope1} how interchanging two zeros of the differential provides information about the generators of the monodromy. Then, by means of the ramification points of the surface,  a dual graph to $\check{\Gamma}$ for which each zero of $\omega_{0}$ is in a face can be constructed.  Copeland's analysis extends  to any element in $\A$ over a hyperelliptic curve \cite[Section 23]{cope1}. Moreover, by work of  Walker \cite[Section 4]{walker} the above construction can be done for any compact Riemann surface.

\begin{remark} Note that from  equation (\ref{decomp}) and the general expression of an element in $H^{0}(\Sigma,K^{2})$, in the case of $g=2$, any quadratic differential is described entirely by the polynomial $p(x)$ which is of degree  at most $2$. 
\end{remark}
\begin{remark}
 Following  \cite[Section 6]{cope1}, we shall consider the graph $\check{\Gamma}$ whose $4g-4$ vertices are given by the ramification divisor of $\rho:S\rightarrow \Sigma$, i.e., the zeros of $a={\rm det}(\Phi)$.
\end{remark}
 \begin{figure}[htbp]
\centering \epsfig{file=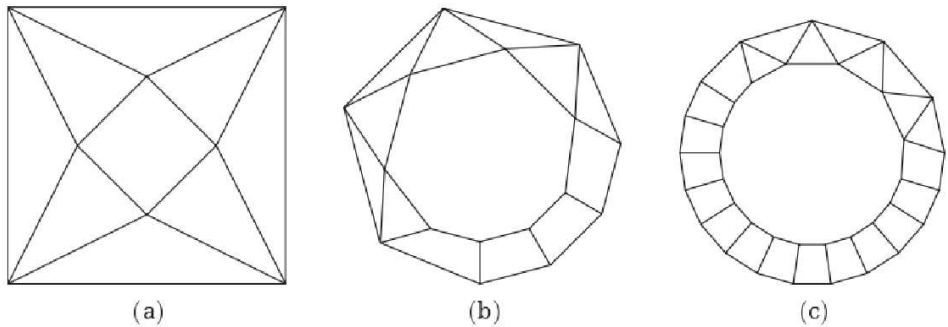, width=13 cm}
\caption{The graph $\check{\Gamma}$ for genus  $g=3,5,$ and $10$ as constructed by Copeland.}
\label{grafo4}
\end{figure}
\newpage

  For  $g> 3$ the graph $\check{\Gamma}$ is given by a ring with 8 triangles next to each other, $2g-6$ quadrilaterals and $4g-4$ vertices.  We shall label its edges as follows:
\begin{figure}[htbp]
\centering \epsfig{file=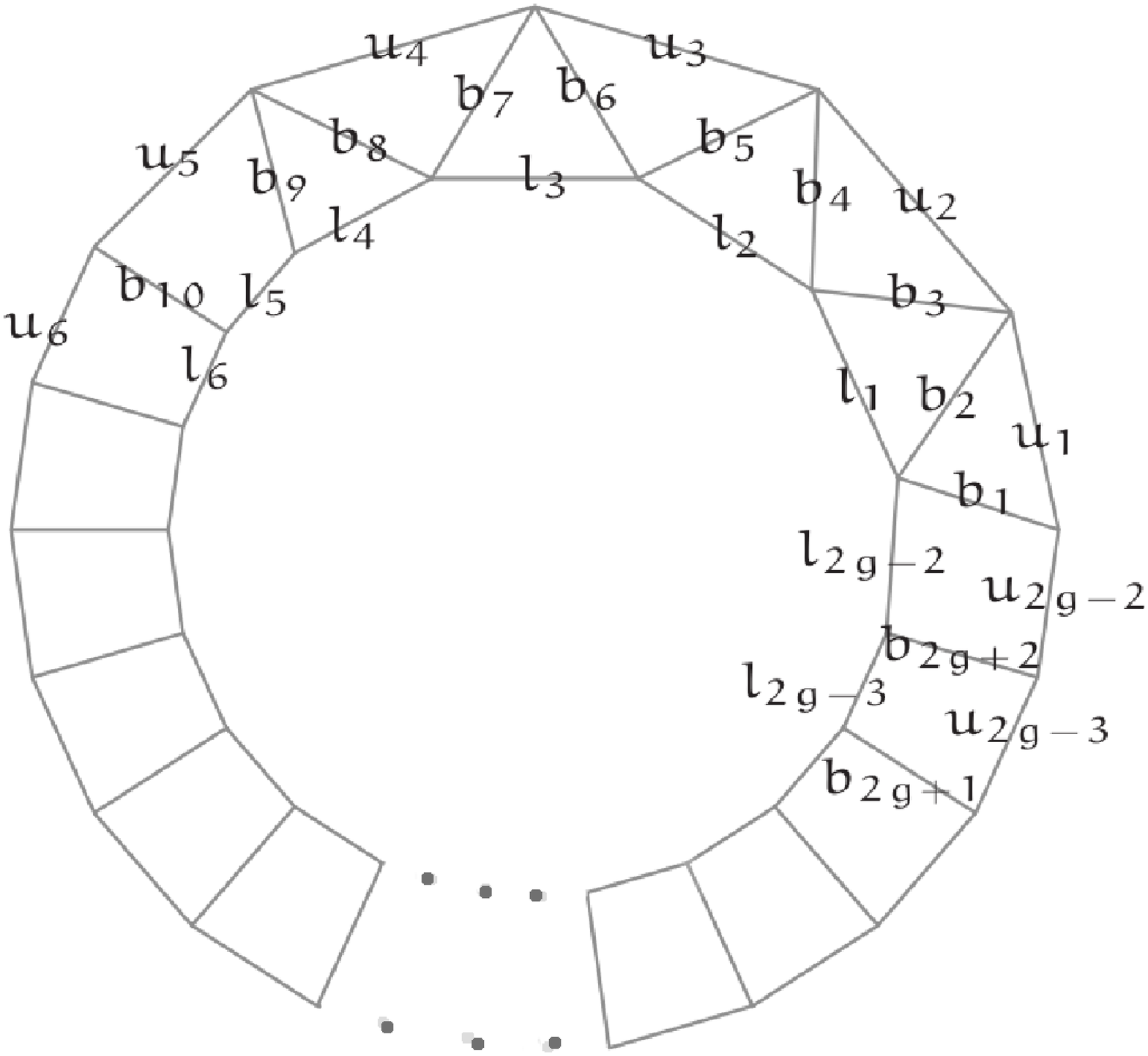, width=7 cm}
\caption{The labels in $\check{\Gamma}$.}
\label{annulus}
\end{figure}

Considering the lifted graph of $\check{\Gamma}$ in the curve $S$ over $\Sigma$, Copeland could show the following:
\begin{proposition}\cite[Theorem 11.1]{cope1}
If $E$ and $F$ are respectively the edge  and face sets of $\check{\Gamma}$, then  there is an induced homeomorphism
\begin{eqnarray}
 {\rm Prym}(S,M)\cong \frac{\mathbb{R}[E]}{\left(\mathbb{R}[F]+\frac{1}{2}\mathbb{Z}[E]\right)},\label{prince}
\end{eqnarray}
where the inclusion $\mathbb{R}[F] \subset \mathbb{R}[E]$ is defined by the following relations involving the boundaries of the faces:
\begin{eqnarray}
 ~\tilde{x}_{1}&:=&\sum_{i=1}^{2g-2}l_{i} ~~~~;~~~ ~\tilde{x}_{2}:=\sum_{i=1}^{2g-2}u_{i}~;\nonumber\\
~ \tilde{x}_{3}&:=&~\sum_{even \geq  6}u_{i}-\sum_{odd \geq 5}l_{i}+\sum_{i=1}^{2g+2}b_{i}~; \nonumber\\
~\tilde{x}_{4}&:=&~l_{1}+l_{3}-u_{2}-u_{4}+\sum_{odd}u_{i}-\sum_{even}l_{i}+\sum_{i=1}^{2g+2}b_{i}.\nonumber
\end{eqnarray}
  \label{principal}\label{prim}
\end{proposition} 
Note that $\mathbb{R}[F] \cap \frac{1}{2}\mathbb{Z}[E]$ can be understood by considering the following sum:
\[\tilde{x}_{3}+\tilde{x}_{4}+\tilde{x}_{1}-\tilde{x}_{2}= 2\left(l_{1}+l_{3}-u_{2}-u_{4}+\sum_{i=1}^{2g+2}b_{i}\right)=2 \tilde{x}_{5}.\]
Although the summands above are not individually in $\mathbb{R}[F] \cap \frac{1}{2}\mathbb{Z}[E]$, when summed they satisfy
\[\tilde{x}_{5}\in \mathbb{R}[F] \cap \frac{1}{2}\mathbb{Z}[E].\]
\begin{remark} 
For $g=2$ it is known that  $\pi_{1}(\A)\cong \mathbb{Z}\times \pi_{1}(S_{6}^{2})$,  where $S_{6}^{2}$ is the sphere $S^{2}$ with $6$ holes (e.g. \cite[Section 6]{cope1}).
\end{remark}

\subsection{\texorpdfstring{The fixed points of $\Theta:(E,\Phi)\mapsto (E,-\Phi)$} {The fixed points of the involution theta}}

Since $S$ is a 2-fold cover in the total space $X$ of $K$, the direct image of the trivial bundle $\mathcal{O}$ in ${\rm Prym}(S,\Sigma)$ is given by $\rho_{*}\mathcal{O}=\mathcal{O}\oplus K^{-1}$ (e.g. see \cite[Remark 3.1]{bobi}). Equivalently, note that $S$ has a natural involution, and hence the sections of $\rho_{*}\mathcal{O}$ can be separated into invariant and anti-invariant ones. As $\mathcal{O}$ corresponds to the invariant sections, and $\Lambda^{2}\rho_{*}\mathcal{O}\cong K^{-1}$, necessarily $\rho_{*}\mathcal{O}=\mathcal{O}\oplus K^{-1}$.

For the line bundle $L=\rho^{*}K^{1/2}$ on $S$, one has 
$$\rho_{*}\rho^{*}K^{1/2}= K^{1/2}\otimes \rho_{*}\mathcal{O}=K^{1/2}\oplus K^{-1/2}.$$
It follows from Section \ref{sec2} that the line bundle $\mathcal{O}\in {\rm Prym}(S,\Sigma)$ has an associated Higgs bundle $(K^{1/2}\oplus K^{-1/2}, \Phi_{a})$, where the Higgs  field $\Phi_{a}$ defined as in Example \ref{exa} by 
\[\Phi_{a}=\left(\begin{array}{cc}
              0&a\\1&0
             \end{array}
\right),~{~\rm~for~}~ a\in H^{2}(\Sigma,K^{2}).\]
The automorphism 
$$\left(
\begin{array}
 {cc}
i&0\\
0&-i
\end{array}
\right)
$$
conjugates $\Phi_{a}$ to $-\Phi_{a}$ and so the equivalence class of the Higgs bundle is fixed by the involution $\Theta$. Thus, this family of Higgs bundles defines an origin in the set of fixed points on each fibre, and gives the Teichm\"uller component as introduced in Chapter \ref{ch:split}.
Furthermore, we have seen in the previous chapter that the points of order two in the fibres of  $\M$ over the regular locus $\A$ correspond to stable $SL(2,\mathbb{R})$ Higgs bundles.

By Proposition \ref{prim} one may write ${\rm Prym}(S,\Sigma)\cong \mathbb{R}^{6g-6}/\wedge$ for
 $$\wedge:=\frac{\frac{1}{2}\mathbb{Z}[E]}{\mathbb{R}[F]\cap\frac{1}{2}\mathbb{Z}[E]}.$$ In particular, one has $\wedge \cong H_{1}({\rm Prym}(S,\Sigma), \mathbb{Z})$. 

The monodromy action on $H^{1}({\rm Prym}(S,\Sigma),\mathbb{Z}_{2})$ is equivalent to the action on $P[2]$, the space of elements of order 2 in ${\rm Prym }(S,\Sigma)$.
 Note that over $\mathbb{Z}_{2}$, the equations for $\tilde{x}_{1},\tilde{x}_{2}, \tilde{x}_{3}, \tilde{x}_{4}$ and  $\tilde{x}_{5}$  are equivalent to
\begin{small}
 \begin{eqnarray}
 ~x_{1}&:=&\sum_{i=1}^{2g-2}l_{i} ~ ;\nonumber\\
~x_{2}&=&\sum_{i=1}^{2g-2}u_{i}~;~\nonumber\\
 ~ x_{3}&:=&~l_{1}+l_{3}+u_{2}+u_{4}+\sum_{even}u_{i}+\sum_{odd}l_{i}+\sum_{i=1}^{2g+2}b_{i};~ ~{~~}~ \label{ceromas}\nonumber\\
~x_{4}&:=&l_{1}+l_{3}+u_{2}+u_{4}+\sum_{odd}u_{i}+\sum_{even}l_{i}+\sum_{i=1}^{2g+2}b_{i}~;~{~~}~\nonumber\\
x_{5}&:=& l_{1}+l_{3}+u_{2}+u_{4}+\sum_{i=1}^{2g+2}b_{i}.\label{cuarta}\nonumber
\end{eqnarray}
\end{small}
\begin{proposition}
 The space $P[2]$ is given by the quotient of $\mathbb{Z}_{2}[E]$ by the subspace generated by $x_{1}, x_{2}, x_{4}$ and $x_{5}$.
\end{proposition}

 The relations   $x_{4},x_{2}$ and $x_{1}$, are represented by the dark edges  in the following graphs:
\begin{figure}[htbp]
\centering \epsfig{file=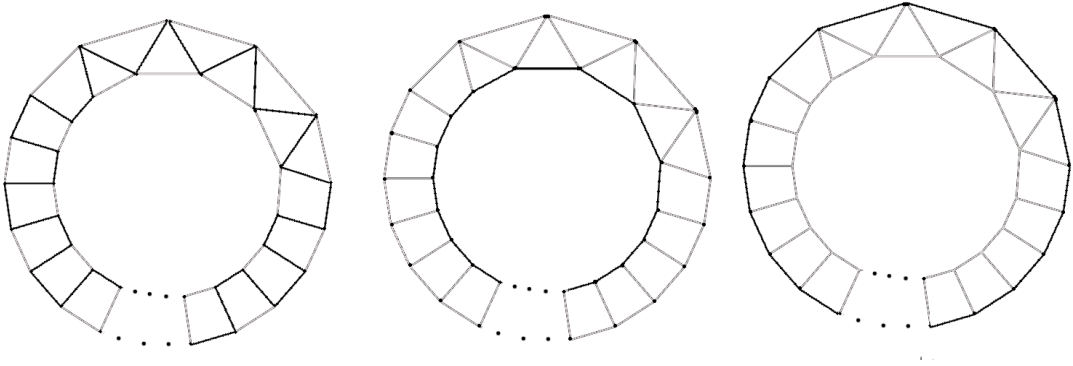, width=15 cm}
\caption{Relations $x_{4}$, $x_{2}$ and $x_{1}$ respectively.}
\label{ relat }
\end{figure}

\section{\texorpdfstring{The $SL(2,\mathbb{R})$ moduli space}{The SL(2,R) moduli space}}\label{SLR}

We shall study here the moduli space of $SL(2,\mathbb{R})$-Higgs bundles via the monodromy action of $\pi_{1}(\A)$ on $P[2]$. For this, we shall first look at the action on $\mathbb{Z}_{2}[E]$, the space spanned by the edge set $E$ with coefficients in $\mathbb{Z}_{2}$. Then, we consider its quotient action on $P[2]$ to calculate the monodromy group.

\subsection{\texorpdfstring{The action on $\mathbb{Z}_{2}[E]$}{The action on Z2[E]}}

It is convenient to consider $\mathbb{Z}_{2}[E]$ as the space  of 1-chains $C_{1}$ for a subdivision of the annulus in Figure \ref{annulus}. The boundary  map $\partial$ to the space $C_{0}$ of 0-chains (spanned by the vertices of $\check{\Gamma}$) is defined   on an edge $e\in C_{1}$ with vertices $v_{1}, ~v_{2}$ as $\partial e= v_{1}+v_{2}$. 
The natural map  $p:\A\rightarrow \Sigma^{[4g-4]}$ introduced in (\ref{mapap}) induces the following maps
\[\pi_{1}(\A) \rightarrow \pi_{1}(\Sigma^{[4g-4]}) \rightarrow S_{4g-4}~,\]
where $S_{4g-4}$ is the symmetric group of $4g-4$ elements. Thus, there is  a natural  permutation action on  $C_{0}$  and Copeland's generators in $\pi_{1}(\A)$ map to transpositions in $S_{4g-4}$. Concretely, these generators are defined as transformations of $\mathbb{Z}_{2}[E]$ as follows. The action $\sigma_{e}$ labelled by the edge $e$ on another edge $x$ is 
\begin{equation}\sigma_{e}(x)= x ~+ <x,e>e,\label{generator}\end{equation}
where $<\cdot,\cdot>$ is the intersection pairing. As this pairing is skew over $\mathbb{Z}$, for any edge $e$ one has $<e,e>=0$.
 Let  $G_{1}$ be the group of transformations of $C_{1}$ generated by $\sigma_{e}$, for $e\in E$.  

\begin{proposition} The group $G_{1}$ acts trivially on $Z_{1}= \ker (\partial:C_{1}\rightarrow C_{0})$.   \label{trivial}
\end{proposition}  

\begin{proof} 
Consider  $a\in C_{1}$ such that $\partial a =0$, i.e., the edges of $a$ have vertices which occur an even number of times. By definition,  $\sigma_{e}\in G_{1}$ acts trivially on $a$ for any edge $e\in E$  non adjacent to $a$.  Furthermore, if $e\in E$ is adjacent to $a$, then $\partial a=0$ implies that an even number of edges in $a$ is adjacent to $e$, and thus the action $\sigma_{e}$ is also trivial on $a$.
\end{proof}

We shall give an ordering to the vertices in $\check{\Gamma}$ as in the figure below, and denote by  $E'\subset E$ the set of dark edges:
\begin{figure}[htbp]
\centering \epsfig{file=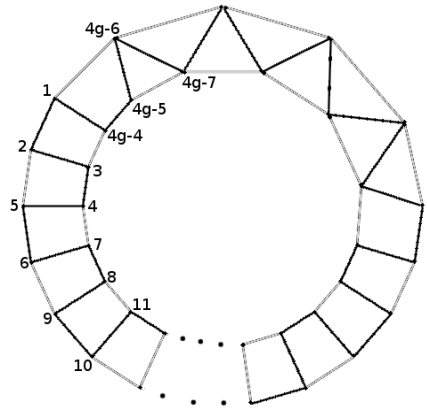, width=6 cm}
\caption{The subset $E'\subset E$.}
\end{figure}
 
For $(i,j)$ the edge between the vertices $i$ and $j$, the set  $E'$ is given by the  natural succession of edges $e_{i}:=(i, i+1)$ for $i=1,\ldots,4g-5$,  together with the edge $e_{4g-4}:=(4g-4,1)$.

\begin{proposition}
 The reflections labelled by the edges in $E' \subset E$ generate a subgroup $S'_{4g-4}$ of $G_{1}$ isomorphic to the symmetric group $S_{4g-4}$. \label{propi1}
\end{proposition}

\begin{proof}
To show this result, one needs to check that the following properties characterising generators of the symmetric group apply to the reflections labelled by $E'$: 
\begin{enumerate}
 \item[(i)] $\sigma_{e_{i}}^{2}=1$ for all $i$,
 \item[(ii)] $\sigma_{e_{i}}\sigma_{e_{j}}=\sigma_{e_{j}}\sigma_{e_{i}}$ if $j\neq i\pm 1 $,
 \item[(iii)] $(\sigma_{e_{i}}\sigma_{e_{i+1}})^{3}=1.$
\end{enumerate}
By equation (\ref{generator}), it is straightforward to see that the properties (i) and (ii) are satisfied by $\sigma_{e_{i}}$ for all $e_{i}\in E$. Indeed, let $e$ be an edge in $E$ and $\sigma_{e}$ be its associated reflection. As 
\begin{eqnarray} 
\sigma_{e}(x)&=& x + <x,e>e,\nonumber
\end{eqnarray}
one has that
\begin{eqnarray} 
 \sigma_{e}^{2}(x)&=& x + <x,e>e+<x,e>e=x. \nonumber
\end{eqnarray}
Similarly, let $e_{i}$ and $e_{j}$ be two edges which do not intersect. Then
\[\sigma_{e_{i}}\sigma_{e_{j}}(x)= \sigma_{e_{i}}(x+<x,e_{j}>e_{j})=x+<x,e_{j}>e_{j}+<x,e_{i}>e_{i},\]
and as $x+<x,e_{j}>e_{j}+<x,e_{i}>e_{i}$ is symmetric in $i,j,$ one has 
$$\sigma_{e_{i}}\sigma_{e_{j}}(x)=\sigma_{e_{j}}\sigma_{e_{i}}(x).$$

In order to check (iii) we shall consider all different options for edges adjacent  to $e_{i}$ and $e_{i+1}$ when  $e_{i},e_{i+1}\in E'$. These can be seen in the following figure:

\begin{figure}[htbp]
\centering \epsfig{file=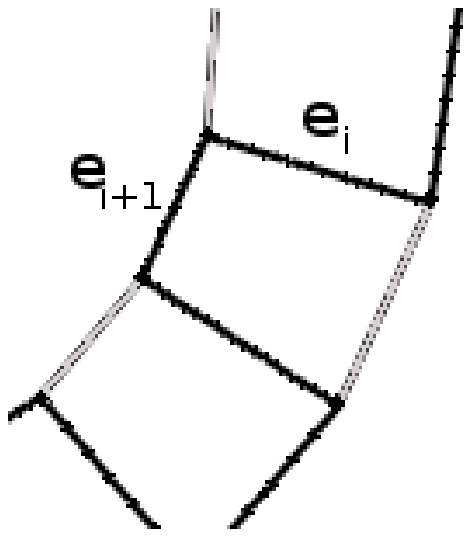, width=3 cm}~~\epsfig{file=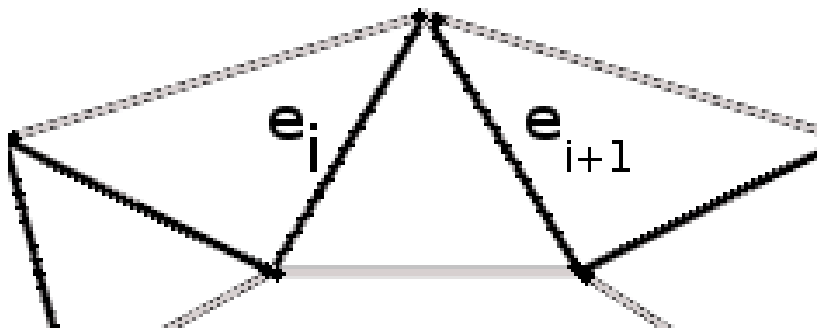, width=4 cm}~~\epsfig{file=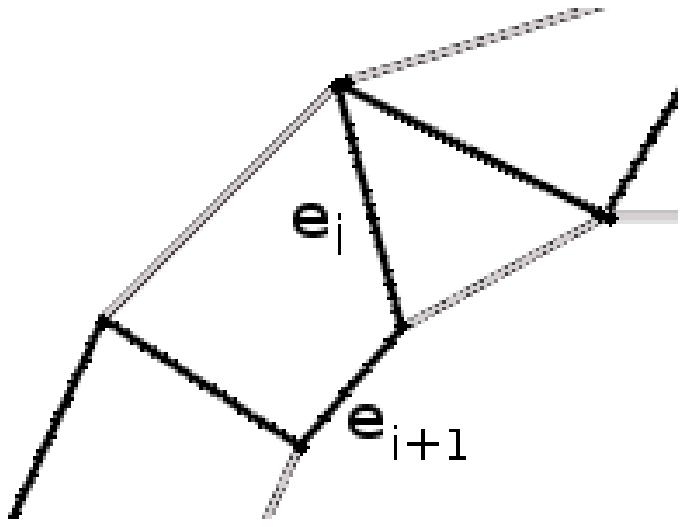, width=3.5 cm}
\caption{5,7 and 6 adjacent edges to $e_{i}e_{i+1}$.}
\label{ }
\end{figure}   

 Let $c_{1},c_{2},\cdots, c_{n}\in E$ be the  edges adjacent to $e_{i}$ and $e_{i+1}$, where $n$ may be $5,6$ or $7$. Taking the basis $\{c_{1}, \cdots, c_{n},e_{i},e_{i+1}\}$, the action  $\sigma_{e_{i}}\sigma_{e_{i+1}}$ is given by  matrices $B$ divided into blocks in the following manner:
\begin{equation}
B:=\left(
\begin{array}{ccc|cc}
   &  &  &   0 & 0 \\
   & I_{n} &   & \vdots & \vdots \\
   &  &  &   0 & 0 \\\hline
 a_{1} &   \cdots   & a_{n} & 0 & 1 \\
 b_{1} &  \cdots     &b_{n}  & 1 & 1
\end{array}
\right),\nonumber
\end{equation}
 where the entries $a_{i}$ and $b_{j}$ are $0$ or $1$, depending on the number of common vertices with the edges and their locations. Over $\mathbb{Z}_{2}$ one has that $B^{3}$ is the identity matrix
and so  property  (iii) is satisfied for all edges $e_{i}\in E'$. \end{proof}

The subgroup $S'_{4g-4}$ preserves $\mathbb{Z}_{2}[E']$. Furthermore, the boundary $\partial: C_{1}\rightarrow C_{0}$ is compatible with the action of $S_{4g-4}'$ on $C_{1}$ and $S_{4g-4}$ on $C_{0}$.
Thus, from  Proposition \ref{propi1}, there is a natural homomorphism  $$\alpha:G_{1}\longrightarrow S_{4g-4}~,$$ which is an isomorphism when restricted to $S'_{4g-4}$. For $N:=\ker \alpha $ one has
\[G_{1}=N \ltimes S'_{4g-4}~.\]

 Any element $g\in G_{1}$ can  be expressed uniquely as $g=s \cdot h$, for $h\in N \subset G_{1}$ and $s \in S'_{4g-4}\subset G_{1}$. The  group action on $C_{1}$ may then be expressed as
\begin{equation}
 (h_{1}s_{1})(h_{2}s_{2})=h_{1}s_{1}h_{2}s_{1}^{-1}s_{1}s_{2}~,\label{conj}
\end{equation}
for $s_{1},s_{2}\in S_{4g-4}'$ and $h_{1},h_{2}\in N$. 
Let  $E_{0}:=E-E'$  and  denote by $\Delta_{e}\in C_{1}$   the boundary of the square or triangle adjacent to the edge $e\in E_{0}$  in $\check{\Gamma}$. In particular, each edge $e\in E_{0}$ is contained in exactly one of such boundaries. 

\begin{figure}[htbp]
\centering \epsfig{file=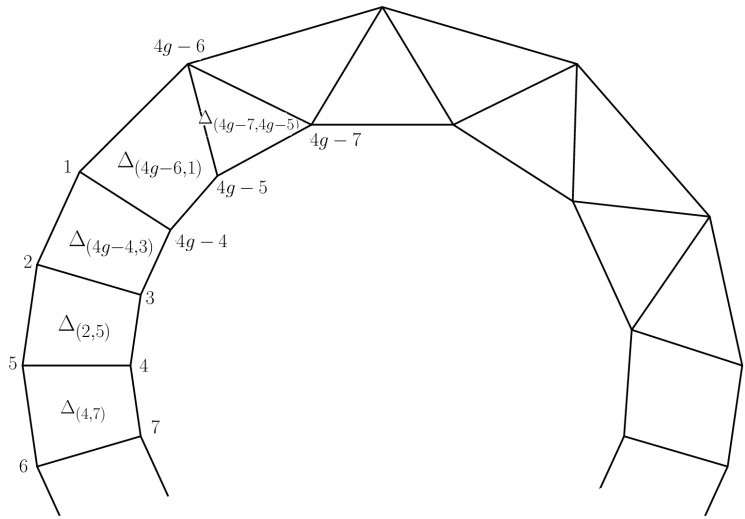, width=10 cm}
\caption{The elements $\Delta_{e}$ in $\check{\Gamma}$.}
\end{figure}
We shall write  $\tilde{E}_{0}:=\{\Delta_{e}~{\rm for}~e\in E_{0}\}$. From Proposition \ref{trivial} the boundaries  $\Delta_{e}$ are acted on trivially by $G_{1}$.
By comparing the action labelled by the  edges in $E_{0}$ with the action of the corresponding transposition in $S'_{4g-4}$ one can find out which elements are in $N$. We shall denote by $\sigma_{(i,j)}$ the action on $C_{1}$ labelled by the edge $(i,j)$.  Furthermore, we consider $s_{(i,j)}$ the element in $S'_{4g-4}$ such that $\alpha(\sigma_{(i,j)})=\alpha(s_{(i,j)})\in S_{4g-4}$, where $\alpha(s_{(i,j)})$ interchanges the vertices of $(i,j)$. Then, we have the following possibilities:
\begin{itemize}
\item For a triangle with vertices  $i,i+1,i+2$, the generators of $S'_{3}\subset S'_{4g-4}$ are  given  by $\sigma_{(i,i+1)}$ and $\sigma_{(i+1,i+2)}$. Then,   $(i,i+1),(i+1,i+2) \in E'$ and the action labelled by  the edge $(i,i+2)\in E_{0}$ can be written as:
\[s_{(i,i+2)}=\sigma_{(i+1,i+2)}\sigma_{(i,i+1)}\sigma_{(i+1,i+2)}~.\]
\item For a square with vertices $i,i+1,i+2, i+3$, the generators of $S'_{4} \subset S'_{4g-4}$ are given by  $\sigma_{(i,i+1)}, \sigma_{(i+1,i+2)}$ and $\sigma_{(i+2,i+3)}$. In this case, the edges $(i,i+1)$,$(i+1,i+2)$, and $(i+2,i+3)$ are in $E'$, and the action labelled by  $(i,i+3)\in E_{0}$ can be written as:
\[s_{(i,i+3)}=\sigma_{(i+2,i+3)}\sigma_{(i+1,i+2)}\sigma_{(i,i+1)}\sigma_{(i+1,i+2)}\sigma_{(i+2,i+3)}~.\]
\end{itemize}
\begin{theorem} The action of  $\sigma_{e}$ labelled by $e\in E$ on $x\in C_{1}$ is given by
\[\sigma_{e} (x) =  h_{e}\cdot s_{e} ~(x)~, \]
where $s_{e}\in S'_{4g-4}$ is the element which maps under $\alpha$ to the transposition of the two vertices of $e$, and  the action of $h_{e} \in N$ is given by
\begin{equation}
  h_{e}(x)=\left\{ \begin{array}{ccc}
x&{\rm if}&e \in E'~,\\
x+<e,x>\Delta_{e}&{\rm if}& e\in E_{0}~.
                \end{array}\right. \label{funh}
\end{equation}
\end{theorem}
\begin{proof}For $e\in E'$, one has $\sigma_{e}\in S'_{4g-4}$. Furthermore, one can see that the action of $\sigma_{e}$ labelled by $e\in E_{0}$ on an adjacent edge $x$ is given by 
\begin{eqnarray}
 \sigma_{e}(x)&=&x+<x,e>e_{i}\nonumber \\
&=&s_{e}(x)+\Delta_{e}.\nonumber
\end{eqnarray}
From Proposition \ref{trivial}, the boundary $\Delta_{e}$ is acted on trivially by $G_{1}$ and thus the above action is given by
\begin{eqnarray}
  h_{e}\cdot  s_{e} (x)
=h_{e}(\sigma_{e}(x)+\Delta_{e})
=h_{e}(x)+<x,e_{i}>e_{i}+\Delta_{e}
=\sigma_{e}(x),\nonumber
\end{eqnarray}
whence proving the proposition.\end{proof}
\begin{example} By definition, the action labelled by the edge $e=(4g-4,3)\in \check{\Gamma}$, is given by 
$\sigma_{e}(i,j)= (i,j)~+ <(i,j),~ e>e.$ 
This is non-trivial only  over adjacent edges. Consider the action on the edge  $x=(4g-4,1)$, which  is
$
\sigma_{e}(x)= x~+ ~e.$ 
This can  be rewritten as
\begin{eqnarray}
\sigma_{e}(x)&=& (1,2)+~(2,3)+~\Delta_{e}.\label{unodos}
\end{eqnarray}
 On the other hand, by acting with elements in $S'_{4g-4}$ we have
\[s_{e}=\sigma_{(2,3)}\sigma_{(1,2)}\sigma_{x}\sigma_{(1,2)}\sigma_{(2,3)},\]
and thus
\begin{eqnarray}
 s_{e}(x)&=& \sigma_{(2,3)}\sigma_{(1,2)}\sigma_{x}\sigma_{(1,2)}(x)\nonumber\\
&=& \sigma_{(2,3)}\sigma_{(1,2)}\sigma_{x}\left(x~+~(1,2)\right)\nonumber \\
&=& \sigma_{(2,3)}\sigma_{(1,2)}(1,2)\nonumber \\
&=& \sigma_{(2,3)}(1,2)\nonumber \\
&=& (1,2) + (2,3).\nonumber
\end{eqnarray}
 Comparing this  with $\sigma_{e}(x)$, since $\Delta_{e}$ is invariant under $s_{e}$, we see that
\begin{small}\begin{eqnarray}
s_{e} \cdot  \sigma_{e} (x)&=&s_{e}((1,2) + (2,3)+~\Delta_{e})\nonumber\\
&=&s_{e}(s_{e}(x)+~\Delta_{e})\nonumber\\
&=&x~+~\Delta_{e}.\nonumber
\end{eqnarray}\end{small}
Hence, for  $h(x):= x + \Delta_{e}$ we have $\sigma_{e}(x)=h \cdot s_{e}(x)$.

\end{example}

\begin{remark}\label{remi2} Note that for $e,e'\in E_{0}$, the maps $h_{e}$ and $h_{e'}$ satisfy
\begin{eqnarray}
 h_{e}h_{e'}(x)
&=& x+<x,e'>\Delta_{e'}+<e,x>\Delta_{e}~. \label{note}
\end{eqnarray}
\end{remark}

In order to construct a representation for the action of $\pi_{1}(\A)$ on $C_{1}$, we shall begin by studying the image $B_{0}$ and the kernel $Z_{1}$ of $\partial: C_{1} \rightarrow C_{0}$.

\subsection{\texorpdfstring{The representation of $G_{1}$} {The representation of G1}}

For $y=(y_{1},\dots, y_{4g-4})\in C_{0}$, we define the linear map $f:C_{0}\rightarrow \mathbb{Z}_{2}$   by
\begin{equation}
 f(y)=\sum_{i=1}^{4g-4}y_{i} ~.\label{effe}
\end{equation}
\begin{proposition}
The image  $B_{0}$ of $\partial: C_{1} \rightarrow C_{0}$ is formed by elements with an even number of $1$'s, i.e.,  $B_{0}=\ker f.$\label{propi2}
\end{proposition}

\begin{proof}It is clear that $B_{0}\subset \ker f$. In order to check surjectivity we consider the edges $e_{i}\in C_{1}$ given by $e_{i}=(i,i+1)\in E'$, for $i=1, \dots,4g-5$. 
Given the elements $R^{k}\in B_{0}$ for $k=2,\cdots,4g-4$ defined as
 \begin{eqnarray}
R^{2}&:=&\partial e_{1}  =(1,1,0,\dots,0)~,\nonumber\\
&\vdots&\nonumber\\
R^{k}&:=& R^{k-1}+\partial e_{k-1}=(1,0,\dots,0,1,0,\dots,0),\nonumber
 \end{eqnarray}
one may generate any distribution of an even number of $1$'s. Hence,  $B_{0}={\rm span}\{R^{k}\}$ which is the kernel of $f$. \end{proof}

\begin{proposition} The image $B_{0}$ and the kernel $Z_{1}$ of the derivative $~\partial: C_{1} \rightarrow C_{0}~$ have, respectively, dimension $4g-5$ and $2g+3$.
\end{proposition}

\begin{proof} From Prop. \ref{propi2} one has that $\dim(B_{0})=\dim(\ker f)=4g-5$. Furthermore, as $\dim C_{1}=\dim Z_{1} + \dim B_{0},$ the kernel $Z_{1}$ of $\partial$ has dimension $2g+3$. \end{proof}

Note that  $x_{1},x_{2},x_{4}\in Z_{1}$ and $x_{5}\notin Z_{1}$. From a homological viewpoint one can see that $x_{4}$ and  $\Delta_{e}~{\rm for}~e\in E_{0}$ form a basis for the kernel $Z_{1}$. We can extend this to a basis of $C_{1}$ by taking the edges $$\beta':= \{e_{i}=(i,i+1) ~{\rm for}~ 1\leq i\leq 4g-5\}\subset E'~,$$ whose images under $\partial$ form a basis for $B_{0}$, and hence a basis for a complementary subspace $V$ of $Z_{1}$. We shall denote by $\beta:=\{\beta_{0},\beta'\}$ the basis of $C_{1}$ for $\beta_{0}=\{\tilde{E}_{0},x_{4}\}$.

In order to generate the whole group $G_{1}$,  we shall study the action of $S'_{4g-4}$  by conjugation on $N$. Considering the basis $\beta$  one may construct a matrix representation of the maps $h_{e_{i}}$ for $e_{i} \in  E=\{E_{0},E'\}$. 

\begin{proposition} For $e\in E$, the matrix $[h_{e}]$ associated to $h_{e}$ in the basis $\beta$  is given by
\begin{equation}
 [h_{e}]=\left(\begin{array} {c|c}
  I_{2g+3}&A_{e}\\ \hline
0&I_{4g-5}\\
 \end{array}
\right)~,\nonumber
\end{equation}
\noindent  where the $(2g+3) \times (4g-5)$ matrix $A_{e}$ satisfies one of the following:
\begin{itemize}
 \item it is the zero matrix for $e\in E'$,
 \item it has only four non-zero entries  in the intersection of the row corresponding to $\Delta_{e}$ and the columns corresponding to an adjacent edge of $e$ for $e\in E_{0}-\{u_{5},l_{6}\}$,
 \item it has three non-zero entries  in the intersection of the row corresponding to $\Delta_{e}$ and the columns corresponding to an adjacent edge of $e$ for $e=u_{5},l_{6}$.
\end{itemize}
\end{proposition}

\begin{proof}As we have seen before, for $e\in E_{0}$ the map $h_{e}$ acts as the identity on the elements of $\beta_{0}$. Furthermore, any edge $e\in E_{0}-\{u_{5},l_{6}\}$ is adjacent to exactly four edges in $\beta'$. In this case $h_{e}$ has exactly four non-zero elements in the intersection of the row corresponding to $\Delta_{e}$ and the columns corresponding to edges in $\beta'$ adjacent to $e$. In the case of $e=u_{5},l_{6}$, the edge $e$ is adjacent to exactly 3 edges in $\beta'$ and thus $h_{e}$ has only 3 non-zero entries. \end{proof}

Recall that $S'_{4g-4}$ preserves the space spanned by $E'$, and hence also the subspace $V$, and acts trivially on $Z_{1}$. In the basis $\beta$ the action of an element $s\in S'_{4g-4}$ has a matrix representation given by
\begin{equation}
 [s]=\left(\begin{array} {c|c}
  I_{2g+3}&0\\ \hline
0&\pi_{s}\\
 \end{array}
\right)~,\nonumber
\end{equation}
where $\pi_{s}$ is the permutation action corresponding to $s$. Hence,  for $f\in E_{0}$ we may construct the matrix for a conjugate of $h_{f}$ as follows:
\begin{eqnarray}
[s] [h_{f}] [s]^{-1}&=&
 \left(\begin{array} {c|c}
  I_{2g+3}&0\\ \hline
0&\pi_{s}\\
 \end{array}
\right) 
\left(\begin{array} {c|c}
  I_{2g+3}&A_{f}\\ \hline
0&I_{4g-5}\\
 \end{array}
\right) 
\left(\begin{array} {c|c}
  I_{2g+3}&0\\ \hline
0&\pi_{s}^{-1}\\
 \end{array}
\right)\nonumber\\
&=&\left(\begin{array} {c|c}
  I_{2g+3}&A_{f}\pi_{s}^{-1}\\ \hline
0&I_{4g-5}\\
 \end{array}
\right).\nonumber
\end{eqnarray}

\begin{proposition} The normal subgroup $N \subset G_{1}$  consists of all matrices of the form
\begin{equation}
H=\left(\begin{array} {c|c}
  I_{2g+3}&A\\ \hline
0&I_{4g-5}\\
 \end{array}
\right),\nonumber
\end{equation}  
where $A$ is any matrix whose rows corresponding to $\Delta_{e}$ for $e\in E_{0}-\{u_{5},l_{6}\}$ have an even number of $1$'s, the row corresponding to $x_{4}$ is zero and the rows corresponding to $\Delta_{u_{5}},\Delta_{l_{6}}$ have any distribution of $1$'s.   
\label{muy}
\end{proposition}

\begin{proof} Given $e_{i} \in E_{0}-\{u_{5},l_{6}\}$,  the matrix $A_{e_{i}}$ has only four non-zero entries in the row corresponding to $\Delta_{e_{i}}$. Thus, for $g>2$ , there exist  elements $s_{1},s_{2}\in S'_{4g-4}$ with associated permutations $\pi_{1}$ and $\pi_{2}$ such that the matrix
$\tilde{A}_{e_{i}}:=A_{e_{i}}\pi_{1}^{-1}+A_{e_{i}}\pi_{2}^{-1}$ has only two non-zero entries in the row corresponding to $\Delta_{e_{i}}$, given by the vector $R^{5}$  defined in the proof of Proposition \ref{propi2}.
Furthermore, by Remark \ref{remi2}, we have 
\begin{eqnarray}[s_{1}h_{e_{i}}s^{-1}_{1}s_{2}h_{e_{i}}s^{-1}_{2}]
&=&[s_{1}h_{e_{i}}s^{-1}_{1}][s_{2}h_{e_{i}}s^{-1}_{2}]\nonumber \\
&=&
\left(\begin{array} {c|c}
  I_{2g+3}&A_{e_{i}}\pi_{1}^{-1}\\ \hline
0&I_{4g-5}\\
 \end{array}
\right)\left(\begin{array} {c|c}
  I_{2g+3}&A_{e_{i}}\pi_{2}^{-1}\\ \hline
0&I_{4g-5}\\
 \end{array}
\right)\nonumber \\
&=&
\left(\begin{array} {c|c}
  I_{2g+3}&\tilde{A}_{e_{i}}\\ \hline
0&I_{4g-5}\\
 \end{array}
\right).\nonumber
\end{eqnarray}
Considering different $s\in S'_{4g-4}$ acting on $s_{1}h_{e_{i}}s^{-1}_{1}s_{2}h_{e_{i}}s^{-1}_{2}$, one can obtain the matrices $\{A_{e_{i}}^{k}\}_{k=2}^{4g-5}$ with $R^{k}$ as the only non-zero row corresponding to $\Delta_{e_{i}}$:
\begin{eqnarray}
A_{e_{i}}^{2}=\left(\begin{array} {ccccc}
 &&\boldsymbol{0}&&\\ \hline
1&1&0&\dots&0\\\hline
  &&\boldsymbol{0}&&\\
 \end{array}\right),
\dots~,~~
 A_{e_{i}}^{4g-5}=
\left(\begin{array} {ccccc}
 &&\boldsymbol{0}&&\\ \hline
1&0&\dots&0&1\\\hline
 &&\boldsymbol{0}&&\\
 \end{array}\right).\label{row}\nonumber
\end{eqnarray}  
Thus, by composing the elements of $N$ to which each $A^{k}_{e_{i}}$ corresponds, we can obtain any possible distribution of an even number of $1$'s in the only non-zero row.
Similar arguments  can be used for the matrices corresponding to $u_{5},l_{6}$ which in this case may have any number of $1$'s in the only non-zero row. 
Recalling Remark \ref{remi2} we are then able to generate any matrix $A\in N$ as described in the proposition.\end{proof}
From the decomposition of $G_{1}$ and above results, we have:
\begin{theorem}
 The representation of  $\sigma \in G_{1}$ in the basis $\beta$ is given by
\begin{equation}
 [\sigma]=\left(\begin{array} {c|c}
I&A\\\hline
0&\pi                 
                \end{array}
\right) ~,
\end{equation}
where  $\pi$ represents a permutation on $\mathbb{Z}_{2}^{4g-5}$ and  $A$ is any matrix whose rows corresponding to $\Delta_{e}$ for $e\in E_{0}-\{u_{5},l_{6}\}$ have an even number of $1$'s, the row corresponding to $x_{4}$ is zero and the rows corresponding to $\Delta_{u_{5}}$ and $\Delta_{l_{6}}$ have any distribution of $1$'s.
\end{theorem}

\subsection{\texorpdfstring{The monodromy action of $\pi_{1}(\A)$ on $P[2]$} {The monodromy action on P[2]}}

As seen previously,  $P[2]$ can be obtained as the quotient of $C_{1}$ by the four relations $x_{1},x_{2},x_{4}$ and $x_{5}$. It is important to note that these relations  are preserved by the action of $\pi_{1}(\A)$.  Furthermore, $x_{1},x_{2},x_{4} \in Z_{1}$ and  $\partial  x_{5} = (1,1,\cdots,1)\in B_{0}$. Hence, we have the following maps
\begin{equation}\mathbb{Z}_{2}^{2g}\cong \frac{Z_{1}}{<x_{1},x_{2},x_{4}>}\rightarrow P[2]\rightarrow \frac{B_{0}}{<(1,1,\cdots,1)>}\cong \mathbb{Z}_{2}^{4g-6}~. \label{exact}\end{equation}
The monodromy group $G_{0}$ is given by the action  on the quotient $P[2]$ induced by the action of $G_{1}$ on $C_{1}$. Note that having $x_{1},x_{2},x_{3},x_{5}=0$ implies 
\begin{eqnarray}
 0&=&x_{1}+x_{4}=\sum_{l_{i}\in E_{0}}\Delta_{l_{i}},\\
 0&=&x_{2}+x_{4}=\sum_{u_{i}\in E_{0}}\Delta_{u_{i}},\\
0&=&x_{5}+x_{4}= \sum_{{\rm even}}u_{i}+\sum_{{\rm odd}} l_{i}.
\end{eqnarray}
Then, one may write
\begin{eqnarray}
 \Delta_{l_{6}}=\sum_{l_{i}\in E_{0}-\{l_{6}\}}\Delta_{l_{i}}\label{sum1}~,~{~\rm and~}~
\Delta_{u_{5}}=\sum_{l_{i}\in E_{0}-\{u_{5}\}}\Delta_{u_{i}}~.\label{sum2}
\end{eqnarray}
Moreover, as $x_{1}+x_{2}+x_{4}+x_{5}=0$, one can express the edge $u_{6}$ in terms of  elements in $\beta'-\{u_{6}\}$.

\begin{definition}\label{def:basis}
  For $\tilde{\beta}_{0}:=\beta_{0}-\{x_{4},\Delta_{l_{6}},\Delta_{u_{5}} \}$ and $\tilde{\beta}':=\beta'-\{u_{6}\}$, let $\tilde{\beta}:=\{\tilde{\beta}_{0},\tilde{\beta}'\}$.  
\end{definition}
\begin{proposition} 
The elements in $\tilde{\beta}_{0}$ generate 
$Z_{1}/<x_{1},x_{2},x_{4}>$
and $\tilde{\beta}'$ generates a complementary subspace in $P[2]$.\label{propultima}
\end{proposition}

From the previous analysis, one obtains our main theorem, which gives an explicit description of the monodromy action on $P[2]$:

\begin{theorem}\label{teo}
 The representation of  $\sigma \in G_{0}$ in the basis $\tilde{\beta}$ is given by
\begin{equation}
 [\sigma]=\left(\begin{array} {c|c}
I_{2g}&A\\\hline
0&\pi                 
                \end{array}
\right),
\end{equation}
where
\begin{itemize}
\item $\pi$ is the quotient action on $\mathbb{Z}_{2}^{4g-5}/(1,\cdots,1)$ induced by the permutation action of the symmetric group $S_{4g-4}$ on $\mathbb{Z}_{2}^{4g-5}$;
\item $A$  is any $(2g)\times (4g-6)$ matrix with entries in $\mathbb{Z}_{2}$.
                  \end{itemize}
\end{theorem}

\begin{proof} As seen before,  the action of $G_{0}$ on $B_{0}/<\partial x_{5}=(1,1,\cdots,1)>$ is given by the quotient action of the symmetric group. Furthermore, replacing $\Delta_{u_{5}}$ and $\Delta_{l_{6}}$ by the sums in (\ref{sum1}), one can use similar arguments to the ones in Proposition \ref{muy}  to obtain any number of 1's in all the rows of the matrix $A$.
\end{proof}

\begin{remark}   The  intersection pairing in the above analysis means that the previous theorem gives an irreducible  representation of the symplectic group $Sp(4g-6, \mathbb{Z}_{2})$ (
 the reader should refer to the work of  Gow and  Kleshchev \cite{kles} for details). 
\end{remark}

\begin{remark} We have seen in Proposition \ref{trivial} that the action of $G_{1}$ is trivial on $Z_{1}$. Moreover, the space $\mathbb{Z}_{2}[x_{1},x_{2},x_{4},x_{5}]$ is preserved by the action of $G_{1}$, and thus one can see that the induced monodromy action  on the $2g$-dimensional subspace  $Z_{1}/\mathbb{Z}_{2}[x_{1},x_{2},x_{4},x_{5}]$ is trivial. Geometrically, there are $2^{2g}$ sections of the Hitchin fibration given by choices of the square root of $K$. These sections meet each Prym in $2^{2g}$ points which also lie in $P[2]$. Since we can lift a closed curve by  a section, these are acted on trivially by the monodromy.
\end{remark}

 The analysis of the monodromy action for other split real forms encounters many new complications. Hence, in the following chapters we shall study principal $G$-Higgs bundles and the corresponding Hitchin fibration considering a different approach.




\chapter{Spectral data for $U(p,p)$-Higgs bundles}\label{ch:supp}

In previous chapters we constructed principal Higgs bundles from a Lie theoretic point of view. In particular,  in Chapter \ref{ch:split} we gave a description of the fixed point set of the involution $$\Theta:~(E,\Phi)\mapsto (\sigma(E),-\sigma(\Phi))$$  associated to split real forms, acting on the Higgs bundle moduli space for the corresponding complex Lie group. These results were then used  to study $SL(2,\mathbb{R})$-Higgs bundles   via the monodromy action in Chapter \ref{ch:monodromy}. For non-split real forms the fixed point set of the corresponding involution $\Theta$ does not give a covering of the regular locus of the Hitchin base. 
Consequently, an alternative approach to  the monodromy action needs to be used in order to understand the geometric and topological properties of the corresponding Higgs bundles moduli spaces.

In this Chapter we shall study the particular case of $U(p,p)$ and $SU(p,p)$-Higgs bundles over a compact Riemann surface $\Sigma$, and define the spectral data associated to them. Many authors have studied connectivity for $SU(p,p)$-Higgs bundles by looking at the  nilpotent cone in the corresponding Hitchin fibration (see, among others, the work of  Bradlow,  Garcia-Prada and Gothen \cite{brad}). In contrast, here we obtain information about $U(p,p)$ and $SU(p,p)$-Higgs bundles via the regular fibres of the Hitchin fibration. 
One should note that $U(p,q)$-Higgs bundles for $p<q$ have Higgs fields of rank at most $2p$. Thus, the characteristic polynomial of $\Phi$ is reducible, and the induced spectral curve is singular. It is natural then to consider separately the cases of $p=q$ and $p\neq q$. 

We  begin this chapter by recalling the main properties of $U(p,p)$-Higgs bundles in Section \ref{sec:propupp}. Then, we  study the spectral data for $U(p,p)$-Higgs bundles in Section \ref{sec:upp}, and for $SU(p,p)$-Higgs bundles in Section \ref{sec:supp}. Some applications of these results are given in Chapter \ref{ch:applications}, and  following the analysis of $SO(2m+1,\mathbb{C})$-Higgs bundles given in Chapter \ref{ch:complex} (see \cite{N3}), we consider $U(p,q)$-Higgs bundles for $p\neq q$ in Chapter \ref{ch:further}.

\section{\texorpdfstring{$U(p,p)$-Higgs bundles}{U(p,p)-Higgs bundles}}\label{sec:propupp}

Recall that the unitary group $U(p,p)$  of signature $(p,p)$ is a non-compact real form of $GL(2p,\mathbb{C})$, and has maximal compact subgroup   $H=U(p)\times U(p)$. The complexified  Lie algebra of $H$ is $\textgoth{h}=\textgoth{gl}(p,\mathbb{C})\oplus\textgoth{gl}(p,\mathbb{C})$. For 
\[I_{p,p}=\left(
\begin{array}
 {cc}
-I_{p}&0\\ 
0&I_{p}
\end{array}\right),\]
the antilinear involution $\tau$ on $\mathfrak{gl}(2p,\mathbb{C})$ which fixes $\mathfrak{u}(p,p)$ is given by
$
 \tau(X)=-I_{p,p}X^{*}I_{p,p}.
$ 
Moreover, the  involution relating the compact real structure and the non-compact real structure $\mathfrak{u}(p,p)$  on $\mathfrak{gl}(2p,\mathbb{C})$ is
$\sigma(X) = I_{p,p}X I_{p,p}.$
As seen in Chapter \ref{ch:real}, the involution $\sigma$  induces an involution $\Theta: (E,\Phi)\mapsto (\sigma(E),-\sigma(\Phi))$ on the classical Higgs bundle moduli space $\mathcal{M}$. In this case, $\Theta$ is given by
\[\Theta: (E,\Phi)\mapsto (E,-\Phi).\]
 The isomorphism classes of $U(p,p)$-Higgs bundles are given by fixed points of the involution $\Theta$ on $\mathcal{M}$ corresponding to vector bundles $E$ which have an automorphism conjugate to $I_{p,p}$ sending $\Phi$ to $-\Phi$, and  whose $\pm 1$ eigenspaces have dimensions $p$ and $q$.

By considering the Cartan involution on the complexified Lie algebra of $U(p,p)$ one has 
$\textgoth{u}(p,p)^{\mathbb{C}}=\textgoth{gl}(2p,\mathbb{C})=(\textgoth{gl}(p,\mathbb{C})\oplus\textgoth{gl}(p,\mathbb{C}))+\textgoth{m}^{\mathbb{C}},$
where $\textgoth{m}^{\mathbb{C}}$ corresponds to the off diagonal elements of $\textgoth{gl}(2p,\mathbb{C})$. Then, a  principal $U(p,p)$-Higgs bundle $(P,\Phi)$ is given by a principal $H^{\mathbb{C}}$-bundle $P$ and a holomorphic section $\Phi$ of $(P \times_{{\rm Ad}} \textgoth{m}^{\mathbb{C}})\otimes K$. 

\begin{definition} \label{higgsupp} A $U(p,p)$-Higgs bundle over $\Sigma$ is a pair $(E,\Phi)$ where
 $E=V\oplus W$ for $V,W$ rank $p$ vector bundles over $\Sigma$, and the Higgs field $\Phi$  given by
 \begin{eqnarray}
        \Phi=\left( \begin{array}
          {cc} 0&\beta\\
\gamma&0
         \end{array}\right), 
        \end{eqnarray}
for $\beta:W\rightarrow V\otimes K$ and $\gamma:V\rightarrow W \otimes K$.
\end{definition}

\begin{definition}\label{def:supp} An $SU(p,p)$-Higgs bundle is a $U(p,p)$-Higgs bundle $(V\oplus W,\Phi)$ for which the vector bundles satisfy
$\Lambda^{p}V\cong\Lambda^{p}W^{*}.$
\end{definition}

\subsection{The spectral curves}

Consider a stable $U(p,p)$-Higgs pair $(V\oplus W, \Phi)$, with $\Phi$
%
%
given  as in Definition \ref{higgsupp}.
Note that
\begin{eqnarray}
{\rm tr}~\Phi^{n}=\left\{\begin{array}
                      {ccl} 0&{\rm for}& n~{\rm odd},\\
2{\rm tr}(\gamma\beta)^{n/2}&{\rm for}& n~{\rm even}.
                     \end{array}\right.\label{tr1}
\end{eqnarray} 
Thus the characteristic polynomial of $\Phi$, whose coefficients are polynomials in ${\rm tr}(\Phi^{i})$,  defines a spectral curve $\rho:S\rightarrow \Sigma$ in the total space $X$ of $K$ with equation 
\begin{eqnarray}\eta^{2p}+a_{1}\eta^{2p-2}+\ldots+a_{p-1}\eta^{2}+a_{p}=0,\label{curves}\end{eqnarray}
where $a_{i}\in H^{0}(\Sigma,K^{2i})$, and $\eta$ is the tautological section of $\rho^{*}K$ in $X$. By Bertini's theorem this curve is generically smooth, and thus in the remainder of this chapter we shall assume $S$ is smooth.
 The $2p$-fold cover $S$ has an involution $\eta \mapsto -\eta$, which we shall denote by $\sigma$ following the notation of Chapter \ref{ch:complex}. The quotient of $S$ by the action of $\sigma$ defines a $p$-fold cover $\overline{\rho}:\overline{S}\rightarrow \Sigma$ in the total space of $K^{2}$, whose equation is  given by
\begin{eqnarray}
\xi^{p}+a_{1}\xi^{p-1}+\ldots+a_{p-1}\xi+a_{p}=0,
 \label{curvess}
\end{eqnarray}
where $\xi=\eta^{2}$ is the tautological section of $\overline{\rho}^{*}K^{2}$. Since $\overline{S}$ is the quotient of a smooth curve, it is also smooth. We let $\pi:S\rightarrow \overline{S}$ be the double cover given by the above quotient:
 \begin{eqnarray}
 \xymatrix{S\ar[rr]_{\pi}^{2:1}\ar[dr]^{\rho}_{2p:1}&&\overline{S}\ar[dl]^{p:1}_{\overline{\rho}}.\\
&\Sigma&}
\end{eqnarray}
 We shall denote by $g_{S}$ and $g_{\overline{S}}$ the genus of  $S$ and $\overline{S}$, respectively.
Since the cotangent bundle has trivial canonical bundle, and the canonical bundle of $K^2$ is $\bar \rho^*K^{-1}$, the adjunction formula gives 
  $K_{S}\cong \rho^{*}K^{2p}$ and  $K_{\bar S}\cong  \bar \rho^{*} K^{2p} \otimes \bar \rho^*K^{-1}$. Thus, one has
\begin{eqnarray}g_{S}&=&4p^{2}(g-1)+1,\\
 g_{\overline{S}}&=& (2p^{2}-p)(g-1)+1.\end{eqnarray}


\section{\texorpdfstring{The spectral data of $U(p,p)$-Higgs bundles}{The spectral data of U(p,p)-Higgs bundles}}\label{fibu}\label{sec:upp}


A description of the fibres of the Hitchin fibration
$h:\mathcal{M}_{G^{c}}\rightarrow \mathcal{A}_{G^{c}},$ for classical complex Lie groups $G^{c}$, was given in  Chapter \ref{ch:complex}, following \cite{N2}, in terms of line bundles on associated curves.  In the case of classical Higgs bundles of rank $n$, a point in the regular locus of the Hitchin base defines a curve $\rho:S\rightarrow \Sigma$ in the total space $X$ of the canonical bundle $K$, whose equation is given by
\[\eta^{n}+a_{1}\eta^{n-1}+ \ldots +\eta^{1}a_{n-1}+a_{n}=0,\] 
where $\eta$ is the tautological section of $\rho^{*}K$ in $X$, and $a_{i}\in H^{0}(\Sigma,K^{i})$. Given a line bundle $M$ on $S$, one could obtain a classical Higgs pair $(E,\Phi)$ by taking the direct image
$E:=\rho_{*}M,$
and letting the Higgs field $\Phi$ be the map induced by the multiplication of the tautological section
\[H^{0}(\rho^{-1}(\mathcal{U}),M)\rightarrow  H^{0}(\rho^{-1}(\mathcal{U}), M \otimes \rho^{*}K),\]
for $\mathcal{U}$ an open set in the compact Riemann surface $\Sigma$.
In this section we shall extend these methods to study the fixed point set of the involution
\[\Theta: (E,\Phi)\mapsto (E,-\Phi), \]
in the classical Higgs bundles moduli space $\mathcal{M}$, which correspond to $U(p,p)$-Higgs bundles. 
In particular, we obtain the following correspondence.  

\begin{proposition}
\label{teo1}\label{remark}
Given a $U(p,p)$-Higgs bundle with non-singular spectral curve one can construct a pair $(S,M)$ where 
\begin{enumerate}
\item[(a)] \label{in1} the curve $\rho:S\rightarrow \Sigma$ is an irreducible non-singular $2p$-fold cover of $\Sigma$  given by the equation
\begin{eqnarray}\eta^{2p}+a_{1}\eta^{2p-2}+\ldots+a_{p-1}\eta^{2}+a_{p}=0,\nonumber\end{eqnarray}
in the total space of $K$, where $a_{i}\in H^{0}(\Sigma, K^{2i})$, and $\eta$ is the tautological section of $\rho^{*}K$. The curve $S$ has an involution $\sigma$ acting by $\sigma(\eta)=-\eta$;

\item[(b)] \label{in2} $M$ is  a line bundle on $S$ such that $\sigma^{*}M\cong M$. 
 
\end{enumerate}

\noindent Conversely, given a pair $(S,M)$ satisfying (a) and (b), there is an associated stable $U(p,p)$-Higgs bundle.

\end{proposition}

\begin{proof}
We shall begin by constructing the Higgs bundle induced by a given pair $(S,M)$ satisfying $(a)$ and $(b)$. Consider  $\rho:S\rightarrow \Sigma$ and $M$ as in $(a)$ and $(b)$,  and denote the lifted action of $\sigma$ to $M$ also  by $\sigma$. Then, on an invariant open set  $\rho^{-1}(\mathcal{U})$  we may decompose the sections of $M$ into the invariant and anti-invariant parts:
\[H^{0}(\rho^{-1}(\mathcal{U}),M)=H^{0}(\rho^{-1}(\mathcal{U}),M)^{+} \oplus H^{0}(\rho^{-1}(\mathcal{U}),M)^{-}.\]
From the definition of the direct image of a line bundle there is a similar decomposition of $H^{0}(\mathcal{U},\rho_{*}M)$ into
\[H^{0}(\mathcal{U},\rho_{*}M)=H^{0}(\mathcal{U},\rho_{*}M)^{+}\oplus H^{0}(\mathcal{U},\rho_{*}M)^{-}.\] 
Thus, generically,  we may write 
\[\rho_{*}M=E_{+}\oplus E_{-},\]
where $E_{\pm}$ are rank $p$ vector bundles on $\Sigma$. At a point $x$ such that $a_p(x)\ne 0$, the involution $\sigma$ has no fixed points on $\rho^{-1}(x)$. Moreover,  if $x$ is not a branch point, $\rho^{-1}(x)$ consists of $2p$ points $e_1,\dots, e_p, \sigma e_1,\dots \sigma e_p$. The fibre of the direct image is then isomorphic to $\mathbb{C}^p\oplus \mathbb{C}^p$ with the involution $(v,w)\mapsto (w,v)$, so that the fibre of $E_+$ is given by the invariant points $(v,v)$ and the one of $E_{-}$ is given by the anti-invariant points $(v,-v)$. 

In order to define the Higgs field associated to $M$, we follow the ideas done for the classical case in Chapter \ref{ch:complex} and consider the tautological section $\eta$ of $\rho^{*}K$, which induces  the multiplication map 
\begin{eqnarray}H^{0}(\rho^{-1}(\mathcal{U}),M)\xrightarrow{\eta}H^{0}(\rho^{-1}(\mathcal{U}),M\otimes \rho^{*}K).\label{map1}\end{eqnarray}
By definition of direct image, this map gives the Higgs field 
\[ \Phi:  \rho_{*}M\rightarrow \rho_{*}M\otimes K.\]
Furthermore, as $\sigma(\eta)=-\eta$, the Higgs field $ \Phi$ maps $E_{+}\mapsto E_{-}\otimes K$ and $E_{-}\mapsto E_{+}\otimes K$. Thus, we may write
\[ \Phi=\left(\begin{array}{cc}0&\beta\\  \gamma&0\end{array}\right)\]
for $\gamma:E_{+}\rightarrow E_{-}\otimes K$ and $\beta:E_{-}\rightarrow E_{+}\otimes K $.
Note that as we have considered an irreducible  curve $S$, there is no proper subbundle of $E_{+}\oplus E_{-}$ which is preserved by $\Phi$, and hence the Higgs pair $(E_{+}\oplus E_{-},\Phi)$ is a stable $U(p,p)$-Higgs bundle associated to $(S,M)$.

Conversely, by considering a stable $U(p,p)$-Higgs bundle $(E=V\oplus W, \Phi)$ defined as in Definition \ref{higgsupp},  one can obtain the induced pair $(S,M)$ as follows. The characteristic polynomial of $\Phi$ defines the curve $S$, which by stability is smooth and irreducible.
We have seen in Chapter \ref{ch:complex} that the regular fibres of the Hitchin fibration for classical Higgs bundles are isomorphic to the Picard variety of $S$. As a classical Higgs bundle, $(E,\Phi)$ has a corresponding line bundle $M$ on the spectral curve $S$. Since the involution $\sigma$ acts trivially on the equation of  $S$, the fixed points correspond to the action on the Picard variety, and thus $\sigma^{*}M\cong M$.
Hence, the pair $(S,M)$ associated to the Higgs bundle $(E,\Phi)$  satisfies Item (a) and Item (b) of Proposition \ref{remark}. \end{proof}

By means of the 2-fold covering $\pi:S \rightarrow \bar S$, the line bundle $M$ from Proposition \ref{teo1} can be interpreted in terms of line bundles on the quotient curve $\bar S$ in the following way.  Let $\mathcal{V}$ be an open set in $\overline{S}$.  Then, the invariant and anti-invariant sections of the line bundle $M$ on $S$ introduced in the Proposition \ref{teo1} give a decomposition
\[H^{0}(\pi^{-1}(\mathcal{V}),M)=H^{0}(\pi^{-1}(\mathcal{V}),M)^{+}\oplus H^{0}(\pi^{-1}(\mathcal{V}),M)^{-}.\]
By definition of direct image, there are two line bundles $U_{1}$ and $U_{2}$ on the quotient curve $\overline{S}$ such that  
\begin{eqnarray}
H^{0}(\pi^{-1}(\mathcal{V}),M)^{+}\cong H^{0}(\mathcal{V},U_{1}) ~{~\rm and~}~ H^{0}(\pi^{-1}(\mathcal{V}),M)^{-}\cong H^{0}(\mathcal{V},U_{2}), \label{us}
\end{eqnarray}
and thus $\pi_{*}M=U_{1}\oplus U_{2}$.
 Therefore the pair $(S,M)$ corresponding to the $U(p,p)-$Higgs bundle $(E,\Phi)$, which satisfies Item (a) and Item (b) of Proposition \ref{remark}, defines two associated line bundles $U_{1}$ and $U_{2}$ on the quotient curve $\overline{S}$.

\subsection{The associated invariants}

 Let $(E=V\oplus W,\Phi)$ be a $U(p,p)$-Higgs bundle. By reduction to the maximal compact subgroup, a flat $U(p,p)$ bundle has two integer invariants. These invariants correspond  to the degrees $v=\deg V$ and $w=\deg W$ of the Higgs bundle $(E,\Phi)$. We shall see next how they arise from the isomorphism  $\sigma^{*}M\cong M$. This isomorphism gives an action of $+1$ or $-1$ on the fixed points of the involution and defines one of the invariants:

\begin{prop}\label{teo2}
Let $(V\oplus W,\Phi)$ be a $U(p,p)$-Higgs bundle corresponding to a pair  $(S,M)$ as in Proposition \ref{teo1}. For $m$ the degree of $M$, and  $\tilde{m}$ the number  of fixed points of the involution $\sigma$ for which it acts as $-1$ on $M$, the degrees $v$ and $w$ of $V$ and $W$ are given by
\begin{eqnarray}
  v &=&\frac{m-\tilde{m}}{2}+(2p-2p^{2})(g-1),\nonumber\\
w &=&\frac{m+\tilde{m}}{2}-2p^{2}(g-1).\nonumber
\end{eqnarray}
\end{prop}
\begin{proof} Consider  the line bundle $M$ on $S$ whose degree we denote by $m$ and which corresponds, through the previous analysis, to the pair  $(V\oplus W,\Phi)$. In particular, we have that  
$\rho_{*}M=V\oplus W.$ 
The involution $\sigma$ preserves $M$, and  on its $4p(g-1)$ fixed points it acts as $\pm 1$. We shall denote by $\tilde{m}$ the number of points on which $\sigma$ acts as $-1$.  
As noted before, in terms of the $\pm 1$ eigenspaces of $\sigma$ we may write
\[H^{0}(\rho^{-1}(\mathcal{U}),M)=H^{0}(\rho^{-1}(\mathcal{U}),M)^{+} \oplus H^{0}(\rho^{-1}(\mathcal{U}),M)^{-}\]
for an open set $\mathcal{U}\subset \Sigma$. Furthermore, there is a similar decomposition of the sections of the direct image of $M$ on $\Sigma$:
\begin{eqnarray}
H^{+}:= H^{0}(\rho^{-1}(\mathcal{U}),M)^{+} &\cong& H^{0}(\mathcal{U},V),\\
H^{-}:= H^{0}(\rho^{-1}(\mathcal{U}),M)^{-} &\cong& H^{0}(\mathcal{U},W).
\end{eqnarray}
Recall that for the $2p$-fold cover $\rho:S \rightarrow \Sigma$,  the degree of the direct image $\rho_{*}M$ is given by
$\deg(\rho_{*}M)=\deg(M)+(1-g_{S})-\deg(\rho)(1-g).$
Then, we have that
\begin{eqnarray}
 v+w=\deg \rho_{*}M&=&m +(1-g_{S})-2p(1-g)\nonumber\\
&=&m +1-1-4p^{2}(g-1)-2p(1-g)\nonumber\\
&=&m +(2p-4p^{2})(g-1),\nonumber
\end{eqnarray}
and thus $m=v+w+(4p^{2}-2p)(g-1)$. 

Given $L$ a line bundle over $\overline{S}$, the involution $\sigma$ acts trivially on the pull-back $\pi^{*}L$. Furthermore, one may consider the tensor product $M\otimes \pi^{*}L$ for which the direct image under $\pi$ is given by
\[\pi_{*}(M\otimes \pi^{*}L)=\pi_{*}M\otimes L. \]
For convenience we shall assume that the degree  $l:=\deg L$ is sufficiently large. Then,
$H^{1}(S,M\otimes \pi^{*}L)$ vanishes and thus one may use Riemann-Roch to calculate the dimension of  $H^{0}( S,M\otimes \pi^{*}L)$.
Over the fixed points of the involution,  $\sigma$ acts as $+1$ on $\pi^{*}L$ and so the involution acts as $-1$ on  $M\otimes \pi^{*}L$ for exactly $\tilde{m}$ points. Thus, we may assume that the degree $m$ of $M$ is sufficiently large to make  $H^{1}(S,M)$  vanish.

 In order to find the relation between $\tilde{m}$ and the degrees $v$ and $w$ we shall first calculate the dimension of the $\pm1$-eigenspaces of $\sigma$ on $H^{0}(S,M)$. Then, by means of Riemann-Roch and Serre duality, we shall find the degrees of $v$ and $w$ of $V$ and $W$  in terms of $\tilde{m}$, and from there give an equation for $\tilde{m}$.
By Riemann-Roch and Serre duality one has that \begin{small}
\[\dim H^{0}( S,M)=\deg M+ 1-g_{S}=m-4p^{2}(g-1).\]\end{small}
For $h^{+}$ and $h^{-}$ the dimension of the $\pm1$ eigenspaces of $H^{0}(S,M)$  respectively,
$$h^{+}+h^{-}=m-4p^{2}(g-1).$$
Consider the Lefschetz number as defined in \cite{at} associated to the involution $\sigma$ on $S$, which is given by:
\begin{eqnarray}
 L(\sigma)&=&\sum (-1)^{q} {\rm trace}~\sigma |_{H^{0,q}(M)}\nonumber\\
&=& {\rm trace}~\sigma|_{H^{0}(M)}\nonumber\\
&=& h^{+}-h^{-}.\nonumber
\end{eqnarray}
From the holomorphic Lefschetz theorem  \cite[Theorem 4.12]{at}, we can express $L(\sigma)$ as:\begin{small}
\begin{eqnarray}
 L(\sigma)&=&\frac{(-\tilde{m})+(4p(g-1)-\tilde{m})}{2}\nonumber\\
&=&2p(g-1)-\tilde{m}.\nonumber
\end{eqnarray}\end{small}
Note that as the Lefschetz number only depends on $\tilde{m}$, its value is the same for the line bundles $M$ and $M\otimes \pi^{*}L$.

The dimensions  $h^{+}$ and $h^{-}$ can then be obtained as follows:\begin{small}
\begin{eqnarray}
 h^{+}&=& \frac{m-4p^{2}(g-1)+2p(g-1)-\tilde{m}}{2}
= \frac{m-\tilde{m}}{2}+(p-2p^{2})(g-1),\nonumber\\
 h^{-}&=&\frac{m-4p^{2}(g-1)-2p(g-1)+\tilde{m}}{2}
=\frac{m+\tilde{m}}{2}-(p+2p^{2})(g-1).\nonumber
\end{eqnarray}\end{small}

These equations may be considered when applying Riemann-Roch and Serre duality in order to give an expression of the degrees  $v$ and $w$ in terms of $m$ and $\tilde{m}$:
\begin{eqnarray}
 v&=&h^{+}+p(g-1)
=\frac{m-\tilde{m}}{2}+(2p-2p^{2})(g-1),\nonumber\\
w&=&h^{-}+p(g-1)
=\frac{m+\tilde{m}}{2}-2p^{2}(g-1).\nonumber
\end{eqnarray}
Note that
\begin{eqnarray}
 w-v&=& \frac{m+\tilde{m}}{2}-2p^{2}(g-1) -\left(\frac{m-\tilde{m}}{2}+(2p-2p^{2})(g-1)\right)\nonumber\\
&=&\tilde{m}-2p(g-1).\nonumber
\end{eqnarray}
Hence, the number $\tilde{m}$ can be expressed as
$\tilde{m}=w-v+2p(g-1).$ 
\end{proof}

\begin{rem}
From the proof of Proposition \ref{teo2}, the parity of the degree of $M$ and the number of fixed points $\tilde{m}$ over which $\sigma$ acts as $-1$ need to be the same. 
\end{rem}
\begin{rem}\label{mvar}
 Throughout this chapter we have assumed that $V$ and $W$ were obtained via the invariant and anti-invariant sections of $M$, respectively. However, this choice is arbitrary, and interchanging the role of $V$ and $W$ corresponds to considering the involution $-\sigma$ acting on $S$. Moreover, considering $-\sigma$ corresponds to taking the invariants $(m,\overline{m})$, instead of $(m,\tilde{m})$, where
$\overline{m}=4p(g-1)-\tilde{m}.$
\end{rem}

The invariant $\tilde{m}$ introduced in Proposition \ref{teo2} gives the number of fixed points of $\sigma$ over which the involution acts as -1 on $M$, i.e., is the degree of a positive divisor  of the section $a_{p}={\rm det}( \beta \gamma)$ associated to a Higgs bundle $(V\oplus W,\Phi )$.
Consider the line bundles $U_{1}$ and $U_{2}$ on $\bar S$ as defined in (\ref{us}), for which $\bar \rho_{*} U_{1}= V$ and  $\bar \rho_{*} U_{2}= W$. Then, the Higgs field $\Phi=(\beta,\gamma)$ induces two sections
\begin{eqnarray}
 b\in H^{0}(\overline{S},U_{2}^*\otimes U_{1}\otimes \overline{\rho}^*K),
~{~\rm~and~}~ c\in H^{0}(\overline{S},U_{1}^*\otimes U_{2}\otimes \overline{\rho}^*K).\nonumber
\end{eqnarray}
 By stability of the Higgs bundle, if $v\geq w$ (otherwise one can interchange $V$ and $W$),  the section $\gamma\neq 0$ induces a non vanishing section $c$ of 
$U_{1}^*\otimes U_{2}\otimes \overline{\rho}^*K$, and thus defines a  positive divisor $D$ of degree $\tilde{m}$ associated to the Higgs bundle. We shall denote by $[s]$ the divisor associated to a section $s$.

\begin{remark}
Since $D$ is a positive divisor of $[a_{p}]$, which has simple zeros and gives the ramification points, one can think of it as a divisor on $\Sigma$ as well as on $\bar \rho :\bar S\rightarrow \Sigma$, by identifying  $D$ and ${\rm Nm}(D)$. 
\end{remark}

Note that since $U_{1}$ and $U_{2}$ are constructed via the invariant and anti-invariant sections of $M$, respectively, the divisor $D$ is given by the fixed points of the involution (the zeros of $a_{p}$) over which $\sigma$ acts as $-1$ on $M$.

\begin{remark}\label{divisorDbis}
 Interchanging the roles of $V$ and $W$ is equivalent to interchanging $\sigma$ by $-\sigma$. In this case, the associated triple is $(S,M,\overline{D})$ where $\overline{D}$ is the divisor satisfying $[a_{p}]= D + \overline{D},$ which has degree $\overline{m}=4p(g-1)-\tilde{m}$.
\end{remark}

 Note than when $p=1$, the surface $\Sigma$ and the curve $\overline{S}$ coincide. For $p>1$, one may express the degrees of the bundles $U_{1}$ and $U_{2}$ in the construction of the spectral data in terms of the degrees of $V$ and $W$, or in terms of the degree of $M$, as follows:
\begin{eqnarray}
 \deg U_{1} &=  v +(2p^{2}-2p)(g-1) &= \frac{m}{2}-\frac{\tilde{m}}{2},\label{degu1}\\
 \deg U_{2} &=  w +(2p^{2}-2p)(g-1) &= \frac{m}{2}+\frac{\tilde{m}}{2}-2p(g-1).\label{degu2}
\end{eqnarray}

From the study of the spectral data for $U(p,p)$-Higgs bundles,  we have the following theorem:
\begin{teo}\label{teo11}
There is a one to one correspondence between  $U(p,p)$-Higgs bundles  with non-singular spectral curve  $(V\oplus W, \Phi)$ on a compact Riemann surface $\Sigma$ of genus $g>1$ for which $\deg V > \deg W $, and triples $(\bar S,U_{1},D)$ where 
\begin{itemize}
 \item $\bar S$ is a non-singular $p$-fold cover of $\Sigma$ in the total space of $K^{2}$ given by the equation
\[\xi^{p}+a_{1}\xi^{p-1}+\ldots+a_{p-1}\xi+a_{p}=0~~{~\rm for~}~a_{i}\in H^{0}(\Sigma,K^{2i})~;\]
 \item $U_{1}$ is a line bundle on $\bar S$ whose degree is 
$\deg U_{1} = \deg V +(2p^{2}-2p)(g-1) ;$
\item $D$ is a positive subdivisor of the divisor of $a_{p}$ of degree $\tilde{m}=\deg W-\deg V+2p(g-1).$
\end{itemize}
\end{teo}

\begin{proof}
From Proposition \ref{teo1} and Proposition \ref{teo2}, it is left to show the equivalence between the line bundles $M$ and $U_{1}$ satisfying the hypothesis, and the curves $S$ and $\bar S$.

Consider a pair $(S,M)$  as in Proposition \ref{teo1} and $D$ a positive subdivisor of the divisor of $a_{p}$ over which $\sigma$ acts as $-1$ on $M$ following Proposition \ref{teo2}. It was shown before  that the direct image  $\pi_{*}M$ on $\bar S=S/\sigma$ is decomposed into  $\pi_{*}M=U_{1}\oplus U_{2}$, for $U_{1}$ and $U_{2}$ line bundles on $\bar S$ satisfying 
\[[D]=U_{1}^{*}\otimes U_{2}\otimes \bar \rho^{*}K. \]
Note that for $v>w$, from equations (\ref{degu1})-(\ref{degu2}) one has that $\deg U_{1}>\deg U_{2}$. Hence, considering the line bundle $U_{1}$ of biggest degree, and identifying ${\rm Pic}^{\deg U_{1}}(\bar S)$ with ${\rm Jac}(\bar S)$, one can construct the triple $(\bar S,U_{1},D)$. 


Conversely, consider a triple $(\bar S,U_{1},D)$ for $\overline{\rho}:\bar S\rightarrow \Sigma$ a smooth $p$-fold cover of $\Sigma$ with  equation
\[\xi^{p}+a_{1}\xi^{p-1}+\ldots+ a_{p-1}\xi+a_{p}=0,\]
for $a_{i}\in H^{0}(\Sigma,K^{2i})$ and $\xi$ the tautological section of $\overline{\rho}^{*}K^{2}$, and  $D$ a positive subdivisor of zeros of $a_{p}$. The line bundle $U_{1}$ on $\bar S$ is considered in ${\rm Jac}(\bar S)$ via the isomorphism ${\rm Pic}^{\deg U_{1}}(\bar S)\cong {\rm Jac}(\bar S)$. Since, by \cite[Remark 3.5]{bobi}, the section $a_{p}$ has simple zeros, following  Remark \ref{divisorDbis}, we   write
$[a_{p}]=D+\overline{D},$ 
for $\overline{D}$ a positive divisor on $\Sigma$. Given the line bundle $U_{1}$ on $\bar S$, we define the line bundle $U_{2}$ also on $\bar S$ by
\[U_{2}= [D]\otimes U_{1}\otimes \rho^{*}K^{*}.\]
On $\bar S$ there is a natural rank $2$ Higgs bundle $(U_{1}\oplus U_{2},\bar \Phi)$ whose Higgs field is obtained via  multiplication by $\xi$, the tautological section of $\bar \rho^{*}K^{2}$, giving
\[(U_{1}\oplus U_{2})\rightarrow (U_{1}\oplus U_{2}) \otimes \bar \rho^{*}K^{2}.\] 
The corresponding spectral curve of this $K^{2}$-twisted Higgs bundle is the double cover $\pi:S\rightarrow \bar S$ and following the procedures of Chapter \ref{ch:complex}, there is a line bundle $M$ on $S$ which is preserved by the involution $\sigma$ on $S$ (see \cite{bobi} for details on spectral curves for these twisted Higgs bundles). Moreover, its direct image has the direct sum decomposition $\pi_{*}M=U_{1}\oplus U_{2}$ via the invariant and anti-invariant sections. Thus, one has the induced pair $(S,M)$, and by construction $\sigma$ acts as $-1$ on $M$ over the divisor $D$, whence proving the proposition. 
\end{proof}

\begin{remark}
By interchanging the involutions $\sigma$ and $-\sigma$, one can show that an equivalent correspondence exists in the case of $U(p,p)$-Higgs bundles for which $\deg V < \deg W$.
\end{remark}

Following \cite{brad}, we define the Toledo invariant $\tau(v,w)$ on $U(p,p)$-Higgs bundles as the number
\begin{eqnarray}\tau (v,w)=v-w.\label{toledoforUpp}\end{eqnarray}
By considering the previous calculations, this invariant  may be expressed as
$$\tau (v,w)=-\tilde{m}+2p(g-1).$$
\begin{rem}
By definition,  $\tilde{m}$  satisfies
$0\leq \tilde{m}\leq 4p(g-1)$, and thus
\[0\leq |\tau (v,w)|\leq 2p(g-1),\]
which agrees in the case of $U(p,p)$-Higgs bundles  with the general bounds given in \cite{brad} for the Toledo invariant.   
\end{rem}

One should note that $U(p,p)$-Higgs bundles $(V\oplus W, \Phi)$ for $\deg V = \deg W$ correspond to $\tilde m = 2p(g-1)$, and the case of $\deg V > \deg W$ (equivalently, $\deg V<\deg W$) corresponds to $0 \leq \tilde m <2p(g-1)$ (equivalently, $2p(g-1) < \tilde m \leq4p(g-1)$).

\section{\texorpdfstring{The spectral data of $SU(p,p)$-Higgs bundles}{The spectral data of SU(p,p)-Higgs bundles}}\label{sec:supp}


We shall consider in this section $SU(p,p)$-Higgs bundles $(E=V\oplus W,\Phi)$   over $\Sigma$ as given in Definition \ref{def:supp}.
In particular, we have that $v=-w$ and thus one can adapt Theorem \ref{teo11} for $SU(p,p)$-Higgs, to obtain the invariants  
 \begin{eqnarray}
m &=&(4p^{2}-2p)(g-1),\\
 \tilde{m}&=& 2 \deg W +2p(g-1).
 \end{eqnarray}

Note that in this case not all triples $(\bar S,U_{1}, D)$ satisfying the conditions of Theorem \ref{teo11} have a corresponding stable $SU(p,p)$-Higgs bundle: the restriction $                                                                                                                                              
            \Lambda^{p}V\cong \Lambda^{p}W^{*}                                                                                                                                                                 $ 
does not only constrain $m$ and $\tilde{m}$ but also the norm of the line bundles $U_{1}$ and $U_{2}$ on the quotient curve $\overline{S}$, and thus the divisor $D$. In the following subsections we study the implications of the isomorphism $\Lambda^{p}V\cong \Lambda^{p}W^{*}$.

\subsection{\texorpdfstring{The condition $\Lambda^{p}V\cong \Lambda^{p}W^{*}$}{The condition for SU(p,p)}}\label{modusu}
 
In terms of the line bundles $U_{1}$ and $U_{2}$ on $\overline{S}$, the  condition $\Lambda^{p}V\cong
\Lambda^{p}W^{*}$ may be written as
$\Lambda^{p}\overline{\rho}_{*}U_{1}\cong
\Lambda^{p}\overline{\rho}_{*}U_{2}^{*}.$ 
Considering the
norm map $Nm:{\rm Pic}(\overline{S})\rightarrow {\rm Pic}(\Sigma)$, the determinant bundles of $V$ and $W$ can be expressed as follows \cite[Section 4]{bobi}:
\begin{eqnarray}
 \Lambda^{p}V&=&\Lambda^{p}\overline{\rho}_{*}U_{i}= {\rm Nm}(U_{i})\otimes K^{-p(p-1)} ~{~\rm~for~}~i=1,2.
\end{eqnarray}
One should note that we are identifying divisors of $\Sigma$ and their
corresponding line bundles. 
The condition $\Lambda^{p}V\cong
\Lambda^{p}W^{*}$ is equivalent to requiring
\begin{eqnarray}{\rm Nm}(U_{1})= -{\rm Nm}(U_{2})+2p(p-1)K,\label{mapaUU}\end{eqnarray}
which is compatible with the degree calculations done in previous sections giving
$$\deg(U_{1})=-\deg(U_{2})+4p(p-1)(g-1).$$

In terms of divisors on $\Sigma$,  one has
$
 D={\rm Nm}~U_{1}^{*}+{\rm Nm}~ U_{2}+ pK.\label{uno}
$
Thus, from (\ref{mapaUU}), the condition for the Higgs bundle to be an $SU(p,p)$-Higgs bundle can be expressed as
\begin{eqnarray}
2{\rm Nm}~U_{1}=p(2p-1)K-D.\label{ecu1}
\end{eqnarray}
Equation (\ref{ecu1}) can be rewritten as
$ {\rm Nm}~([D]\otimes U^{2}_{1}\otimes\overline{\rho}^{*}K^{1-2p})=0. $ 
The choice of $U_{1}$ is thus determined by the choice of an element in $\ker {\rm Nm}$, i.e., in the Prym variety ${\rm Prym}(\overline{S},\Sigma)$, and thus one has the following description of the spectral data for $SU(p,p)$-Higgs bundles.

\begin{proposition}\label{teo22}
 The spectral data $(\bar S, U_{1}, D)$ of a $U(p,p)$-Higgs bundle as in Theorem \ref{teo11} corresponds to an $SU(p,p)$-Higgs bundle with $\deg V> \deg W$,  if and only if
 the line bundle $U_{1}$ and the divisor $D$ satisfy (\ref{ecu1}), i.e., 
$$
2{\rm Nm}~U_{1}=p(2p-1)K-D.\nonumber
$$
\end{proposition}

 \begin{remark}
Recall from Chapter \ref{ch:complex} that an $SU(p,p)$-Higgs bundle is a fixed point of the involution
$\Theta:~(E,\Phi)\mapsto (E,-\Phi)$
in the moduli space of $SL(2p, \mathbb{C})$-Higgs bundles. As such, the corresponding line bundle $M$ in the Jacobian variety of $S$ for which $\rho_{*}M=V\oplus W$ satisfies 
the Prym conditions\[ M\otimes \rho^{*}K^{(2p-1)/2}\in {\rm Prym}(S,\Sigma).\]
\end{remark}




\chapter{Spectral data for $Sp(2p,2p)$-Higgs bundles  }\label{ch:sppp}

In previous chapters we studied $G$-Higgs bundles from different points of view. In particular, in Chapter \ref{ch:supp} we looked at the spectral data corresponding to $U(p,p)$-Higgs bundles in the regular fibres of the Hitchin fibration. In this chapter we shall follow a similar approach and consider $Sp(2p,2p)$-Higgs bundles on a compact Riemann surface $\Sigma$ as fixed points of the involution
\[\Theta: (E,\Phi)\mapsto (E,\Phi^{\rm T}),\]
acting on  the moduli space $\mathcal{M}_{Sp(4p,\mathbb{C})}$ of $Sp(4p,\mathbb{C})$-Higgs bundles.  Moreover, $Sp(2p,2p)$-Higgs bundles can be seen as $SU(2p,2p)$-Higgs bundles but, as we shall see, they correspond to points in the singular fibres of the $SL(4p,\mathbb{C})$ Hitchin fibration
\[h:\mathcal{M}_{SL(4p,\mathbb{C})}\rightarrow \mathcal{A}_{SL(4p,\mathbb{C})},\]
 fixed by the involution
\[\Theta_{su}: (E,\Phi)\mapsto (E,-\Phi).\]
  Hence, by understanding the fixed point set  we obtain information about the singular fibres of the $SL(4p,\mathbb{C})$-Hitchin fibration.

We   begin the chapter recalling the main properties of $Sp(2p,2p)$-Higgs bundles and describing their different associated curves.
In Section \ref{sec:rank} we study the relation between $Sp(2p,2p)$-Higgs bundles and vector bundles of rank 2, and define the spectral data associated to stable $Sp(2p,2p)$-Higgs bundles. In Chapter \ref{ch:applications} we shall  relate the spectral data to a certain moduli space of parabolic vector bundles of rank 2 on a covering of $\Sigma$.

\section{\texorpdfstring{$Sp(2p,2p)$-Higgs bundles}{Sp(2p,2p)-Higgs bundles}}\label{sec:curv}

The symplectic Lie algebra $\mathfrak{sp}(4p,\mathbb{C})$ is given by the set of $4p\times 4p$ complex matrices $A$ that satisfy $J_{2p}A + A^{t}J_{2p} = 0 $ for
\[ J_{2p}=\left(
\begin{array}
 {cc}
0&I_{2p}\\
-I_{2p}&0
\end{array}
\right).\]
In this chapter we shall consider Higgs bundles for the real form $Sp(2p,2p)$ of the Lie group $Sp(4p,\mathbb{C})$, whose Lie algebra is \begin{small}
\[\mathfrak{sp}(2p,2p)=\left\{
\left(
\begin{array}{cccc}
 Z_{11}&Z_{12}&Z_{13}&Z_{14}\\
 \overline{Z}^{t}_{12}&Z_{22}&Z^{t}_{14}&Z_{24}\\
 -\overline{Z}_{13}& \overline{Z}_{14}& \overline{Z}_{11}& -\overline{Z}_{12}\\
 \overline{Z}^{t}_{14}& -\overline{Z}_{24}&-Z^{t}_{12}& \overline{Z}_{22}\\
\end{array}
\right)~\left|
\begin{array}{c}
 Z_{i,j}~ {~\rm ~ complex~matrices, } \\
Z_{11}, ~Z_{13}, Z_{12}, ~Z_{14} ~p\times p~{\rm matrices},\\
Z_{11}, ~Z_{22} ~{\rm skew ~Hermitian},\\
Z_{13}, ~Z_{24}~{\rm symmetric}
\end{array}\right.
\right\}.\]\end{small}
 
We have seen in Chapter \ref{ch:real} that $Sp(2p,2p)$-Higgs bundles are  fixed points of the involution
\[\Theta:~(E,\Phi)\mapsto (E, \Phi^{{\rm T}})\]
on $Sp(4p,\mathbb{C})$-Higgs bundles corresponding to vector bundles $E$ which have an endomorphism $f:E\rightarrow E$ conjugate to $\tilde{K}_{p,p}$, sending $\Phi$ to $\Phi^{{\rm T}}$ as in Section \ref{sec:sp2}, and whose $\pm 1$ eigenspaces are of dimension $2p$, where
\[ \tilde{K}_{p,p}=\left(
 \begin{array}
  {cccc}
 0&0&-I_{p}&0\\
 0&0&0&I_{p}\\
 I_{p}&0&0&0\\
0&-I_{p}&0&0
 \end{array}
\right).\] Concretely,   $Sp(2p,2p)$-Higgs bundles are defined as follows:
\begin{definition}\label{defsp}
 An $Sp(2p,2p)$-Higgs bundle is  a pair $(E,\Phi)$ where $E=V\oplus W$ for $V$ and $W$ rank $2p$ symplectic vector bundles, and where the Higgs field is  
\[\Phi=
\left(\begin{array}{cc}
0&\beta\\
\gamma&0
        \end{array}\right)~{\rm for~}~\left\{
\begin{array}{c}
\beta: W\rightarrow V\otimes K  \\
\gamma: V\rightarrow W \otimes K
\end{array} \right. ~{~\rm ~and ~}~\beta=-\gamma^{{\rm T}},\]
for $\gamma^{{\rm T}}$   the symplectic transpose of $\gamma$ as in Section \ref{sec:sp2}.
\end{definition}

Following the approach of Chapter \ref{ch:supp}, we shall study the different curves associated to an $Sp(2p,2p)$-Higgs bundle. 

\subsection{The associated curves}

Let $K$ be the canonical bundle of the compact Riemann surface $\Sigma$, and $X$ its total space with projection $\rho:X\rightarrow \Sigma$. Consider $(V\oplus W,\Phi)$ an $Sp(2p,2p)$-Higgs bundle as in Definition \ref{defsp}. Since $\Phi$ is an $Sp(4p,\mathbb{C})$-Higgs field, from Chapter \ref{ch:complex}, its characteristic polynomial is  
\[ {\rm det}(x-\Phi)= x^{4p}+x^{4p-2}a_{1}+x^{4p-4}a_{2}+\ldots+x^{2}a_{2p-1}+a_{2p}.\]
The above coefficients, which are given by polynomials in ${\rm tr}(\Phi^{i})$, define the spectral curve of the Higgs field, whose equation is
\begin{eqnarray} \mathcal{P}(\eta^{2}):=\eta^{4p}+\eta^{4p-2}a_{1}+\eta^{4p-4}a_{2}+\ldots+\eta^{2}a_{2p-1}+a_{2p}=0,\label{curva1}\end{eqnarray}
where $\eta$ is the tautological section of the pullback of $K$ to $X$, and $a_{i} \in H^{0}(\Sigma, K^{2i})$.
One may also consider the maps $\gamma^{\rm T}\gamma$ and $\gamma\gamma^{\rm T}$, where
\begin{eqnarray}
 \gamma\gamma^{\rm T}&:& W\rightarrow W  \otimes K^{2}~,\nonumber \\
\gamma^{\rm T}\gamma&:& V \rightarrow V\otimes K^{2}~.~\nonumber 
\end{eqnarray}
For $y=x^{2}$, the characteristic polynomial of $\gamma\gamma^{\rm T}$ and $\gamma^{\rm T}\gamma$  is  
\[   y^{2p}+y^{2p-1}a_{1}+y^{2p-2}a_{2}+\ldots+ya_{2p-1}+a_{2p}.\]
Hence, for $\xi:=\eta^{2}$, the coefficients of the above polynomial define the spectral curve of  $\gamma\gamma^{\rm T}$ and $\gamma^{\rm T}\gamma$, whose equation is  
\begin{eqnarray} \mathcal{P}(\xi):= \xi^{2p}+\xi^{2p-1}a_{1}+\xi^{2p-2}a_{2}+\ldots+\xi a_{2p-1}+a_{2p}=0,\label{curva2}\end{eqnarray}
and which lives in the total space of $K^{2}$. 
\begin{proposition}\label{propdiv2}
 The divisors defining the curves (\ref{curva1}) and (\ref{curva2}) have multiplicity 2.
\end{proposition}
\begin{proof} Let $\omega : V\cong V^{*}$ be the symplectic form on the vector bundle $V$. The composition of $\omega$ with the map $\gamma^{\rm T}\gamma:  V\otimes K^{-2}\rightarrow V  $ defines the section
\[\Gamma:~ V\otimes K^{-2}\rightarrow V^{*}, \]
which is skew symmetric. 
 As $\gamma^{\rm T}\gamma:  V\rightarrow V  \otimes K^{2}$ is a symplectically self adjoint operator,  
\begin{eqnarray}
 {\rm det}(\xi - \gamma^{\rm T}\gamma)&= &{\rm det}(\xi \omega^{-1}-\Gamma).\nonumber
\end{eqnarray}
Moreover, since  $\xi \omega^{-1}-\Gamma$  
is skew symmetric, its determinant 
is  the square of the corresponding Pfaffian and  $\xi - \gamma\gamma^{\rm T}$   has even-dimensional eigenspaces. Therefore,  
\[{\rm det}(\xi - \gamma^{\rm T}\gamma)={\rm Pf}^{2}(\xi \omega^{-1}- \Gamma):=P(\xi)^{2}.\]
The equation of the spectral curve (\ref{curva2}) associated to $\gamma^{\rm T}\gamma$ can be rewritten as
\[\left(\xi^{p}+b_{1}\xi^{p-1}+\ldots+b_{p-1}\xi+b_{p} \right)^{2}=0,\]
for $b_{i}\in H^{0}(\Sigma,K^{2i})$, and thus is reducible.  

The same argument can be made for $\gamma\gamma^{\rm T} $ and $\Phi$, and thus
  the spectral curve of $\Phi$ has equation
$ \left(\eta^{2p}+b_{1}\eta^{2p-2}+\ldots+b_{p-1}\eta^{2}+b_{p} \right)^{2}=0.$
\end{proof}

Consider a 2-dimensional generic eigenspace of $\gamma^{\rm T}\gamma$, and let $v_{i}$ for $i=1,2$ be a basis of  eigenvectors such that $\gamma^{\rm T}\gamma v_{i}=-\lambda^{2} v_{i}$, where $\lambda\neq0$. Then, for $w_{i}:=\gamma v_{i}/\lambda$ one has that 
\begin{eqnarray}
 \gamma\gamma^{\rm T} w_{i}=  \gamma\gamma^{\rm T}\gamma  v_{i}/\lambda
= -\gamma \lambda^{2}v_{i}/\lambda=-\lambda^{2}w_{i},\nonumber
\end{eqnarray}
and thus $w_{i}$ is an eigenvector of $\gamma\gamma^{\rm T}$ with eigenvalue $-\lambda^{2}$. Moreover, the vector $(v_{i},w_{i})$ satisfies
\begin{eqnarray}
\left(\begin{array}
       {cc}
0 &-\gamma^{\rm T}\\
\gamma&0  
    \end{array}
\right)
\left(\begin{array}
       {c}
v_{i}\\
w_{i}
\end{array}
\right)=
\left(\begin{array}
       {c}
-\gamma^{\rm T}w_{i}\\
\gamma v_{i}
\end{array}
\right)
=
\left(\begin{array}
       {c}
-\gamma^{\rm T}\gamma v_{i}/\lambda\\
\gamma v_{i}
\end{array}
\right)= \lambda \left(\begin{array}
       {c}
v_{i}\\
w_{i}
\end{array}
\right). \nonumber
\end{eqnarray}
Hence, $(v_{i},w_{i})$  is a $\lambda$-eigenvector of $\Phi$ for $i=1,2$.

\begin{remark} For $i=1,2$, the vectors  $(v_{i},-w_{i})$ span an eigenspace of $\Phi$ with eigenvalue $-\lambda$. Hence, there is an isomorphism  between the 2-dimensional eigenspaces of $\Phi$ for eigenvalues $+\lambda$ and  $-\lambda$, given by $(v_{i},w_{i})\mapsto (v_{i},-w_{i})$. From here one can see, again,  that the characteristic polynomial of $\Phi$ is invariant under the map $\eta\mapsto-\eta $.\label{remarketa}
\end{remark}

\begin{definition}
We denote by $\rho:S\rightarrow \Sigma$  the $2p$-fold cover of $\Sigma$ in the total space of $K$ whose equation is  
\begin{eqnarray}\label{ecuS}
p(\eta^{2}):= \eta^{2p}+b_{1}\eta^{2p-2}+\ldots+b_{p-1}\eta^{2}+b_{p}=0.
\end{eqnarray}
\end{definition}
Note that (\ref{ecuS}) defines a divisor of $\rho^{*}K^{2p}$, and since a cotangent bundle has trivial canonical bundle it follows that $K_{S}\cong \rho^{*}K^{2p}$. Hence, 
\[g_{S}=4p^{2}(g-1)+1.\]

\begin{definition}
We denote by $\overline{\rho}:\overline{S}\rightarrow \Sigma$ the $p$-fold cover of $\Sigma$ in the total space of $K^{2}$ whose equation is
\begin{eqnarray}
  p(\xi):=\xi^{p}+b_{1}\xi^{p-1}+\ldots+b_{p-1}\xi+b_{p}=0,
\end{eqnarray}
where $b_{i}\in H^{0}(\Sigma,K^{2i})$. In particular, $a_{2p}=b^{2}_{p}$. 
\end{definition}
The characteristic polynomials $\mathcal{P}(\eta^{2})$ and $\mathcal{P}(\xi)$ of $\Phi$ and $\gamma\gamma^{\rm T}$, respectively, can be written as
\[\mathcal{P}(\eta^{2})=p(\eta^{2})^{2}~{~\rm~and~}~\mathcal{P}(\xi)=p(\xi)^{2}.\]
  The curve $S$ has a natural involution $\sigma(\eta)=-\eta$, and taking the quotient of $S$ by the action of $\sigma$ one obtains 
the  curve $\overline{S}$. In this case, since the canonical bundle of $K^2$ is $\bar \rho^*K^{-1}$ the adjunction formula gives
\[K_{\bar S}\cong  \bar \rho^{*} K^{2p} \otimes \bar \rho^*K^{-1} ,\]
and thus one has that
\[g_{\overline{S}}= (2p^{2}-p)(g-1)+1.\]
We have defined the following coverings of the compact Riemann surface:
\begin{eqnarray}
 \xymatrix{
~& S\ar[rr]^{2:1}_{\pi}\ar[dr]_{2p:1}^{\rho}&&\overline{S}\ar[dl]_{\overline{\rho}}^{p:1} &\\
&&\Sigma&&} \nonumber
\end{eqnarray}

From Proposition \ref{propdiv2}, all $Sp(2p,2p)$-Higgs bundles have singular spectral curves, i.e., have reducible characteristic polynomial ${\rm det}(\rho^{*}\Phi-\eta)=p(\eta^{2})^{2}$. Moreover,  the map $\rho^{*}\Phi-\eta$ has 2-dimensional cokernel. Thus, whilst in previous chapters we considered line bundles on spectral curves, in this chapter we shall look at rank 2 vector bundles on the curve $S$, which we  require to be smooth.

 \section{\texorpdfstring{The spectral data of $Sp(2p,2p)$-Higgs bundles}{The spectral data of Sp(2p,2p)-Higgs bundles}}\label{sec:rank}

In this section we shall analyse the relation between $Sp(2p,2p)$-Higgs bundles and rank 2 vector bundles on the smooth curves $S$ and $\overline{S}$. 
Note that a rank $2$ vector bundle $M$ on $S$ which is preserved by the involution $\sigma$ has an induced action of ${\rm det } \sigma$ on $\Lambda^{2}M$ over the fixed points. If  $\Lambda^{2}M$ is a pullback from $\bar S$, then there are only two induced actions on the bundle: the trivial action and the opposite one.  The former case corresponds to  $\sigma$ having equal eigenvalues on $M$ over all fixed points, and the latter case corresponds to $\sigma$ having different eigenvalues over all fixed points. 

\begin{proposition}\label{propsymp1}
 Let $M$ be a vector bundle on the $2p$-fold cover $\rho:S\rightarrow \Sigma$ which is preserved by the involution $\sigma$, and whose determinant bundle is given by
${\rm det} M \cong \rho^{*}K^{2p-1}.$
 If the involution $\sigma$ acts with eigenvalues $+1$ and $-1$ on $M$ over all fixed points, then 
$\rho_{*}M=V\oplus W,$
for $V$ and $W$ rank $2p$ symplectic vector bundles on $\Sigma$.
\end{proposition}

\begin{proof}
 Consider $M$ a rank 2 vector bundle on $\rho:S\rightarrow \Sigma$ for which $\Lambda^{2}M\cong \rho^{*}K^{2p-1}$. By the relative duality theorem \cite{bobi}, one has that
\begin{eqnarray}
 \rho_{*}(M)^{*}&\cong&\rho_{*}(K_{S}\otimes \rho^{*}K^{-1}\otimes M^{*})\nonumber\\
&\cong&\rho_{*}( \rho^{*}K^{2p-1}\otimes M^{*})\nonumber\\
&\cong&\rho_{*}( \rho^{*}K^{2p-1}\otimes M \otimes \Lambda^{2}M^{*} )\nonumber\\
&\cong&\rho_{*}(M).\nonumber
\end{eqnarray}
Similarly, taking the direct image in the 2-fold cover $\pi:S\rightarrow \bar S$ one has that
\begin{eqnarray}\pi_{*}(M)^{*}&\cong& \pi_{*}(K_{S}\otimes \pi^{*}K_{\bar S}^{-1}\otimes M^{*})\nonumber\\
&\cong& \pi_{*}(  \rho^{*}K^{2p}\otimes \pi^{*}K_{\bar S}^{-1}\otimes M^{*})\nonumber\\
&\cong& \pi_{*}(  \rho^{*}K^{2p}\otimes \pi^{*}\overline{\rho}^{*}K^{-2p+1}\otimes M^{*})\nonumber\\
&\cong& \pi_{*}(  \rho^{*}K \otimes M^{*})\nonumber\\
&\cong& \pi_{*}(M)\otimes \bar\rho^{*}K^{-2p+2}.\nonumber\end{eqnarray}
This duality can be described explicitly in the following way. The $2$-fold covering $\pi$ defines a map
\[d\pi:K^{-1}_{S}\rightarrow \pi^{*}K_{\bar S}^{-1},\]
or equivalently, a  section of $K_{  S} \otimes \pi^{*}K_{\bar S}^{-1}$. At a generic point $x$ in $\bar S$,  for $s$ and $f$  sections of $M$ and $K_{S}\otimes \pi^{*}K_{\bar S}^{-1}\otimes M^{*} $, respectively, there is a natural pairing 
 \begin{eqnarray}\sum_{y\in \pi^{-1}(x)}\frac{<s,f>_{y}}{d\pi}, \label{parformula}\end{eqnarray}
where $<~,~>$ is the canonical  pairing between $M$ and $M^{*}\otimes \rho^{*}K$ with values on $\rho^{*}K$. 

In the case of the double covering $\pi:S\rightarrow \bar S$,  the duality in the proximity of a ramification point can be seen following a similar approach to Example \ref{newexample}. Indeed, the fixed points of the involution $\sigma$ acting on $S$ give the ramification points of the double cover $\pi:S\rightarrow \bar S$. For $\mathcal{U}$ and  $\overline{\mathcal{U}}$ local neighbourhoods of $S$ and $\bar S$, respectively, and $z$ a local coordinate near a ramification point, the cover is given by \begin{eqnarray} \mathcal{U}&\rightarrow& \overline{\mathcal{U}}\nonumber\\
      z &\mapsto &z^{2}:=w.                      \nonumber                                                                                                                                                                                                                                                                                                                                                                                                                                                                                                                                                                                                                                                                                                                                                                                             \end{eqnarray}
In a neighbourhood of a ramification point, a section of $M$ looks like
\[s_{z}(w)=h_{0}(w)+ zh_{1}(w),\] 
and a  section of $\rho^{*}K\otimes M^{*}$ looks like
\[f_{z}(w)=\zeta_{0}(w)+ z\zeta_{1}(w),\]
where $h_{0},h_{1},\zeta_{0}$ and $\zeta_{1}$ are vector valued. Thus, the space of sections of the direct image $\pi_{*}M$ is the direct sum of 2 copies of the vector valued holomorphic functions on $\mathcal{U}$. We shall denote by $s_{-z}(w)$ and $f_{-z}(w)$ the functions for $-z$.
Then, for $s(w)=(s_{z}(w),s_{-z}(w))$ and $f(w)=(f_{z}(w),f_{-z}(w))$, a concrete expression of the duality theorem in a neighbourhood of $w_{0}$ is given by
\begin{eqnarray}
 <s(w),f(w)>_{\bar S}&=&\sum_{i=+z,-z}\frac{<s_{i}(w),f_{i}(w)>}{d\pi|_{ i}}\nonumber\\
&=& \frac{<\zeta_{0}(w)+z\zeta_{1}(w),h_{0}(w)+zh_{1}(w)>}{2z}\nonumber \\~&~&~+\frac{<\zeta_{0}(w)-z\zeta_{1}(w),h_{0}(w)-zh_{1}(w)>}{-2z} \nonumber\\
&=& <\zeta_{1}(w),h_{0}(w)>+<\zeta_{0}(w),h_{1}(w)>.\label{formula}
\end{eqnarray}
Since $M$ is preserved by the involution $\sigma$, following the ideas of Chapter \ref{ch:supp}, its direct image can be expressed as direct sum  defined via the invariant and anti-invariant sections
\begin{eqnarray}
 \pi_{*}(M)=M_{V}\oplus M_{W}. 
\end{eqnarray}
Without loss of generality we let  $M_{V}$ correspond to the invariant sections of $M$, and define
$\overline{\rho}_{*}M_{V}=V$ and $\overline{\rho}_{*}M_{W}=W.$ Since $\rho=\overline{\rho}\pi$, we have 
$\rho_{*}M=V\oplus W. $
Suppose that  $\sigma$ acted at a fixed point only with $+1$ eigenvalues. Let $h_{0}(0)=e_{i}$ and $\zeta_{0}(0)=e_{j}$, for
$e_{1}=(1,0)$ and $e_{2}=(0,1)$. Since
\[M_{V}\oplus M_{W}\cong (M_{V}^{*}\oplus M_{W}^{*})\otimes \bar \rho^{*}K^{2p-2},\]
the pairing between the factors can be seen as follows.
Recall that $\sigma$ sends $z\mapsto -z$, and so the trivial action at a fixed point would give
\begin{eqnarray} e_{i} +ze_{j}  &\mapsto& e_{i}-ze_{j},\end{eqnarray}
for $i,j=1,2$. This implies, in particular,  that the pairing $<~,~>_{\bar S}$ is degenerate on the invariant sections defining $V$. Indeed, the invariant sections would be given by $h_{1}(0)=0$ and $\zeta_{1}(0)=0$, which would make the pairing (\ref{formula}) vanish. 

On the other hand, if the action of $\sigma$ has different eigenvalues over the fixed points, say $+1$ on $(0,1)$ and $-1$ on $(1,0)$, then the invariant sections of $M$ are spanned by $(0,1)$ and $z(1,0)$  (and similarly for the invariant sections of $\rho^{*}K\otimes M^{*}$). Hence, in this case from (\ref{formula}) the pairing $<~,~>_{\bar S}$ is non degenerate over the invariant sections, where for constants $c_{j}$ for $j=1,\ldots ,4$  we have
\begin{eqnarray}
<s(0),f(0)>_{\bar S}&=& <c_{1}(1,0),c_{2}(0,1)>+<c_{3}(0,1),c_{4}(1,0)>\nonumber\\
&=& c_{1}c_{2}-c_{3}c_{4}.
\end{eqnarray}
 Similarly, from (\ref{formula}), one can see that in this case the pairing is non-degenerate over the anti-invariant sections, which are spanned by $(1,0)$ and $z(0,1)$. Therefore, the bracket $<~,~>_{\bar S}$ pairs $M_V$ (and $M_{W}$) non degenerately with itself
 only if the involution $\sigma$ acts with different eigenvalues over a fixed point.

If $\sigma$ acts with different eigenvalues, the non degenerate pairing induces skew isomorphisms  
\[M_{V}\cong M^{*}_{V}\otimes \bar\rho^{*}K^{2p-2}~{~\rm~and~}~M_{W}\cong M^{*}_{W}\otimes \bar\rho^{*}K^{2p-2}.\] 
By the relative duality theorem, this implies that $V$ and $W$ are isomorphic to their duals. Moreover, the isomorphisms are symplectic. Indeed, since, at a generic point, $V$ is a direct sum of $p$ copies of $M_{V}$, for each $i=1,\ldots,p$ we have the skew symmetric pairing $<~,~>^{i}_{\bar S}$, and thus we define the natural skew symmetric pairing on the direct image $V=\bar\rho_{*}M_{V}$  given by
\[\sum_{i=1}^{p}\frac{<~,~>^{i}_{\bar S}}{d\bar \rho}.\]
The symplectic structure on $W$ is defined in a similar fashion. 
\end{proof}

A rank 2 vector bundle $M$  satisfying Proposition \ref{propsymp1} defines a symplectic Higgs field on $\Sigma$ in the following way. Consider the tautological section $\eta$ of $\rho^{*}K$, and the induced multiplication map
\begin{eqnarray}
 \xymatrix{H^{0} (\rho^{-1}(\mathcal{U}), M) \ar[r]^{\eta} &H^{0}(\rho^{-1}(\mathcal{U}), M\otimes \rho^{*}K),}
\end{eqnarray}
for $\mathcal{U}$ an open set in $\Sigma$.
By the definition of direct image, this map gives the Higgs field $\Phi: \rho_{*}M \rightarrow \rho_{*}M \otimes K$. Since $V$ and $W$ are the images of invariant and anti-invariant sections respectively, $\Phi$ is a symplectic Higgs field which maps $V\mapsto W\otimes K$ and $W\mapsto V\otimes K$.

Conversely, starting with a symplectic $Sp(2p,2p)$-Higgs bundle one can define a rank 2 vector bundle on the associated $2p$-fold cover $\rho:S\rightarrow \Sigma$, such that it is preserved by the natural involution on $S$, and it is acted on by the involution with different eigenvalues over the fixed points: 

\begin{proposition}\label{propsymp2}
Each  $Sp(2p,2p)$-Higgs bundle on $\Sigma$ with smooth curve $S$, defines a rank $2$ vector bundle on $S$ of determinant $\rho^{*}K^{2p-1}$ which is preserved by the involution $\sigma$ and acted by as $+1$ and $-1$ over all fixed points. 
\end{proposition} 

\begin{proof}
 
Consider the map in $S$ induced by the $Sp(2p,2p)$-Higgs bundle  
 \[\rho^{*}\Phi: \rho^{*}E\rightarrow \rho^{*} E\otimes \rho^{*}K.\]
As seen previously, there is a canonically defined vector bundle of rank 2 on $S$ given by
\[M := {\rm coker}( \rho^{*}\Phi -\eta).\]
Since the ramification divisor of $\rho:S\rightarrow \Sigma$ is a section of $\rho^{*}K^{2p-1}$, from \cite{bobi},
the vector bundle $M$ fits in the exact sequence
\begin{equation}
 \xymatrix{0\ar[r]& M\otimes \rho^{*}K^{-2p+1}\ar[r]& \rho^{*}E  \ar[r]^{\rho^{*}\Phi-\eta}&  \rho^{*}(E\otimes K)\ar[r]& M \otimes \rho^{*}K\ar[r]&0.}\label{Sec2}
\end{equation} 
Dialysing the above sequence and tensoring by $\rho^{*}K$, via the identification with the symplectic form one has
\begin{equation}
 \xymatrix{0\ar[r]& M^{*}\ar[r]&   \rho^{*}E  \ar[r]^{\rho^{*}\Phi+\eta} & \rho^{*}E \otimes \rho^{*}K \ar[r]& \ar[r]M^{*} \otimes \rho^{*}K^{2p}&0.}\label{Sec3}\nonumber
\end{equation}
Since $\sigma:\eta\mapsto -\eta$, the vector bundle $ M^{*}$ for eigenvalue $-\eta$ is transformed to $$ \sigma^{*}M^{*}\cong M\otimes \rho^{*}K^{-2p+1}$$ for eigenvalue $\eta$. By Remark \ref{remarketa} there is an isomorphism between these two spaces and thus
$\sigma^{*}M\cong M,$
i.e., the vector bundle $M$ is preserved by the involution $\sigma$, and
\[M\cong M^{*}\otimes \rho^{*}K^{2p-1}.\]
Thus, $\Lambda^{2}M=\rho^{*}K^{2p-1}$, as in the hypothesis of Proposition \ref{propsymp1}. 

We shall now understand how the involution $\sigma$ acts on $M$ at a fixed point. 
Let $w$ be a local coordinate on $\Sigma$ where at $w=0$ the Higgs field is singular. Since $S$ is smooth, $\gamma$ drops rank by 1 and we can find bases of $V$ and $W$ such that $\gamma$ is 
\begin{eqnarray}
\gamma= \left(
\begin{array}
 {cccccc}
0&w&0& \ldots&0\\
1&0&0& \ldots&0\\
0&0&1& \ldots&0\\
\vdots&\vdots&\vdots& \ddots&\vdots\\
0&0&0& 0&1\\
\end{array}\right).
\end{eqnarray}
This leads to a local form for the Higgs field  $\Phi$ which is given by a non-singular component times
\begin{eqnarray}
 \left(
\begin{array}
 {cccc}
0&0&0&w\\
0&0&1&0\\
0&w&0&0\\
1&0&0&0
\end{array}
\right).
\end{eqnarray}
 Without loss of generality, we shall let $\Phi$ equal the above form, in which case a local coordinate $\eta$ on $S$ corresponds to $w= \eta^{2}$.
The  rank 2 vector bundle $M$ on $S$ may be expressed as  
$M={\rm coker}(\rho^{*}\Phi-\eta)={\rm ker}(\rho^{*} \Phi-\eta)^{t}.$

Following the approach of Proposition \ref{propsymp1}, consider   $v_1=(1,0)$ and $v_{2}=(0,1)$ in the $-\eta^{2}$-eigenspace of $\gamma^{\rm T} \gamma$ and  let  
\begin{eqnarray}w_{1}&:=&\gamma v_1/\eta.=\left(\begin{array}
                                {cc} 0&w\\
1&0
                               \end{array}
 \right)\left(\begin{array}
                                {c} 1\\
0
                               \end{array}
 \right)  \frac{1}{\eta}= \left(\begin{array}
                                {c} 0 \\
1/\eta 
                               \end{array}
 \right), \\
w_{2}&:=&\gamma v_2/\eta.=\left(\begin{array}
                                {cc} 0&w\\
1&0
                               \end{array}
 \right)\left(\begin{array}
                                {c} 0\\
1
                               \end{array}
 \right)  \frac{1}{\eta}= \left(\begin{array}
                                {c} \eta \\
0
                               \end{array}
 \right). \end{eqnarray}  Then, as seen previously,  the $\eta$-eigenspace of $\Phi$ is spanned by $x_{1}:=(v_{1},w_{1})=(1,0,0,1/\eta)$ and $x_{2}:=(v_{2},w_{2})=(0,1,\eta,0)$, whilst the $-\eta$-eigenspace of $\Phi$ is spanned by the vectors $y_{1}:=(v_{1},-w_{1})=(1,0,0,-1/\eta)$ and $y_{2}:=(v_{2},-w_{2})=(0,1,-\eta,0)$.
 As $\eta$ goes to zero the eigenspace is spanned by  $(0,0,0,1)$ and $(0,1,0,0)$ in the first case, and   by $(0,0,0,-1)$ and $(0,1,0,0)$ in the second. By Remark \ref{remarketa} the eigenspaces are identified via an isomorphism $Ax_{1}=y_{1}$ and $Ax_{2}=y_{2}$, and thus $A(\eta x_{1})=\eta y_{1}$. As $\eta$ goes to zero, the two spaces coincide and $\eta x_{1}\rightarrow (0,0,0,1)$ and $\eta y_{1}\rightarrow -(0,0,0,1)$. Hence, the limit of the isomorphism between the two eigenspaces is $-1$ on the first factor and $+1$ on the second, i.e., the action of $\sigma$ has different eigenvalues, proving the proposition.\end{proof}

\begin{remark} If a rank 2 vector bundle $M$ on $S$ preserved by $\sigma$ and with determinant $\rho^{*}K^{2p-1}$ is acted by the involution as $+1$ over all fixed points, or $-1$ over all fixed points, then its direct image  would not induce a symplectic  Higgs bundle. In fact,  the map $\Phi$ in this case is given by
 \begin{eqnarray}
\Phi=\left(\begin{array}
   {cccc}
0&0&w&0\\
0&0&0&w\\
1&0&0&0\\
0&1&0&0
  \end{array}
\right),\label{pi1}
 \end{eqnarray}
which is not of the correct form since if $\gamma$ becomes singular so should $\gamma^T$.

 \end{remark}

\subsection{Stability conditions}

The relation between the stability of the rank 2 vector bundles on $S$ introduced in Proposition \ref{propsymp1} and Proposition \ref{propsymp2}, and stability of the corresponding $Sp(2p,2p)$-Higgs bundles on $\Sigma$ is given as follows.
\begin{proposition}
 An $Sp(2p,2p)$-Higgs bundle on $\Sigma$ with smooth curve $S$ is stable if and only if its corresponding rank 2 vector bundle $M$ on $S$ is stable.
\end{proposition}
\begin{proof}Under the above hypothesis, the square root of the characteristic polynomial of an $Sp(2p,2p)$-Higgs bundle $(E:=\rho_{*}M,\Phi)$  is an irreducible polynomial $p(\eta^{2})$ defining the smooth curve $S$. Since the characteristic polynomial of $(E,\Phi)$ restricted to any invariant subbundle must divide  ${\rm det}(\eta-\rho^{*}\Phi)=p(\eta^{2})^{2}$,  the only invariant subbundles of $E$ are  rank $2p$ vector bundles $F$ corresponding to the factor $p(\eta^{2})$. 

Since there are no invariant subbundles of $F$, the restriction $\Phi_{F}$ of the Higgs field to $F$ gives a rank $2p$ classical  Higgs bundle $(F,\Phi_{F})$ which is stable. Moreover, by construction its spectral curve is given by the equation $p(\eta^{2})=0$, i.e., by the $2p$-fold cover $S$. From Chapter \ref{ch:complex}, there is a line bundle $L$ on $S$ for which $\rho_{*}L=F$. Since $F\subset E=\rho_{*}M$, the line bundle $L$ is a subbundle of the rank 2 vector bundle $M$.
Hence, for $M$ such that  $\rho_{*}M=E$, the invariant subbundles of $E$ correspond to line subbundles $L\subset M$ on $S$.

The vector bundle $M$ is stable if and only if for all line subbundles $L\subset M$ one has that $\mu(L)<\mu(M)$, or equivalently, that
\begin{eqnarray}
 {\rm deg} L < (2p^{2}-p)(2g-2)\nonumber.
\end{eqnarray}
Hence, $M$ is stable if and only if the direct image of invariant subbundles satisfies
\begin{eqnarray}
 {\rm deg} \rho_{*}L &=&  {\rm deg} L + (1-g_{S})-2p(1-g)\nonumber\\
&=&  {\rm deg} L +(2p-4p^{2})(g-1)<0.\nonumber
\end{eqnarray}
Since the vector bundle $\rho_{*}M$ has degree zero, if $M$ is stable, the direct image $\rho_{*}L$ does not make  the Higgs bundle unstable. Conversely, if the  Higgs bundle is stable, all its subbundles are of the form $\rho_{*}L$ for $L\subset M$, and satisfy $ {\rm deg} \rho_{*}L<0$. Hence, in this case one has  ${\rm deg} L < (2p^{2}-p)(2g-2)$ and thus $M$ is stable. 
Therefore, 
the vector bundle $M$ on $S$ is stable if and only if the induced Higgs pair $(E:=\rho_{*}M,\Phi)$ on $\Sigma$ is stable.
\end{proof}

The spectral data associated to $Sp(2p,2p)$-Higgs bundles is thus described as follows.

\begin{teo} \label{T0} 
Each stable  $Sp(2p,2p)$-Higgs bundle $ (E=V\oplus W, \Phi)$ on a compact Riemann surface $\Sigma$ of genus $g\geq 2$ for which $p(\eta^{2})=0$ defines a smooth curve, has an associated pair $(S,M)$ where 
\begin{enumerate}
 \item[(a)] the curve $\rho:S\rightarrow \Sigma$ is a smooth $2p$-fold cover of $\Sigma$ given by the equation
\[\eta^{2p}+b_{1}\eta^{2p-2}+\ldots+b_{p-1}\eta^{2}+b_{p}=0,\]
in the total space of $K$,  where $b_{i}\in H^{0}(\Sigma,K^{2i})$, and $\eta$ is the tautological section of $\rho^{*}K$. The curve $S$ has a natural involution $\sigma$ acting by $\eta \mapsto -\eta$;
 \item[(b)]  the vector bundle $M$ is a rank 2  vector bundle on the smooth curve $\rho:S\rightarrow \Sigma$ with determinant bundle $\Lambda^{2}M\cong\rho^{*}K^{-2p+1}$, and such that $\sigma^{*}M\cong M$.
Over the fixed points of the involution, the vector bundle $M$ is acted on by $\sigma$ with eigenvalues $+1$ and $-1$.
\end{enumerate}
Conversely, a pair $(S,M)$ satisfying (a) and (b) induces a stable  $Sp(2p,2p)$-Higgs bundle $ (\rho_{*} M=V\oplus W, \Phi)$  on the compact Riemann surface $\Sigma$.
\end{teo}

\subsection{Dimensional calculations via the spectral data}\label{sec:dim}

Through dimensional calculations, we shall show here that the space of $Sp(2p,2p)$-Higgs bundles satisfying the hypothesis of Theorem \ref{T0} sits inside the moduli space $\mathcal{M}_{Sp(2p,2p)}$ of  $Sp(2p,2p)$-Higgs bundles as a Zariski open set.  
Recall that the dimension of the space of stable $Sp(4p,\mathbb{C})$-Higgs bundles is 
\[{\rm dim}\mathcal{M}_{Sp(4p,\mathbb{C})}= {\rm dim}Sp(4p,\mathbb{C})(g-1)=4p(4p+1)(g-1).\]

\begin{remark}\label{spremark}
Since the   dimension of the real group $Sp(2p,2p)$ is $2p(4p+1)$,  the expected dimension of  the moduli space of   $Sp(2p,2p)$-Higgs bundles is
\[{\rm dim}\mathcal{M}_{Sp(2p,2p)}=2p(4p+1)(g-1).\]
\end{remark}
Let $\mathcal{N}$ be the moduli space of semi-stable rank $2$ vector bundles on $S$ with fixed determinant bundle $\rho^{*}K^{2p-1}$, and $V$ a representative of a point in $\mathcal{N}$. Then,   the tangent space of $\mathcal{N}$ at $[V]$ is  
\[T\mathcal{N}_{[V]} \cong H^{1}(S,{\rm End}_{0}V).\]
Note that   the dimension of  $\mathcal{N}$ is
$
 {\rm dim}\mathcal{N}=  3g_{S}-3
= 12p^{2}(g-1)
$. 
On $\mathcal{N}$ there is an induced action of $\sigma$, and hence one may consider $\mathcal{N}^{\sigma}$, the subspace of fixed isomorphism classes of vector bundles. From \cite[Section 2]{AG}, any  point in the moduli space  $\mathcal{N}$ fixed by $\sigma$ can be represented by a semi-stable bundle with an action of $\sigma$, covering the action on $S$. Moreover, since the action of $\sigma$ extends to an action $\tilde{\sigma}$ on the tangent space $T\mathcal{N}$, one has
\[{\rm dim} \mathcal{N}^{\sigma} = {\rm dim} H^{1}(S,{\rm End}_{0}V)^{+},\]
for $ H^{1}(S,{\rm End}_{0}V)^{+}$ the fixed set of $\tilde{\sigma}$.
\begin{proposition}
  The dimension of the parameter space of  $Sp(2p,2p)$-Higgs bundles as in Theorem \ref{T0} is, as expected, $2p(4p+1)(g-1)$.
\end{proposition}
\begin{proof}
Recall from Theorem \ref{T0} that $Sp(2p,2p)$-Higgs bundles satisfying the hypothesis have an associated pair $(S,M)$. Hence, via the spectral data $(S,M)$ it is sufficient to show that the dimension of the space $\mathcal{N}^{\sigma}$ of rank 2 vector bundles preserved by the involution $\sigma$, together with the dimension of the parameter space defining $S$ sum up to $2p(4p+1)(g-1)$. 
From its construction, the curve $S$ is defined by the coefficients $b_{i}\in H^{0}(\Sigma, K^{2i})$, and thus the dimension of the parameter space is given by  
 \begin{eqnarray}
 {\rm dim} \left(\bigoplus_{i=1}^{p} H^{0}(\Sigma, K^{2i})\right)
 &=&\sum_{i=1}^{p} {\rm dim} \left(  H^{0}(\Sigma, K^{2i})\right)\nonumber\\
 &=& \sum_{i=1}^{p} (4i-1)(g-1)\nonumber\\
 &=& (2p^{2}+p)(g-1).\nonumber
 \end{eqnarray}
Hence, to prove the proposition one needs to show that the dimension of $\mathcal{N}^{\sigma}$ satisfies
\begin{eqnarray}
 {\rm dim}(\mathcal{N}^{\sigma})
&=&{\rm dim}\mathcal{M}_{Sp(2p,2p)}- (2p^{2}+p)(g-1)
\nonumber\\&=&2p(4p+1)(g-1)-(2p^{2}+p)(g-1)\nonumber\\
&=&(8p^{2}+2p-2p^{2}-p)(g-1)\nonumber \\
&=& (6p^{2}+p)(g-1)\nonumber
\end{eqnarray}
 Since from Theorem \ref{T0} the  action of $\sigma$ on $V$ maps $(v_{1},v_{2})\mapsto (v_{1},-v_{2})$,  the induced action on  ${\rm End}_{0}V$  is given by
\begin{eqnarray}
 \left(\begin{array}
{cc}
        1&0\\
0&-1
       \end{array}
\right) \left(\begin{array}
{cc}
        a&b\\
c&-a
       \end{array}
\right) \left(\begin{array}
{cc}
                1&0\\
0&-1
       \end{array}
\right)=
 \left(\begin{array}
{cc}
        a&-b\\
-c&-a
       \end{array}
\right),\nonumber
\end{eqnarray}
and thus it can be expressed in a suitable basis as
\begin{eqnarray}
\tilde{\sigma}:= \left(\begin{array}
        {ccc}
1&0&0\\
0&-1&0\\
0&0&-1
       \end{array}
\right).
\end{eqnarray}
By means of $\tilde{\sigma}$ and its action on $H^{1}(S,{\rm End}_{0}V)$, we shall study the dimension $$ {\rm dim} \mathcal{N}^{\sigma} ={\rm dim} H^{1}(S,{\rm End}_{0}V)^{+}$$
through the Lefschetz number $ L_{\tilde{\sigma}}$ associated to $\tilde{\sigma}$ on ${\rm End}_{0}V$ as defined in \cite{at}. By \cite[Eq. 4.11]{at},  
the number  $ L_{\tilde{\sigma}}$ is given by
 \[L_{\tilde{\sigma}} = {\rm trace} ~\tilde{\sigma}|_{ H^{0}(S,{\rm End}_{0}(V))} - {\rm trace} ~\tilde{\sigma}|_{ H^{1}(S,{\rm End}_{0}(V))}.\]
  By stability,
${\rm dim}H^{0}(S,{\rm End}_{0}V)=0$ and hence 
 \begin{eqnarray}
   L_{\tilde{\sigma}} &=&  - {\rm trace} ~\tilde{\sigma} |_{H^{1}(S,{\rm End}_{0}(V)) }\nonumber\\
&=&-{\rm dim}H^{1}(S,{\rm End}_{0}(V))^{+}+{\rm dim}H^{1}(S,{\rm End}_{0}(V))^{-}.\nonumber
 \end{eqnarray}
Furthermore, by \cite[Theorem 4.12]{at} we have that
\begin{eqnarray}
L_{\tilde{\sigma}}&=& \sum_{p~fixed}  \frac{{\rm trace} ~\tilde{\sigma}|_{{\rm End}_{0}V_{p}} }{{\rm det}(1-d\sigma)}=  \sum_{p~fixed}  \frac{{\rm trace} ~\tilde{\sigma}|_{{\rm End}_{0}V_{p}}}{2}.\nonumber   
\end{eqnarray}
Since ${\rm trace} ~\tilde{\sigma}|_{{\rm End}_{0}V_{p}}=-1$ over all $4p(g-1)$ fixed points, one has that
\[L_{\tilde{\sigma}}=-2p(g-1).\]
This implies that
\[-{\rm dim}H^{1}(S,{\rm End}_{0}(V))^{+}+{\rm dim}H^{1}(S,{\rm End}_{0}(V))^{-}=-2p(g-1).\]
By Riemann-Roch and Serre duality  
\[{\rm dim}H^{1}(S,{\rm End}_{0}(V))^{+}+{\rm dim}H^{1}(S,{\rm End}_{0}(V))^{-}= 12p^{2}(g-1),\]
Thus, the dimension of $\mathcal{N}^{\sigma}$ is given by
\[{\rm dim}(\mathcal{N}^{\sigma})={\rm dim}H^{1}(S,{\rm End}_{0}(V))^{+}=(6p^{2}+p)(g-1),\]
whence the result follows.\end{proof}

Since the locus in $\mathcal{A}_{Sp(4p,\mathbb{C})}$ defining singular curves $S$ is given by algebraic conditions, the space of isomorphism classes of Higgs bundles $(E,\Phi)$ satisfying the conditions of Theorem \ref{T0} is included in $\mathcal{M}_{Sp(2p,2p)}$ as a Zariski open set.



\chapter{Applications }\label{ch:applications}

In this chapter we shall describe some  applications of the results given in Chapters \ref{ch:monodromy}-\ref{ch:sppp}. In Section \ref{ap:mono} we calculate the number of connected components of the moduli space of  $SL(2,\mathbb{R})$-Higgs bundles (see \cite{N1}, \cite{goldman2} ) via the monodromy action studied in Chapter \ref{ch:monodromy}. In Section \ref{ap:up} we give a new geometric description of the connected components of the moduli space of $U(p,p)$ and $SU(p,p)$-Higgs bundles which intersect the generic fibres of the corresponding Hitchin fibration.  Finally,  in Section \ref{ap:sp} we give new results on connectivity for the moduli space of stable $Sp(2p,2p)$-Higgs bundles via the spectral data given in Chapter \ref{ch:sppp}.

\section{\texorpdfstring{The moduli space of $SL(2,\mathbb{R})$-Higgs bundles}{The moduli space of SL(2,R)-Higgs bundles}}\label{ap:mono}

Let $\mathcal{M} $ be the moduli space of $SL(2,\mathbb{C})$-Higgs bundles. From Theorem \ref{teo:split}, the intersection of the moduli space $\mathcal{M}_{SL(2,\mathbb{R})}$ of $SL(2,\mathbb{R})$-Higgs bundles with the smooth fibres of the Hitchin fibration 
\[h:\mathcal{M} \rightarrow \mathcal{A}\]
is the space of  elements of order two, giving a covering of the regular locus in $\mathcal{A} $. Hence, a natural way of obtaining information about the topology of the moduli space $\mathcal{M}_{SL(2,\mathbb{R})}$ is via the monodromy action of the covering.  In this section we calculate the number of connected components of $\mathcal{M}_{SL(2,\mathbb{R})}$ via Theorem \ref{teo}, and thus   describe the connected components of the space of surface group representations into $SL(2,\mathbb{R})$. We shall begin by studying the orbits of the monodromy action in Section \ref{sec:orb}, and then relate this analysis to the connected components of the moduli space of $SL(2,\mathbb{R})$-Higgs bundles in Section \ref{sec:con}.

 Since we consider the moduli space $\mathcal{M}_{SL(2,\mathbb{R})}$ as sitting inside $\mathcal{M}$, as mentioned in Chapter \ref{ch:real}, two real representations may be conjugate in the space of complex representations but not by real elements.  Indeed, if $A$ is in $GL(2,\mathbb{R})$ and ${\rm det} A=-1$ then $iA$ is in $SL(2,\mathbb{C})$ and conjugates $SL(2,\mathbb{R})$ to itself. Thus, flat $SL(2,\mathbb{R})$ bundles can be equivalent as $SL(2,\mathbb{C})$ bundles but not as $SL(2,\mathbb{R})$ bundles. In particular,  since $A$ changes orientation in $\mathbb{R}^2$ it will take Euler class $k$ to Euler class $-k$.
 
In the remainder of the section we  first  calculate the number of connected components of $\mathcal{M}_{SL(2,\mathbb{R})}$ via the inclusion $\mathcal{M}_{SL(2,\mathbb{R})}\subset \mathcal{M}$, and then add the number of connected components corresponding to $-k$ which could not be seen through the monodromy analysis.

\subsection{The orbits of the monodromy action}\label{sec:orb}

Recall from Theorem \ref{teo} that for $G_{0}$ the monodromy group, the representation of  $\sigma \in G_{0}$ in the basis $\tilde{\beta}$ as defined in Definition \ref{def:basis}, is given by \begin{small}
\begin{equation}
 [\sigma]=\left(\begin{array} {c|c}
I_{2g}&A\\\hline
0&\pi                 
                \end{array}
\right),
\end{equation}\end{small}
where $\pi$ is the quotient action on $\mathbb{Z}_{2}^{4g-5}/(1,\cdots,1)$ induced by the permutation representation on $\mathbb{Z}_{2}^{4g-5}$,
and $A$  is any $(2g)\times (4g-6)$ matrix with entries in $\mathbb{Z}_{2}$. We shall denote by        
 $G_{(s,x)}$ the orbit  of  $(s,x)\in P[2]\cong \mathbb{Z}_{2}^{2g}\oplus\mathbb{Z}_{2}^{4g-6} $ under the action of $G_{0}$.
For $g\in G_{0}$, the induced action  on $(s,x)$ is given by
\begin{small}
\begin{equation}
 g\cdot (s,x)^{t}=\left(\begin{array}
  {c|c}I&A\\ \hline 0&\pi
 \end{array}\right)\left(\begin{array}{c}s\\x
\end{array}\right)=\left(\begin{array}{c}s+Ax\\\pi x
\end{array}\right) ~.\label{arriba}
\end{equation}
\end{small}

\begin{proposition}\label{numero}
 The action of $G_{0}$ on $P[2]$ has $2^{2g}+g-1$ different orbits.
\end{proposition}

\begin{proof} The matrices $A$ have any possible number of $1$'s in each row,  and so for $x \neq 0$  any $s'\in \mathbb{Z}_{2}^{2g}$ may be written as  $s'=s+Ax$ for some $A$. Hence, the number of orbits $G_{(s,x)}$ for $x\neq0$ is determined by the number of orbits of the action in $\mathbb{Z}_{2}^{4g-6}$ defined by 
$\xi:x\rightarrow \pi x$. 
The map $\xi$ permutes the non-zero entries of $x$ and thus the orbits of this action are given by elements with the same number of $1$'s.

From equation (\ref{exact}) the space $\mathbb{Z}_{2}^{4g-6}$ can be thought of as  vectors in $ \mathbb{Z}_{2}^{4g-4}$  with an even number of $1$'s, modulo $(1,\dots,1)$. 
  Thus, for $x\in \mathbb{Z}_{2}^{4g-6}$ and $x\neq 0$, each orbit $\xi_{x}$ is defined by a constant $m$ such that $x$ has $2m$ non-zero entries, for $0<2m\leq4g-4$. Let us call $\bar{x}\in \mathbb{Z}^{4g-6}$  the  element  defined by the constant $\bar{m}$ such that $2\bar{m}=(4g-4)-2m$. With this notation, we  can see that $\bar{x}$ and $x$ belong to the same equivalence class in $\mathbb{Z}_{2}^{4g-6}$.  Note that for $m\neq g-1,2g-2$, the corresponding $x$ is equivalent to $\bar{x}$ under $\xi$ and thus in this case there are $g-2$ equivalent classes $\xi_{x}$. Then, considering  the equivalence class for $m=g-1$, one has $g-1$ different classes for the action of $\xi$.
From equation (\ref{arriba}), the action of $G_{0}$ on an element $(s,0)\in P[2]$ is trivial. Thus, in this case one has $2^{2g}$ different orbits of $G_{0}$.
\end{proof}

Recall that the fixed point set in $\mathcal{M}$ of the involution $\Theta$ is given by the moduli space of semistable $SL(2,\mathbb{R})$-Higgs bundles.  As an application of our results, we shall study the connected components corresponding to stable and strictly semistable $SL(2,\mathbb{R})$-Higgs bundles. In particular, we shall check  that no connected component of the moduli space $\mathcal{M}_{SL(2,\mathbb{R})}$ lies entirely over the discriminant locus of $\mathcal{A}$. 

\subsection{\texorpdfstring{Connected components of $\mathcal{M}_{SL(2,\mathbb{R})}$}{Connected components of the moduli space}} \label{sec:con}

 The bundle of $P[2]$ is a finite covering of $\A$ of degree $2^{6g-6}$. Considering two points $p$ and $q$ which are in the same orbit under the monodromy action, there is a path in $\A$ whose action connects them. The horizontal lift of the path in $\A$ is a path in $P[2]$ which connects these two points. Hence $p$ and $q$ are in the same connected component  of the fixed point submanifold of  $\Theta:(E,\Phi) \mapsto (E,-\Phi)$.

Since $SL(2,\mathbb{R})$ is isomorphic to $SU(1,1)$, a stable  $SL(2,\mathbb{C})$-Higgs bundle $(E,\Phi)$ whose isomorphism class is fixed by the involution $\Theta:(E,\Phi)\mapsto (E,-\Phi)$ is given by an $SU(1,1)$-Higgs bundle, and thus from Chapter \ref{ch:real} 
the Higgs field can be expressed as
\begin{equation}\Phi=\left(\begin{array}
              {cc}
0&\beta\\\gamma&0
             \end{array}
\right)\in H^{0}(\Sigma, {\rm End}_{0}(E)\otimes K).\label{Higgs}\end{equation}
In particular, $\Theta$ acts on $E$ via transformations of the form
\begin{equation}\pm \left(
\begin{array}{cc}
 i&0\\0&-i
\end{array}
\right).\label{sigma}\end{equation}
In the case of strictly semistable Higgs bundles, the following Proposition applies:

\begin{proposition}\label{part1}
 Any point representing a strictly semistable Higgs bundle in $\mathcal{M}$  fixed by $\Theta$ is in the connected component of a Higgs bundle with zero Higgs field.
\end{proposition}

\begin{proof} The moduli space $\mathcal{M}$ is the space of $S$-equivalence classes of semistable Higgs bundles. A strictly semistable $SL(2,\mathbb{C})$-Higgs bundle $(E,\Phi)$ is represented by  $E=V\oplus V^{*}$ for a degree zero line bundle $V$, and
\[\Phi=\left(\begin{array}
              {cc}a&0\\0&-a
             \end{array}\right)~{~\rm for~}~a\in H^{0}(\Sigma,  K).\]
 
If $V^2$ is nontrivial, 
one has an $SL(2,\mathbb{R})$-Higgs bundle only for $\Phi= 0$, which corresponds to a flat connection with  holonomy 
in the group $SO(2)\subset SL(2,R)$. If $V^2$ is trivial then the automorphism $(u,v)\mapsto (v,-u)$ takes $\Phi$ to $-\Phi$, corresponding to a flat bundle with holonomy in 
$\mathbb{R}^*\subset SL(2,\mathbb{R})$.     By scaling $\Phi$ to zero this is connected to the zero  Higgs field.
The differential $a$ can be continuously deformed to zero by considering $ta$ for $0\leq t\leq 1$.  Hence, by stability of line bundles,   one can continuously deform  $(E,\Phi)$ to a Higgs bundle with zero Higgs field via strictly semistable pairs.  
\end{proof}

In the case of stable $SL(2,\mathbb{R})$-Higgs bundles we have the following result:

\begin{proposition}\label{prop2}  Any stable $SL(2,\mathbb{R})$-Higgs bundle is in a connected component which intersects $\M$.
\end{proposition}

\begin{proof} Let $(E=V\oplus V^{*},\Phi)$ be a stable $SL(2,\mathbb{R})$-Higgs bundle with  
\begin{equation}
 \Phi=\left(\begin{array}
              {cc}0&\beta\\\gamma&0
             \end{array}\right)\in H^{0}(\Sigma,{\rm End}_{0}(E)\otimes K),\nonumber
\end{equation}
and  $d:= \deg V\geq 0$. Stability implies that   the section $\gamma\in H^{0}(\Sigma, V^{-2}K)$ is non-zero, and thus   $0\leq 2d \leq 2g-2$. Moreover, the section $\beta$ of $V^{2}K$ can be deformed continuously to zero.

 The section $\gamma$ defines a divisor $[\gamma]$ in the symmetric product $S^{2g-2-2d}\Sigma$. As this space is connected, one can continuously deform the divisor $[\gamma]$ to any $[\tilde{\gamma}]$ composed of distinct points. For   $a\in \A$ with zeros $x_{1},\ldots,x_{4g-4} $, we may deform $[\gamma]$ to $[\tilde{\gamma}]$ given by the points $x_{1},\ldots,x_{n}\in \Sigma$ for  $n:= 2g-2-2d$, and such that $\tilde{\gamma}$ is a section of $U^{-2}K$ for some line bundle $U$.

 The complementary zeros  $x_{n+1},\ldots, x_{4g-4}$ of $a$ form a divisor of $U^{2}K$. Any section $\tilde{\beta}$ with this divisor can be reached by continuously deforming $[\beta]$ from zero to the set $x_{n+1},\ldots, x_{4g-4}$. Hence,  we may continuously deform any stable Higgs bundle given by $(V\oplus V^{*},\Phi=\{\beta,\gamma\})$ to  $(U\oplus U^{*}, \tilde{\Phi}=\{\tilde{\beta}, \tilde{\gamma}\})$ in $\M.$
\end{proof}

The above analysis establishes  that   the number of connected components of the fixed point set of the involution $\Theta$ on $\mathcal{M}$ is less than or equal to the number of orbits of the monodromy group in $P[2]$.
 A flat $SL(2,\mathbb{R})$-Higgs bundle has an associated $\mathbb{R}\mathbb{P}^{1}$ bundle whose Euler class $k$ is a topological invariant which satisfies the Milnor-Wood  inequality $-(g-1)\leq k\leq g-1$.  In particular,  $SL(2,\mathbb{R})$-Higgs bundles with different Euler class lie in different connected components $\mathcal{M}^{k}_{SL(2,\mathbb{R})}$ of the fixed  point set of $\Theta$. Moreover, for $k=g-1, -(g-1)$ one has $2^{2g}$ connected components corresponding to the so-called Hitchin components. Hence, the lower bound  to the number of connected components of the fixed point set of the involution $\Theta$ for $k\geq 0$ is $2^{2g}+g-1$. As this lower bound equals the number of orbits of the monodromy group on the fixed points of $\Theta$ on $\M$, one has that the closures of these orbits can not intersect.   

Since we consider $k\geq 0$, the number of connected components of the fixed points of the involution $$\Theta:~(E,\Phi)\mapsto (E,-\Phi)$$ on  $\mathcal{M}_{reg}$ is equal to the number of orbits of the monodromy group on the points of order two of the regular fibres $\M$, i.e. $2^{2g}+g-1$. From the above analysis, one has the following application of Theorem \ref{teo}:

\begin{proposition} The number of connected components of the moduli space of semistable $SL(2,\mathbb{R})$-Higgs bundles  as $SL(2,\mathbb{C})$-Higgs bundles is $2^{2g}+g-1$. 
\end{proposition}

As mentioned previously,  the Euler classes $-k$ for $0<k< g-1$  should also considered, since this invariant labels connected components  $\mathcal{M}^{-k}_{SL(2,\mathbb{R})}$
which are mapped into $\mathcal{M}^{k}_{SL(2,\mathbb{R})}\subset \mathcal{M}.$
Hence, we have a decomposition 
\begin{eqnarray}\mathcal{M}_{ SL(2,\mathbb{R})}= \mathcal{M}^{k=0}_{ SL(2,\mathbb{R})} \sqcup \left(\bigsqcup_{i=1}^{2^{2g}}  \mathcal{M}^{k=\pm (g-1),i}_{SL(2,\mathbb{R})}    \right) \sqcup  \left( \bigsqcup_{k=1}^{g-2} \mathcal{M}^{\pm k}_{ SL(2,\mathbb{R})} \right),\nonumber \end{eqnarray}
which implies the following result:

\begin{proposition} The number of connected components of the moduli space of semistable $SL(2,\mathbb{R})$-Higgs bundles  is $
2\cdot 2^{2g}+2g-3
$. \label{coro}
\end{proposition}

The construction of the orbits of the monodromy action provides a decomposition of the  $4g-4$ zeros of $\det \Phi$ via the $2m$ non-zero entries in Proposition \ref{numero}. An element $M$ of order two in the ${\rm Prym}$ variety has the property that $\tau^{*}M\cong M$.  Considering the notation of Section \ref{sec2}, the distinguished subset of zeros correspond to the points in the spectral curve $S$ where the action on the line bundle $U$ is trivial.

\section{\texorpdfstring{The moduli space of  $U(p,p)$-Higgs bundles}{The moduli space of U(p,p)-Higgs bundles}} \label{ap:up}

As seen in Chapter \ref{ch:supp}, each stable $U(p,p)$-Higgs bundle satisfying the hypothesis of  Theorem \ref{teo11} has an associated triple $(\bar S,U_{1}, D)$, where $\bar S$ is a smooth $p$-fold cover of the Riemann surface, $U_{1}$ is a line bundle on the curve and $D$ is a positive divisor on $\Sigma$. 
By means of dimensional arguments one can show that the space of isomorphism classes of Higgs bundles satisfying the theorem   is included in the moduli space $\mathcal{M}_{U(p,p)}$ as a Zariski open set. Indeed, recall from Chapter \ref{ch:complex} that the dimension of the moduli space $\mathcal{M}_{GL(2p,\mathbb{C})}$ is
\[{\rm dim} \mathcal{M}_{GL(2p,\mathbb{C})} ={\rm dim}GL(2p,\mathbb{C})(g-1)=8p^{2}(g-1)+2.\] 
\begin{remark}
The expected dimension of the moduli space of $U(p,p)$-Higgs bundles is
\[{\rm dim}\mathcal{M}_{U(p,p)}=4p^{2}(g-1)+1.\]
\end{remark}

Via Theorem \ref{teo11} the dimension of the parameter space of triples $(\bar S,U_{1}, D)$ is calculated as follows. 

 \begin{proposition}
The dimension of the parameter space of the triples $(\bar S,U_{1}, D)$ associated
to $U(p,p)$-Higgs bundles as in Theorem \ref{teo11} is, as expected,  $4p^{2}(g-1)+1$.
\end{proposition}
\begin{proof}
 Since the curve $\bar S$ is defined by the coefficients $a_{i}\in H^{0}(\Sigma,K^{2i})$ as in Theorem \ref{teo11}, the corresponding parameter space of $\bar S$  has  dimension 
\begin{eqnarray}
 {\rm dim} \left(\bigoplus_{i=1}^{p}H^{0}(\Sigma,K^{2i})\right) 
 &=& \sum_{i=1}^{p} (4i-1)(g-1)\nonumber\\
 &=& \left( 4\frac{p(p+1)}{2}-p \right)(g-1) \nonumber\\
 &=& (2p^{2}+p)(g-1).\nonumber
 \end{eqnarray}
The choice of the divisor $D$ gives a partition of the zeros of $a_{p}$. Finally, from Theorem \ref{teo11}, the choice of $U_{1}$ is given by an element in ${\rm Jac}(\bar S)$ and thus its parameter space has dimension
\[{\rm dim}~{\rm Jac}(\bar S)=g_{\bar S}= 1+(2p^{2}-p)(g-1).\] 
Hence, the parameter space defining the pairs $(\bar S,U_{1},D)$ has dimension 
\[(2p^{2}+p)(g-1)+1+(2p^{2}-p)(g-1)=4p^{2}(g-1)+1,\]
which proves the proposition.
\end{proof}

Since the singular locus in the parameter space is given by algebraic equations, the   space of isomorphism classes of $U(p,p)$-Higgs bundles satisfying the hypothesis of Theorem \ref{teo11} is included in the moduli space $\mathcal{M}_{U(p,p)}$ as a Zariski open set.

\subsection{\texorpdfstring{Connected components of $\mathcal{M}_{U(p,p)}$}{Connected components of the moduli space}}\label{defEE}

We shall now describe how a connected component of $\mathcal{M}_{U(p,p)}$ which intersects a regular fibre of the classical Hitchin fibration looks like. For this, we shall consider the data defining the triples $(\bar S,U_{1},D)$ from Theorem \ref{teo11}.

\begin{remark}
One should note that since we consider $U(p,p)$-Higgs bundles as sitting inside the $GL(2p,\mathbb{C})$ Hitchin fibration, not all connected components are detected: a similar situation as the one described in Section \ref{sec:con} for $SU(1,1)$-Higgs bundles is encountered here. 
\end{remark}

 Recall that the spectral curve is defined by the coefficients $a_{i}\in H^{0}(\Sigma,K^{2i})$ such that $a_{p}$ has simple zeros. Hence, the choice of $\bar S$ is given by a Zariski open in the total space of 
\[H^{0}(\Sigma,K^{2p})\oplus \bigoplus_{i=1}^{p-1}H^{0}(\Sigma,K^{2i}). \]
From Remark \ref{divisorDbis}, the divisor $[a_{p}]$ of the section $a_{p}$ can be written in terms of positive divisors as
$[a_{p}]=D + \overline{D},$ where the degree $0\leq \tilde{m}< 2p(g-1)$ of $D$ gives a separation of the zeros of $a_{p}$ . 
The choice of $D$ lies in  the symmetric product $\Sigma^{[\tilde{m}]}\subset S^{\tilde{m}}\Sigma$, and together with a section $s$ of $K^{2p}[-D]$ with distinct zeros, gives the map  $a_{p}\in H^{0}(\Sigma,K^{2p})$.
Note that by Serre duality, 
\begin{eqnarray}\dim H^{1}(\Sigma, K^{2p}[-D])=\dim H^{0}(\Sigma, K^{1-2p}[D]).\label{dimensionUp}\end{eqnarray}
For $v\geq w\neq 0$, one has  $0\leq\tilde{m}< 2p(g-1)$, and thus for $p>1$ we have that
\begin{eqnarray}
 \deg (K^{1-2p}[D])&=& (1-2p)(2g-2)+\tilde{m}< 0.\nonumber     
\end{eqnarray}
By Riemann-Roch it follows that
\begin{eqnarray}
 \dim H^{0}(\Sigma, K^{2p}[-D])&=&2p(2g-2)-\tilde{m}+1-g\nonumber\\
&=& (4p-1)(g-1)-\tilde{m}.\nonumber
\end{eqnarray}
Thus, the choice of $D$ lies in  a vector bundle of rank $(4p-1)(g-1)-\tilde{m}$ over the symmetric product $S^{\tilde{m}}\Sigma$.

\begin{definition}\label{totalspaceE}
We denote by $\mathcal{E}$ the total space of a vector bundle of rank $(4p-1)(g-1)-\tilde{m}$ over the symmetric product $S^{\tilde{m}}\Sigma$ which gives the choice of $D$.
\end{definition}
Note that there is a natural map $\alpha:\mathcal{E} \rightarrow H^{0}(\Sigma,K^{2p})$, which sends a divisor in $\Sigma$ to the corresponding section $a_{p}$.

\begin{remark}
 Interchanging the roles of $V$ and $W$, which corresponds to interchanging $\sigma$ by $-\sigma$, restricts $\tilde{m}$ to be $2p(g-1)< \tilde{m}\leq 4p(g-1)$. 
\end{remark}

\begin{remark}\label{remdeg1}
In the case of $\deg V =\deg W$, the divisors $D$ and $\bar D$ have the same degree. In this case the choice of $\bar D$ also lies in $S^{\tilde{m}}\Sigma$, and in the above notation, it is given by the zeros of the section $s$.
\end{remark}

Recall that the line bundle $M$ associated to a $U(p,p)$-Higgs bundle $(V\oplus W,\Phi)$ has degree
\begin{eqnarray}m=\deg{V}+\deg{W}+(4p^{2}-2p)(g-1).\label{label4}\end{eqnarray} 
Hence, points in the moduli space $\mathcal{M}_{U(p,p)}$ of $U(p,p)$-Higgs bundle of fixed degree, correspond to the triples $(\bar S, U_{1}, D)$ with fixed invariant $m$.

As seen in Theorem \ref{teo11}, the choice of $U_{1}$ is given by a fibration of Jacobians  $\mathcal{J}$ over the space defining $\bar S$, and thus the choice of a triple $(\bar S,U_{1},D)$ is given by a point in a Zariski open subset of the total space of the fibration $\alpha^{*}\mathcal{J}$ over
\[ \mathcal{E} \oplus \bigoplus_{i=1}^{p-1}H^{0}(\Sigma,K^{2i}).\]

Therefore one has the following description of the connected components which intersect a regular fibre of the classical Hitchin fibration.

\begin{proposition}\label{something4}
For each fixed invariant $m$ as in (\ref{label4}), the invariant $0\leq \tilde{m}< 2p(g-1)$ labels exactly one connected component which intersects the non-singular fibres of the Hitchin fibration
$\mathcal{M}_{GL(2p,\mathbb{C})}\rightarrow \mathcal{A}_{GL(2p,\mathbb{C})}.$
This component is given by the fibration of $\alpha^{*}\mathcal{J}$
 over  a Zariski open subset in 
\[\mathcal{E}\oplus \bigoplus_{i=1}^{p-1}H^{0}(\Sigma,K^{2i}),\]
where $\mathcal{E}$ is as in Definition \ref{totalspaceE}. 
\end{proposition}

\subsection{\texorpdfstring{Connected components of $\mathcal{M}_{SU(p,p)}$}{Connected components of the moduli space for SU(p,p)}}

 As seen in Chapter \ref{ch:supp}, one can obtain the spectral data associated to  $SU(p,p)$-Higgs bundles by imposing restrictions on the data associated to $U(p,p)$-Higgs bundles. In particular, recall that in this case  $m=(4p^{2}-2p)(g-1)$ is fixed.

 As in the case of  $U(p,p)$-Higgs bundles, the choice of $D$ is given by a point in the symmetric product of  $\Sigma$, and since it has simple zeros, a point in  $\Sigma^{[\tilde{m}]}\subset S^{\tilde{m}}\Sigma$. 
The divisor of the section $a_{p}$ is separated by the divisor $D$ together with a section of $K^{2p}[-D]$ (giving the divisor $\bar D$). The maximal Toledo invariant corresponds to $\tilde{m}=0$,  in which case from (\ref{toledoforUpp}) one has $$\tau(v,w)=w-v= 2p(g-1).$$ Hence, for $\tilde{m}=0$ the divisor $D$ 
 is empty and thus from Proposition \ref{teo22} one has
\begin{eqnarray}2{\rm Nm}U_{1}= p(2p-1)K.\label{aboveeq}
\end{eqnarray}
Hence, following the analysis done in Section \ref{defEE}, in this case the choice of $\bar S$ and $D$ is given by a point in a Zariski open subset $\overline{ \mathcal{A}}$ of
\[\bigoplus_{i=1}^{p}H^{0}(\Sigma, K^{2i}). \]
From (\ref{aboveeq}), one can see that for each choice of square root of $p(2p-1)K$, 
the line bundle $U_{1}$  is determined by an element in the Prym variety ${\rm Prym}(\bar S, \Sigma)$.  Therefore, for each of the $2^{2g}$ choices one has a copy of  ${\rm Prym}(\bar S,\Sigma)$, giving a fibration over $\overline{\mathcal{A}}$ whose fibres are the disjoint union $2^{2g}$ Prym varieties. Note that if $p$ is even, then there is a distinguished choice of square root.

\begin{remark}
In the case of $\tilde{m}\neq 0$, further study of the parameter space defining the triples $(\bar S, U_{1}, D)$ needs to be done in order to obtain a description of the connected components of the moduli space $\mathcal{M}_{SU(p,p)}$. 
\end{remark}

\section{\texorpdfstring{The moduli space of  $Sp(2p,2p)$-Higgs bundles}{The moduli space of Sp(2p,2p)-Higgs bundles}} \label{ap:sp}
 
As seen in Theorem \ref{T0},   each stable  $Sp(2p,2p)$-Higgs bundle $ (E=V\oplus W, \Phi)$ on a compact Riemann surface $\Sigma$ of genus $g\geq 2$ for which the square root of the characteristic polynomial defines a smooth curve, has an associated pair $(S,M)$. Recall that
 $S$ is a smooth $2p$-fold cover of $\Sigma$, and $M$ is a rank 2 vector bundle on $S$ which is preserved by the natural involution $\sigma$ on $S$, and is acted on by $\sigma$ with different eigenvalues over the fixed points. 

Via the work of \cite{AG}, in Section \ref{sec:par} we relate the spectral data $(S,M)$ to certain parabolic vector bundles on the quotient curve $\bar S=S/\sigma$, and use this description to study connectivity of $\mathcal{M}_{Sp(2p,2p)}$ in Section \ref{sec:consp}.

\subsection{\texorpdfstring{Parabolic vector bundles on $\bar S$}{Parabolic vector bundles}}\label{sec:par}

From Chapter \ref{ch:sppp}, the spectral data of $Sp(2p,2p)$-Higgs bundles is given by $\mathcal{N}^{\sigma}$, the moduli space of stable rank 2 vector bundles on $S$ preserved by the  involution $\sigma$.  Following the notation of \cite{AG}, we shall write $\mathcal{N}^{\sigma}=(\mathcal{M}_{L}^{s})^{G}$, where $L:=\rho^{*}K^{2p-1}$ is the fixed determinant of the rank 2 vector bundles and $G$ is the group generated by $\sigma$. With this notation, $\mathcal{M}_{L}^{G}$ denotes the moduli space of rank 2 vector bundles on $S$ which have fixed determinant $L$ and are preserved by the action of $\sigma$.

\begin{remark}
 From \cite[Section 1]{AG}, any fixed point in $\mathcal{M}_{L}^{G}$ can be represented by a semi-stable bundle $E$ on $S$ with an action of the group $G$ covering the action on $S$. The $G$-equivariant structure is equivalent to a bundle homomorphism $\overline{\sigma}:E\rightarrow E$ covering $\sigma$ such that $\overline{\sigma}^{2}={\rm Id}_{E}$, and thus there is a unique (up to sign) such structure.  
\end{remark}

\begin{definition}\label{invotau}\label{inv:tau}
 We denote by $\tau$ the involution
$\tau: (E,\overline{\sigma})\mapsto(E,-\overline{\sigma}), $
which interchanges the two possible equivariant structures at a point in $\mathcal{M}_{L}^{G}$ .
\end{definition}

From \cite[Definition 2.19]{AG}, we denote by $\mathcal{P}_{a}$ the moduli space of admissible  parabolic bundles on $\bar S$, which are parabolic rank 2 vector bundles whose marked points are the fixed points of the involution $\sigma$, whose weights are $1/2$, and whose flag is defined as in \cite[Section 2]{AG} by the distinguished eigenspaces corresponding to the eigenvalue $-1$ of $\sigma$.
From \cite[Definition 3.3]{AG} there is a natural map
$\Psi:\mathcal{P}_{a}\rightarrow  \mathcal{M}_{L}^{G},$
and the space $\mathcal{P}_{a}$ may be decomposed as the disjoint union
$$\mathcal{P}_{a}=\mathcal{P}^{ns}_{a}\sqcup\mathcal{P}^{s}_{a},$$
where $\mathcal{P}^{s}_{a}$ is the set of points in $\mathcal{P}_{a}$ represented by stable parabolic bundles.

\begin{remark}
 By \cite[Section 3, p. 76]{nitin}, under our conditions any admissible semi-stable bundle is automatically stable, and thus $\mathcal{P}^{s}_{a}=\mathcal{P}_{a}$.
\end{remark}

The involution $\tau$ defined in (\ref{inv:tau}) induces an action of $\mathbb{Z}_{2}$ on $\mathcal{P}^{s}_{a}$ (see \cite[Section 3, p.17]{AG}), which we shall denote by $\tilde{\tau}$. Thus, we may decompose $\mathcal{P}^{s}_{a}$ via the action of $\tau$ as
\[\mathcal{P}^{s}_{a}=\mathcal{P}^{s,i}_{a}\sqcup \mathcal{P}^{s,g}_{a},\]
where $\mathcal{P}^{s,i}_{a}$ is the fixed point set in $\mathcal{P}^{s}_{a}$ under the action of $\tilde{\tau}$.

\begin{proposition}{\cite[Theorem 3.4]{AG}}  The space $\mathcal{N}^{\sigma}=(\mathcal{M}_{L}^{s})^{G}$ is in a bijection
\begin{eqnarray}
\Psi|_{  \mathcal{P}_{a}^{s,g}}~:~\mathcal{P}_{a}^{s,g}\rightarrow  (\mathcal{M}_{L}^{s})^{G}.
\end{eqnarray}
 \end{proposition}
In particular, one should note that $\mathcal{P}_{a}^{s,g}$ is the complement of the fixed point set of an involution in the moduli space of admissible  parabolic bundles $\mathcal{P}_{a}^{s}$, and thus is a Zariski open set in $\mathcal{P}_{a}^{s}$.   
 
\subsection{\texorpdfstring{Connectivity of $\mathcal{M}_{Sp(2p,2p)}$}{Connectivity for Sp(2p,2p)}}\label{sec:consp}

As seen previously, the spectral curve of an $\mathcal{M}_{Sp(2p,2p)}$-Higgs bundle is always singular, and thus we shall say that a point $a\in \mathcal{A}_{Sp(4p,\mathbb{C})}$ is a regular point in the $Sp(4p,\mathbb{C})$ Hitchin base, in terms of $\mathcal{M}_{Sp(2p,2p)}$-Higgs bundles,  if it defines a smooth curve $S$. Hence, from the previous section we have the following:

\begin{proposition}\label{something3}
 The intersection of the moduli space $\mathcal{M}_{Sp(2p,2p)}$ with the regular fibres of the  $Sp(4p,\mathbb{C})$ Hitchin fibration is given by the space $\mathcal{P}_{a}^{s,g}$.
\end{proposition}

Moreover, in \cite[Proposition 3.3]{nitin} the Poincar\'e polynomial of the moduli space $\mathcal{P}_{a}^{s}$ is described, from where one has that the moduli space $\mathcal{P}_{a}^{s}$ is connected, and thus so is $\mathcal{P}_{a}^{s,g}$. The choice of the curve $S$ in the spectral data $(S,M)$ is given by a point in the complement of the singular locus in $\mathcal{A}_{Sp(4p,\mathbb{C})}$, and therefore by a point in a Zariski open set in a connected space. Hence, $\mathcal{M}_{Sp(2p,2p)}$ can be described as follows:

\begin{proposition}
 \label{something7}
The moduli space $\mathcal{M}_{Sp(2p,2p)}$ is connected, and is given by the fibration of a Zariski open set in $\mathcal{P}_{a}^{s}$, over a Zariski open set in the space 
 $$\bigoplus_{i=1}^{p} H^{0}(\Sigma, K^{2i}) .$$
\end{proposition}



\chapter{Further questions }\label{ch:further}

In this thesis we have developed new techniques for studying principal $G$-Higgs bundles over a compact Riemann surface via the spectral data associated to each Higgs pair. This new approach  has enlightened paths of further investigation which we shall describe in this section. 

 We shall begin by stating a conjecture about connectivity of $\mathcal{M}_{G}$ in relation with the results presented in this thesis in Section \ref{sec:conn}. In Section \ref{sec:hig} we discuss   implications of our results related to higher cohomology groups of the moduli spaces $\mathcal{M}_{G}$ and finally in Section \ref{sec:other} we consider certain groups $G$ for which one could study $\mathcal{M}_{G}$ via the spectral data method.

\section{\texorpdfstring{Connectivity of $\mathcal{M}_{G}$}{Connectivity of the moduli spaces}}\label{sec:conn}

As mentioned before, during the last decade it has been of interest to calculate the connected components of the moduli spaces of principal G-Higgs bundles and their relation with surface group representations.  
The state of the art as of the time of writing this thesis is given in the following table \cite{brad2}:

\begin{small}
\begin{center}
\begin{tabular}{|c|c|c|}
\hline G & Constraints & Reference \\ \hline
  $PSL(n, \mathbb{R})$  &         -    & Hitchin           \\ \hline
$SU(p, q)$, $PU(p, q)$   &    $d = 0, d_{max}$         &    Bradlow, Gothen, Garcia-Prada, Markman,  Xia        \\ \hline
 $Sp(4, \mathbb{R})$   &      -       &        Gothen, Garcia-Prada, Mundet    \\ \hline
 $Sp(2n, \mathbb{R})$, $n > 2$  &    $d = 0, d_{max}$          &       Gothen, Garcia-Prada, Mundet
    \\ \hline
 $SO^{∗} (2n)$   & $ d_{max}$             &  Bradlow, Gothen, Garcia-Prada         \\ \hline
 $SO_{0} (p, q)$  &        $p = 1, q ~odd$      &       Aparicio     \\ \hline
 $U^{*}(2n)$   &         -    &         Garcia-Prada, Oliveira  \\ \hline
 $GL(n, \mathbb{R})$   &  -            &  Bradlow, Gothen, Garcia-Prada
        \\ \hline
\end{tabular}
\end{center}
\end{small}

 \bigskip
 \bigskip

The results presented in Chapter \ref{ch:applications} already contribute to the above table, and we conjecture the following:

\begin{conjecture}\label{con1}
 For $G=U(p,p)$, $SU(p,p)$ and $Sp(2p,2p)$, any component of the fixed point set of the involution $\Theta_{G}$ intersects the so-called generic fibre of the Hitchin fibration $\mathcal{M}_{G^{c}}\rightarrow \mathcal{A}_{G^{c}}$. 
\end{conjecture}

Validity of the conjecture would imply, in particular,  that the connectivity results given in Chapter \ref{ch:applications} do in fact realise all components of the moduli spaces $\mathcal{M}_{G}$ for $G=U(p,p)$, $SU(p,p)$ and $Sp(2p,2p)$.

\section{\texorpdfstring{Cohomology groups of $\mathcal{M}_{G}$}{Cohomology groups of the moduli spaces}}\label{sec:hig}

As seen in Chapter \ref{ch:complex}, the moduli space of classical Higgs bundles is a Hyper-K\"ahler manifold. The K\"ahler form for the Higgs bundle complex structure restricts in the flat connection complex structure to the canonical symplectic form $\omega^{(B)}$ defined by Goldman \cite{goldman}. 

If one can find compact complex subvarieties in the Higgs bundle complex structure then a suitable power of the symplectic form must be nonzero, providing information about the corresponding space of representations. In particular, a Jacobian for $U(p,p)$-Higgs bundles, as described in Chapter \ref{ch:supp}, or the  moduli space of rank 2 vector bundles for $Sp(2p,2p)$-Higgs bundles, as described in Chapter \ref{ch:sppp}, will give a lower bound for the non-vanishing of cohomology groups.

\section{Spectral data for other real forms}\label{sec:other}

It would be very interesting to continue the research by extending the methods developed in this thesis to other real forms $G$ which appear to carry spectral data closely related to the one studied in this thesis. In particular, the following groups present interesting scenarios:

\begin{itemize}

\item $SU(p+1,1)$-\textit{Higgs bundles}.  From the description of  $SU(p,q)$-Higgs bundles given in Chapter \ref{ch:real}, the spectral curve of an $SU(p+1,1)$-Higgs bundle $(E,\Phi)$ lives in the total space $X$ of $K$ with projection $\rho:X\rightarrow \Sigma$, and has equation
$\eta^{p}( \eta^{2}+a_{1})=0$ 
for $\eta$ the tautological section of $\rho^{*}K$ and $a_{1}\in H^{0}(\Sigma,K^{2})$.
Following the ideas of Chapter \ref{ch:sppp}, we denote by $S$ the curve given by a component of ${\rm det}(\rho^{*}\Phi-\eta)$, whose equation is
$\eta^{2}+a_{1}=0.$

One could try to follow the methods in \cite{Xia2} to understand $SU(p+1,1)$-Higgs bundles via $U(1,1)$-Higgs bundles $(E_{1}:=V/V_{\gamma},\Phi_{1}=(\beta_{1},\gamma_{1}))$ with additional data obtained from the kernel $V_{\gamma}$ of $\gamma$. For this, one may need to consider the set of vector bundles 
\[\mathcal{H}^{p}_{d}:=\{  ~ F~:~{\rm rk}F=p~,~{\rm deg}F=d<0~,~ {\rm deg}F' \leq 0 ~{~\rm for ~all~}~ F'\subset F\}.\] 

Note that, by stability, $V_{\gamma}\in \mathcal{H}^{p}_{d} $.
In the case of $p=2$, and ${\rm deg }V_{\gamma}=-1$, one can see that starting with a $U(1,1)$-Higgs bundle, the map $\beta_{1}$ can generically be lifted to $\beta$ as given in the following diagram:
\begin{eqnarray}\xymatrix{
0 \ar[r]&V_{\gamma}\ar[r]&V\ar[r]^{\pi}&V/V_{\gamma} \ar[r]
&0, \\
&&L\otimes K^{*}\ar  @{.>} [u]^{\beta} \ar[ur]^{\beta_{1}}&
}\end{eqnarray}
The complete analysis of the case $p=2$ should follow through the study of the corresponding extensions and obstructions to the lifting of $\beta_{1}$.

\item $SU(p+1,p)$-\textit{Higgs bundles}.  The spectral curve associated to a stable $SU(p+1,p)$-Higgs bundle $(E, \Phi)$ has equation 
 \[ \eta( \eta^{2p}+\eta^{2p-2}a_{1}+\eta^{2p-4}a_{2}+\ldots+\eta^{2}a_{p-1}+a_{p})=0,\]
and thus one may follow the approach done for $SO(2m+1,\mathbb{C})$ \cite{N3}, which is described in Chapter \ref{ch:complex}, and  consider the associated curve $\rho:S\rightarrow \Sigma$ given by
\[ \eta^{2p}+\eta^{2p-2}a_{1}+\eta^{2p-4}a_{2}+\ldots+\eta^{2}a_{p-1}+a_{p}=0.\]

There is a natural  involution $\sigma(\eta)=-\eta$ on $S$, and hence one can obtain  the quotient curve $\overline{\rho}:\overline{S}\rightarrow \Sigma$ in the total space of $K^{2}$ whose equation is
\begin{eqnarray}\xi^{p}+a_{1}\xi^{p-1}+a_{2}\xi^{p-2}+\ldots +a_{p}=0,\label{curveS2}\end{eqnarray}
where $\xi=\eta^{2}$. Let $L_{\gamma}:={\rm ker} \gamma\subset V$ and $ V_{\beta}:={\rm im \beta}\otimes K^{*}\subset V.$
 By restricting the maps $\beta$ and $\gamma$, one has an induced $U(p,p)$-Higgs pair $( E_{1},\Phi_{1})$ where the vector bundle is $E_{1}=V/L_{\gamma}\oplus W$ and
\begin{eqnarray}
&\beta_{1}&:W\rightarrow V/L_{\gamma}\otimes K,\nonumber \\ 
&\gamma_{1}&:V/L_{\gamma}\rightarrow W\otimes K.\nonumber 
\end{eqnarray}

By stability ${\rm deg} L_{\gamma}<0$ and the pair $( E_{1},\Phi_{1})$ is stable. Recall that in the case of $SO(2m+1,\mathbb{C})$-Higgs bundles, one could recover the $SO(2m+1,\mathbb{C})$-Higgs bundle from a rank $2m$ symplectic Higgs bundle together with a 1-dimensional eigenspace. 

By means of the spectral data defined in Chapter \ref{ch:supp}, one could expect to obtain information of the $SU(p+1,p)$-Higgs bundle by considering a $U(p,p)$-Higgs bundle, whose data is understood, and extending it to a $SU(p+1,p)$-Higgs bundle by taking an appropriate extension 
\begin{eqnarray}\xymatrix{0\ar[r]& V_{\gamma}\ar[r]& V\ar[r]& V/V_{\gamma}\ar[r]& 0,\\
&&W\otimes K^{*}\ar @{.>} [u]^{\beta}\ar[ur]^{\beta_{1}}&&}
\end{eqnarray}
and studying the obstruction for $\beta$ to be a lift of $\beta_{1}$.\\

\item $SO(p+2,p)$-\textit{Higgs bundles}.  Consider $(E, \Phi)$ an $SO(p+2,p)$-Higgs bundle as defined in Chapter \ref{ch:real}, and let $V_{\gamma}$ and $V_{\beta}$ be  the rank 2 vector bundles defined as $$V_{\gamma}:={\rm ker} \gamma\subset V ~{~\rm~ and~}~ V_{\beta}:={\rm im \beta}\otimes K^{*}\subset V.$$ 

Although $SO(q,p)$-Higgs bundles were studied in \cite{ap}, not much has been said about the particular case of $q=p+2$. By restricting the maps $\beta$ and $\gamma$, one has two induced $U(p,p)$-Higgs pairs $(E_{i},\Phi_{i})$, for $i=1,2$, where $E_{1}:=V/V_{\gamma}\oplus W$ and $E_{2}:=V_{\beta}\oplus W$. 

 It is expected that  considering the extensions of these particular $U(p,p)$-Higgs bundles $(E_{i},\Phi_{i})$ by the rank 2 vector bundle $V_{\gamma}$, one could recover $SO(p+2,p)$-Higgs bundles. Although stability of an $SO(p+2,p)$-Higgs bundle does not imply directly stability of the vector bundle  $V_{\gamma}$, additional conditions could be set in order for the implication to hold. In particular, understanding the low rank cases would provide useful understanding of the corresponding space of surface group representations.
 \end{itemize}



 
\clearpage
\phantomsection
\addcontentsline{toc}{chapter}{Bibliography} 
\fancyhead[CO,CE]{\fancyplain{}{\textit{Bibliography}}} 

\end{document}